\documentclass[12pt]{amsart}
\usepackage{amsmath,amssymb,latexsym, amsthm, amscd, mathrsfs, stmaryrd}
\usepackage[linktocpage=true]{hyperref}
\usepackage{color} 
\usepackage[all]{xy}

\newtheorem{thm}{Theorem} [section]
\theoremstyle{definition}
\newtheorem{rem}[thm]{Remark}
\theoremstyle{plain}
\newtheorem{prop}[thm]{Proposition}
\newtheorem{lem}[thm]{Lemma}
\newtheorem{cor}[thm]{Corollary}

\numberwithin{equation}{section}

\newcommand{\mbf}{\mathbf}
\newcommand{\mbb}{\mathbb}
\newcommand{\mrm}{\mathrm}
\newcommand{\A}{\mathcal A}
\newcommand{\Axx}{\A_{O(D)}(X\times X)}
\newcommand{\Axy}{\A_{O(D)}(X\times Y)}
\newcommand{\Axxc}{\A_{Sp(D)}(X_{\mbf C_d} \times X_{\mbf C_d})}
\newcommand{\Axyc}{\A_{Sp(D)}(X_{\mbf C_d} \times Y_{\mbf C_d})}
\newcommand{\Axxi}{\A_{O(D)}({}^\imath X\times {}^\imath X)}
\newcommand{\Axyi}{\A_{O(D)}({}^\imath X\times Y)}
\newcommand{\Ayy}{\A_{O(D)}(Y\times Y)}
\newcommand{\Ayyc}{\A_{Sp(D)}(Y_{\mbf C_d}\times Y_{\mbf C_d})}

\newcommand{\Bi}{\mbf B^{\imath}}
\newcommand{\Bj}{\mathbf{B}^{\jmath}}

\newcommand{\E}{\mathbf e}
\newcommand{\F}{\mbf  f}

\newcommand{\gl}{\mathfrak{gl}}
\newcommand{\Hb}{\mathbf{H}_{B_d}}
\newcommand{\IC}{\mrm{IC}}
\newcommand{\J}{\mathcal J}
\newcommand{\K}{\mbf K}
\newcommand{\Ki}{\mathbf{K}^{\imath}}   
\newcommand{\Kj}{\mathbf{K}^{\jmath}}   
\newcommand{\la}{\lambda}

\newcommand{\Ob}{\mathcal O}
\newcommand{\Q}{\mathbb Q}
\newcommand{\Qq}{\mathbb{Q}(v)}
\newcommand{\QQ}{{}_{\mathbb Q}}
\newcommand{\Sj}{\mathbf{S}^{\jmath}}
\newcommand{\Si}{\mathbf{S}^{\imath}}
\newcommand{\Sjc}{{}^{\mbf C}\mathbf{S}^{\jmath}}
\newcommand{\Sic}{{}^{\mbf C}\mathbf{S}^{\imath}}

\newcommand{\tM}{\texttt{M}}
\newcommand{\txi}{\tilde{\Xi}}
\newcommand{\txid}{\tilde{\Xi}^{\text{diag}}}
\newcommand{\txii}{{}^{\imath}\tilde{\Xi}}
\newcommand{\T}{\mbf t}
\newcommand{\Td}{{\mathbf T}_d}
\newcommand{\Tdc}{{}^{\mbf C} {\mathbf T}_d}
\newcommand{\Tdi}{ {\mathbf T}_d^\imath }
\newcommand{\U}{\mbf U}

\newcommand{\Udot}{\dot{\U}}
\newcommand{\Ui}{\mathbf{U}^{\imath}}
\newcommand{\Uidot}{\dot{\U}^{\imath} }
\newcommand{\Uj}{\mathbf{U}^{\jmath}}
\newcommand{\Ujdot}{\dot{\U}^{\jmath} }
\newcommand{\VV}{\mathbb{V}}
\newcommand{\VVi}{{}^\imath\mathbb{V}}

\newcommand{\Z}{\mathbb Z}

\newcommand{\bt}{\mbf t}

\newcommand{\ro}{\mrm{ro}}
\newcommand{\co}{\mrm{co}}

\newcommand{\End}{\mrm{End}}

\newcommand{\ibe}[1]{e_{{#1}}}
\newcommand{\ibff}[1]{f_{{#1}}}
\newcommand{\ibd}[1]{d_{{#1}}}

\title[Geometric Schur  Duality of Classic  Type]{Geometric Schur  Duality of Classical Type}

 \author[Huanchen Bao]{Huanchen Bao}
\address{Department of Mathematics\\ University of Virginia\\ Charlottesville, VA 22904. 
{\em Current address:}  Institute for Advanced Study,  1 Einstein Drive, Princeton, New Jersey 08540}
\email{huanchen@math.umd.edu}

\author[Kujawa]{Jonathan Kujawa}
\address{Department of Mathematics\\ University of Oklahoma\\ Norman, OK 73019}
\email{kujawa@math.ou.edu}

\author[Li]{Yiqiang Li}
\address{Department of Mathematics\\ University at Buffalo, SUNY  \\Buffalo, NY 14260}
\email{yiqiang@buffalo.edu}

\author[Wang]{Weiqiang Wang}
\address{ Department of Mathematics\\ University of Virginia\\ Charlottesville, VA 22904}
\email{ww9c@virginia.edu}


\begin{document}

\begin{abstract}
This is a generalization of the classic work of Beilinson, Lusztig and MacPherson. 
In this paper (and an Appendix) we show that
the quantum algebras obtained via a BLM-type stabilization procedure    
in the setting of partial flag varieties of type $B/C$
are two  (modified) coideal subalgebras of the quantum general linear Lie algebra, $\dot{\mbf U}^{\jmath}$ and $\dot{\mbf U}^{\imath}$. 
We provide a geometric realization of the Schur-type duality of Bao-Wang between such a coideal algebra and Iwahori-Hecke algebra of type $B$.
The monomial bases and canonical bases of the Schur algebras and the modified coideal algebra $\dot{\mbf U}^{\jmath}$ 
are constructed.

 In an Appendix by three authors,  a more subtle $2$-step stabilization procedure leading to $\dot{\mbf U}^{\imath}$ is developed,
 and then monomial and canonical bases of $\dot{\mbf U}^{\imath}$ are constructed.
 It is shown that $\dot{\mbf U}^{\imath}$ is a subquotient of $\dot{\mbf U}^{\jmath}$ with compatible canonical bases. Moreover,
 a compatibility between canonical bases for modified coideal algebras
and Schur algebras is established.  
 \end{abstract}

\maketitle

\newcounter{introsections} \setcounter{introsections}{1}

\newcommand{\introsection}{
\vspace{1em}

\noindent 
{ 1.\arabic{introsections}.\addtocounter{introsections}{1} } }

\setcounter{tocdepth}{1}
\tableofcontents

\section{Introduction}

\introsection{}
In their 1990 paper \cite{BLM}, Beilinson, Lusztig and MacPherson (BLM) provided a geometric construction
of the quantum Schur algebra and more importantly the (modified) quantum 
groups $\U(\gl(N))$ and  $\dot{\U}(\gl(N))$, using the partial flag varieties of type $A$.
They further constructed a canonical basis for this modified quantum  group (see
Lusztig \cite{Lu93} for  generalization to all types). The BLM construction
has played a fundamental role in categorification; see Khovanov-Lauda \cite{KhL10}.
The BLM construction is adapted subsequently by Grojnowski and Lusztig \cite{GL92} to realize
geometrically the Schur-Jimbo duality \cite{Jim}.
These works raised an immediate question which remains open until now:  

\vspace{.1cm}
(Q1). What are the quantum algebras
arising from flag varieties of classical type?  
\vspace{.1cm}

At the beginning, one was even tempted to hope that the corresponding Drinfeld-Jimbo quantum groups \cite{Dr86} would provide the answer.
However, the expectation for an answer being the quantum groups of classical type was somewhat diminished
after Nakajima's quiver variety construction \cite{Na94} which has provided
a powerful geometric realization of integrable modules of (quantum) Kac-Moody algebras of symmetric type.  
Nevertheless, there has been a successful generalization  to the affine type $A$ in \cite{GV93, L99, Mc12, VV99, SV00}
via affine flag varieties of type $A$. 

In a seemingly unrelated direction two of the authors \cite{BW}  recently developed a new approach to Kazhdan-Lusztig theory 
for the BGG category $\mathcal O$ of classical type
by initiating a new theory of {\em $\imath$canonical basis} arising from quantum symmetric pairs,
and used it to solve the irreducible character problem for the ortho-symplectic Lie superalgebras. 
At a decategorification level, a duality was established in {\em loc. cit.} between a quantum algebra (denoted by $\Ui$ or $\Uj$) 
and the Iwahori-Hecke algebra $\Hb$ of type $B$ acting on a tensor space, generalizing the Schur-Jimbo duality (we shall refer to this
new duality as $\imath$Schur duality, where $\imath$ partly stands for ``involution"). The $\imath$Schur duality can also
be formulated in terms of Schur algebras $\Si$ or $\Sj$, and indeed an algebraic version of the $(\Si, \Hb)$-duality
already appeared in \cite{G97}. 

The aforementioned quantum algebras $\Uj$ and $\U^\imath$ are the so-called coideal subalgebras of the quantum group $\U(\gl(N))$ 
and they form quantum symmetric pairs $(\U(\gl(N)), \Ui)$ and $(\U(\gl(N)), \Uj)$, depending on whether $N$ is odd or even.
A general theory of quantum symmetric pairs was developed by Letzter \cite{Le02}. 
The categorical significance of $\imath$Schur duality  \cite{BW} raises the following natural question: 

\vspace{.1cm}
(Q2). Is there a  geometric realization of $\imath$Schur duality and $\imath$canonical basis?

\introsection{}
The goal of this paper and an Appendix
is to settle the two questions (Q1) and (Q2) for type $B$/$C$ completely by showing they provide
answers to each other. 

The coideal algebras admit modified (i.e., idempotented) versions $\Ujdot$ and $\Uidot$, following the by-now-standard 
algebraic construction. 
In the paper we establish multiplication formulas, generating sets as well as canonical bases for the Schur algebras $\Sj$ and $\Si$
via a geometric approach. 
We show that $\Ujdot$ 
(instead of  the type $B/C$ quantum group) 
is the quantum algebra \`a la BLM stabilization behind the family of Schur algebras $\Sj$ using the geometry of
partial flag varieties of type $B/C$. 
As applications, the monomial and $\imath$canonical bases of the  algebra $\Ujdot$ are constructed for the first time,
and the $\imath$Schur dualities  
are realized geometrically \`a la Grojnowski-Lusztig.

The Appendix written by three of the authors provides the more subtle $2$-step stabilization procedure
which leads to the remaining algebra $\Uidot$ and its canonical basis. 
Precise relations between $\Uidot$ and $\Ujdot$ are established. It is also shown that there exist surjective
homomorphisms from the modified coideal algebras to the Schur algebras (of both types), which send canonical bases
to canonical bases or zero.

The (partial) flag varieties and (parabolic) BGG category $\mathcal O$ are well known to be intimately related
(cf. \cite{KL79, KL80, GL92, BGS}). 
Our construction in this paper (and also \cite{BW}) suggests some more connection between flag varieties and category $\mathcal O$
for classical type $B/C$,  and it will shed new light 
in further understanding of these fundamental objects in representation theory. 
While the Chevalley generators of $\Uj$ and $\Ui$ afford interpretations of translation functors in category $\mathcal O$
~\cite{ES13, BW}, they are now realized via ``Hecke correspondences" of partial flag varieties.

\introsection{}
Let us describe our constructions and results in detail. 

We provide in Section~\ref{sec:convolution}
a geometric convolution construction of a Schur $\A$-algebra $\Sj$ on pairs of $N$-step partial flags of type $B_d$
as a convolution algebra, where $\A =\Z[v, v^{-1}]$. This  algebra and the  Iwahori-Hecke algebra of type $B_d$
(via  a similar convolution construction on pairs of complete flags) satisfy a double centralizer property acting on an $\A$-module $\Td$, which
is also defined geometrically. We compute the geometric action of the Iwahori-Hecke algebra on $\Td$ explicitly. 

In Section~\ref{sec:Schur}, 
various basic multiplication formulas
for $\Sj$,  $\Hb$, and their commuting actions on $\Td$ are worked out precisely (whose type $A$ counterparts can be found in \cite{BLM, GL92}).
We establish a generating set for $\Sj$ and several explicit relations satisfied by these generators. 
From geometry, we construct a standard basis, a monomial basis, and a canonical basis of $\Sj$, respectively. 
The signed canonical basis of $\Sj$ is shown to be characterized by an almost orthonormality similar to the more familiar
Lusztig-Kashiwara canonical basis  \cite{Lu90, Ka91, Lu93} arising from the Drinfeld-Jimbo quantum groups. 

In Section~\ref{sec:qalg}, 
we establish a remarkable stabilization property for $\Sj$ as $d$ goes to infinity in a suitable sense, following 
the original approach of \cite{BLM} (see \cite{DDPW} for an exposition). This
stabilization allows us to construct an $\A$-algebra $\Kj$. The algebra $\Kj$
is again naturally equipped with a standard basis, a monomial basis, as well as a canonical basis. 
We give a presentation of the modified coideal
algebra $\Ujdot$, and establish an isomorphism of $\Qq$-algebras $\Ujdot \cong {\Qq} \otimes_\A \Kj$. 
We then establish a surjective $\A$-algebra homomorphism $\phi_d$ from $\Kj$ to $\Sj$. 
We then identify the $\Uj$-module $\Td$ as a quantum $d$-th tensor space; in this way, we obtain
a geometric realization of the $\imath$Schur duality in \cite{BW} between $\Uj$ and $\Hb$ acting on $\Td$. 

So far we have been focused on the case of $\Sj$ and $\Ujdot$ in detail. 
One would like to have geometric realizations of  the modified coideal algebra $\Uidot$ and
related Schur algebra $\Si$  as well. 
As we shall see in Section~\ref{sec:typeC}, the Schur  algebra $\Si$ arises most naturally using the type $C$ flag varieties 
by constructions similar to Section~\ref{sec:qalg}. But in this way it is difficult  if not impossible to see any direct connection
between the two types of Schur algebras $\Si$ and $\Sj$ (not to mention any further connection between
$\Uidot$ and $\Ujdot$), since we cannot compare the geometries
of type $B$ and $C$ directly. 

In  Section~\ref{sec:2B} we provide geometric realizations, via a refined construction of type $B_d$ partial flag varieties, of 
the Schur algebra $\Si$, its
canonical basis, as well as the $\imath$Schur $(\Si, \Hb)$-duality.
Among others, the geometric meaning of a distinguished generator $\bt$ in $\Si$ is made precise
as sums of   simple perverse sheaves, up to a shift. 
We further show that the standard bases, the monomial bases, as well as the 
canonical bases of $\Si$ and $\Sj$ are all compatible under an algebra embedding $\Si \subset \Sj$.

The constructions of the previous sections on Schur algebras
are further adapted in the setting of flag varieties of type $C$  in Section~\ref{sec:typeC}.
It is interesting to note a Langlands type duality phenomenon, 
namely, $\Si$ arises most naturally in the type $C$ setting as done in Section~\ref{sec:qalg} (for type $B$ where $\Sj$ arises most naturally),
and then $\Sj$ appears also in a refined type $C$ construction.  

In Appendix~ A written by three of the authors, a more subtle stabilization procedure is developed to
construct an $\A$-algebra $\Ki$ from the family of algebras $\Si$. The reason is that the divided power elements
of the distinguished generator $\bt$ are poorly understood (see \cite[Conjecture~4.13]{BW}) and there is no explicit
multiplication formula in $\Si$ for these elements. Our construction of $\Ki$ is a two-step process, and we 
show that $\Ki$ is isomorphic to a subquotient of $\Kj$. 
It is  further shown that $\Uidot \cong  {\Qq} \otimes_\A \Ki$, 
and this leads to an integral $\A$-form ${}_\A \Uidot$. 
We then establish a remarkable property that the standard bases, the monomial bases, as well as the 
canonical bases are all compatible between $\Kj$ and $\Ki$ (equivalently, between ${}_\A \Ujdot$ and ${}_\A \Uidot$) under the subquotient construction.
We also construct a surjective homomorphism $\phi_d^\imath: \Ki \rightarrow \Si$, and hence obtain
a geometric realization of the $(\Uidot, \Hb)$-duality.

We further show in Appendix~A that the surjective homomorphism $\phi_d: \Kj \rightarrow \Sj$ in Section~\ref{sec:qalg}
maps an arbitrary canonical (respectively, standard) basis element of $\Kj$ to a canonical (respectively, standard) basis element 
in $\Sj$ or zero; see Theorem~\ref{th:CBtoCB} (compare \cite{DF14}). An analogous compatibility of standard and canonical bases  
under the homomorphism $\phi_d^\imath: \Ki \rightarrow \Si$ also holds.

We caution that the convention for $\Ui$, $\Uj$, and $\Hb$ used in this paper are chosen to fit most naturally 
to the convention from geometry, and it is not quite the same as in \cite{BW}. 

\introsection{}
This paper forms an essential part of a larger program as outlined in \cite[\S0.5]{BW}. 
The coideal algebras $\Ui$ and $\Uj$ (instead of the quantum groups of classical type) appear to be the suitable quantum algebras behind the
flag varieties and BGG category $\mathcal O$ of classical types (beyond type $A$). 

There are a few  followup projects of   this paper (and \cite{BW}) which will be pursued elsewhere. 
The type $D$ case will be treated separately. Various geometric realizations of coideal subalgebras of the quantum affine  algebras of type $A$ 
as well as classical symmetric pairs via Steinberg varieties will also be studied.
The constructions of $\imath$canonical bases are leading to an $\imath$categorification program, in which the geometric
constructions in this paper play a fundamental role. 

\vspace{.2cm}

Notations: $\mbb N$ denotes the set of nonnegative integers, $[a,b]$ denotes the set of integers between $a$ and $b$.

\vspace{.2cm}

\noindent{\bf Acknowledgement.}
We  thank Zhaobing Fan, Leonardo Mihalcea and Mark Shimozono for helpful conversations.
J.K. is partially supported by the NSF DMS-1160763,
Y.L. is partially supported by the NSF  DMS-1160351, while W.W. is
partially supported by the NSF DMS-1101268 and  DMS-1405131.  
We thank the referees for their careful readings and helpful comments.

\section{Geometric convolution algebras  of type $B$}
 \label{sec:convolution}

In this section, we construct a Schur algebra $\Sj$ and Iwahori-Hecke algebra $\Hb$ via convolution products
in the framework of partial flag varieties of type $B_d$. We also construct a $(\Sj, \Hb)$-bimodule $\Td$ geometrically.

\subsection{Preliminaries in type $A$}
\label{prelim}

Let us fix a pair $(N, D)$ of positive integers. 
Let $\mbb F_q$ be a finite field of $q$ elements where $q$ is always assumed to be odd in this paper. 
We shall denote by $|U|$ the dimension of a vector space $U$ over $\mbb F_q$. 
Consider the following data:
\begin{itemize}
\item The general linear group $GL(D)$ over $\mbb F_q$ of rank $D$;

\item The variety $\tilde X$ of $N$-step flags  $V=(0=V_0\subseteq V_1 \subseteq \cdots \subseteq   V_{N-1} \subseteq V_N= \mbb F_q^D)$  in $\mbb F_q^D$;

\item The variety $\tilde Y$  of complete flags $F=(0=F_0\subset F_1\subset \cdots \subset F_{D-1} \subset F_D=\mbb F_q^D)$  in $\mbb F_q^D$.
\end{itemize}

Let $GL(D)$ act diagonally on the products $\tilde X\times \tilde X$, $\tilde X\times \tilde Y$ and $\tilde Y\times \tilde Y$.
To a pair $(V, V') \in \tilde X\times \tilde X$, we can associate an $N \times N$ matrix $A=(a_{ij})$ by setting
\begin{equation*}
a_{ij} =| (V_{i-1}+ V_i \cap V_j') / (V_{i-1} +V_i \cap V_{j-1}')|, \quad \forall 1\leq i, j\leq N. 
\end{equation*}
The above assignment $(V, V') \mapsto A$
defines a bijection
\begin{equation}
\label{assignment}
GL(D)\backslash \tilde X\times \tilde X \longleftrightarrow \Theta_D,
\end{equation}
 where $GL(D) \backslash \tilde X\times \tilde X$ is  the set of $GL(D)$-orbits in $\tilde X\times \tilde X$ and  
\begin{equation*}
\Theta_D =\Big\{ A=(a_{ij}) \in \mrm{Mat}_{N\times N} (\mbb N) \mid \sum_{i, j\in[1, N]} a_{ij}= D \Big\}.
\end{equation*}
Here and below, $\mrm{Mat}_{k\times \ell}(R)$ denotes the set of $k\times \ell$ matrices with coefficients in $R$. 
Similar to (\ref{assignment}), we have bijections
\begin{equation}
\label{assignment-2}
GL(D) \backslash \tilde X\times \tilde Y \longleftrightarrow \tilde \Pi , \qquad GL(D)\backslash \tilde Y\times \tilde Y \longleftrightarrow \tilde \Sigma, 
\end{equation}
where  
\begin{equation*}
\begin{split}
\tilde \Pi & = \Big\{ B=(b_{ij}) \in \mrm{Mat}_{N\times D} (\mbb N) \mid \sum_{i\in [1,N]} b_{ij} =1, \forall j\in [1, D] \Big\}, \quad\\
\tilde \Sigma & =\Big\{ \sigma=(\sigma_{ij})\in \mbox{Mat}_{D\times D} (\mbb N) \mid \sum_{i\in [1, D]} \sigma_{ij} =1=\sum_{j \in [1,D]} \sigma_{ij}, \forall i, j\in [1, D]\Big\}.
\end{split}
\end{equation*}
By ~\cite{BLM} and ~\cite{GL92}, we have 
\begin{equation}
\label{number}
\# \tilde \Sigma = D!, \quad \# \tilde \Pi = N^D, \quad \mbox{and}\quad \# \Theta_D = \binom{N^2+D-1}{D}.
\end{equation}

For any $N \times N$ matrix $A = (a_{ij})$, we define 
\begin{align*}
\text{ro}(A) = \Big(\sum_{j}a_{1j}, \sum_{j}a_{2j}, \dots, \sum_{j}a_{Nj}\Big),\\
\text{co}(A) = \Big(\sum_{i}a_{i1}, \sum_{i}a_{i2}, \dots, \sum_{i}a_{iN}\Big).
\end{align*}

\subsection{Parametrizing  $O(D)$-orbits}
\label{odd-orth}

We fix a pair $(n, d)$ of positive integers such that 
\[
N=2n+1, \quad D= 2d+1,
\]
where $(N, D)$ is a pair of positive integers considered in Section \ref{prelim}.

Let us fix a non-degenerate symmetric  bilinear form
  $Q: \mbb F_q^D\times \mbb F_q^D \to \mbb F_q$. 
Let  $O(D)$ be the orthogonal  subgroup of $GL(D)$ 
consisting of elements $g$ such that $Q(gu, g u') = Q(u, u')$ for any $u, u' \in \mbb F_q^D$. 
If $U$ is a subspace of $\mbb F_q^D$, we write $U^{\perp}$ for its orthogonal complement. 
Consider the following subsets of $\tilde X$ and $\tilde Y$:
\begin{itemize}
\item $X =\{ V =(V_k) \in \tilde X\mid  V_i= V_j^{\perp}, \text{ if } i+j = N\}$; 
\item  $Y=\{ F =(F_\ell) \in \tilde Y \mid F_i = F_j^{\perp}, \text{ if } i+ j = D\}$. 
\end{itemize}
It is well known that  $O(D)$ acts on $X$ and $Y$. Let $O(D) $ act  diagonally on $X\times X$, $X\times Y$ and $Y\times Y$, respectively. 
Consider the following subsets of $\Theta_D$, $\tilde \Pi$ and $\tilde \Sigma$:
\begin{itemize}
\item $\Xi_d =\{ A =(a_{ij}) \in \Theta_D \mid a_{ij} = a_{N+1 -i, N+1-j}, \forall  i, j\in [1, N]\}$;
\item  $\Pi=\{B =(b_{ij}) \in \tilde \Pi \mid b_{ij} = b_{N+1- i, D+1-j}, \forall   i\in [1, N], j\in [1, D]\}$; 
\item $\Sigma=\{ \sigma =(\sigma_{ij})  \in \tilde \Sigma \mid \sigma_{ij}= \sigma_{D+1-i, D+1-j}, \forall   i, j\in [1, D]\}.$
\end{itemize}
Note that $a_{n+1,n+1}$ is odd for all $A\in \Xi_d$, and similarly, $b_{n+1, d+1} =1$ for all $B\in \Pi$. 
Also note that ${}^t A \in \Xi_d$ for $A \in \Xi_d$, where ${}^t A$ denotes the transpose of $A$.

\begin{lem}\label{lem:bijectionorbits}
The bijections in (\ref{assignment}) and (\ref{assignment-2}) induce the following bijections: 
\begin{equation}
\label{bijection}
O(D) \backslash X\times X \longleftrightarrow \Xi_d, \quad
O(D) \backslash X\times Y \longleftrightarrow \Pi, \quad \mrm{and}\quad 
O(D) \backslash Y\times Y \longleftrightarrow \Sigma. 
\end{equation}
\end{lem}
(We shall denote the orbit corresponding to a matrix $A$ by $\Ob_A$.)

\begin{proof}
The third bijection is the well-known Bruhat decomposition. We shall only prove the first one since the second one is similar.

Pick a pair $(V, V') \in X\times X$. Let $A$ be the associated matrix under the bijection (\ref{assignment}). We must show that $A\in \Xi_d$.
Observe that we have
\[
a_{ij} = | V_i\cap V_j' | - | V_{i-1}\cap V_j' | - |  V_i\cap V'_{j-1} | + | V_{i-1} \cap V'_{j-1}|.
\]
Since $V_i=V_{N-i}^{\perp}$, we have
\begin{equation*}
\begin{split}
| V_i\cap V_j' | &= |V_{N-i}^{\perp} \cap V_{N-j}^{'\perp}| = | (V_{N-i}+ V'_{N-j})^{\perp} | \\
&=D- |  V_{N-i} | - | V'_{N-j}| + | V_{N-i}\cap V'_{N-j} |.
\end{split}
\end{equation*}
This implies that 
$a_{ij} = a_{N+1-i, N+1-j}$. So we have $A\in\Xi_d$.

Then we need to show that the map $(V, V')\mapsto A$ is surjective. 
If we take a pair $(F,F')\in Y\times Y$ and throw away $F_i$ and $F_j'$ for some fixed $i$ and $j$, we get a pair
$(V, V') \in \tilde X_{2n}\times \tilde X_{2n}$, where $\tilde X_{2n}$ is the set of $2n$-step partial flags. 
Suppose that $\sigma$ is the associated matrix of $(F, F')$.
It is clear from ~\cite[1.1]{BLM} that the associated matrix $A$ of $(V, V')$ is the one obtained from $\sigma$ by
merging, i.e., adding componentwise,  the $i$-th row (resp. $(D+1-i)$-th) and $(i+1)$-th row  (resp. $(D-i)$-th row) and then merging
 the $j$-th column (resp. $(D+1-j)$-th) and $(j+1)$-th column (resp. $(D-j)$-th column). 
So any matrix $A$ in $\Xi_d$ can be obtained from a not necessarily unique $\sigma \in \Sigma$ by repetitively applying the above observation.
From this, we see that there is a pair $(V, V')\in X\times X$, obtained from a  pair in  $Y\times Y$ 
by throwing away subspaces at appropriate steps, such that its associated matrix is $A$. 
This shows that the map defined by $(V, V')\mapsto A$ is surjective. 

It remains to show that the assignment of  each $O(D)$-orbit of $(V, V')$ the matrix $A$ is injective. 
Without loss of generality, we assume that $n\leq d$ and $V_i\neq V_j$ for any $i\neq j$. 
Under such assumption, we have a nonzero entry at each row and each column and each flag in $X$ can be obtained from a flag in $Y$ by dropping  certain steps of flags.
Let $(V, V')$ and $(\tilde  V, V'')$ be two pairs in $X\times X $ such that their associated matrices are the same, say $A$.
We further assume that all entries in $A$ are either $0$ or $1$, except the entries $(i_0, j_0) $ and $(N+1-i_0, N+1-j_0)$ at which $A$ takes   value $2$.
The general case can be shown with a similar argument.
Since $O(D)$ acts transitively on $X$ and $Y$, we can assume that $V=\tilde V$. 
We suppose that $(V, V')$ and $(V, V'')$ are obtained from two pairs $(F, F')$ and $(F, F'')$ in $Y\times Y$, respectively. 
It is clear then that the associated matrices, say $\sigma'$ and $\sigma''$, of  $(F, F')$ and $(F, F'')$ differ only at the rank $2$ submatrix whose upper left corner is $(i_0, j_0)$.
Since  the associated matrix of $(F, F')$  is  $\sigma'$, the associated matrix of  the pair $(F, \tilde F')$ must be $\sigma''$, 
where $\tilde F'$ is a flag such that $F_i'=\tilde F_i'$ for all $i\neq j_0, D+1-j_0$ and $F_{i_0}' \neq F'_{i_0}$. 
From this, we see that there is a $g$ in the stabilizer of $F$ in $O(D)$ such that $g(F_i') = g(F''_i)$ for any $i\neq i_0, D+1-i_0$. This shows that
$(V, V')$ and $(V, V'')$ are in the same $O(D)$-orbit. We are done.
\end{proof}

To each $B =(b_{i,j}) \in \Pi$, we associate a sequence 
$\mbf r=(r_1, \ldots, r_D)$
of integers where 
\begin{equation}
\label{eq:rc}
\text{$r_c $ is the integer such that $b_{r_c, c} =1$ for each $c\in [1, D]$. }
\end{equation}
Note that $\mbf r$ is completely determined by the $d$-tuple $r_1 \cdots r_d$. 
This defines the following bijections 
\begin{align}
\label{bij:rr}
\begin{split}
\Pi &\longleftrightarrow \{\text{sequences } \mbf r =(r_1, \ldots, r_D) \text{ with each } r_c\in [1, N] \text{ and }  r_c + r_{D+1-c} = N+1\}
\\
&\longleftrightarrow \{\text{sequences } r_1\cdots r_d \text{ with each } r_c\in [1, N]\}.
\end{split}
\end{align}
 Hence we can and shall denote the characteristic function of the $O(D)$-orbit $\Ob_B$ by $e_{r_1 \dots r_d}$.

The following is a counterpart of (\ref{number}).

\begin{lem}
\label{Case-I-dimensions}
We have
$\# \Sigma = 2^d \cdot d!, \quad
\# \Pi = (2n+1)^d, \quad$  and  
$\quad \#\Xi_d = \binom{2n^2+ 2n +d }{d}$.
\end{lem}

\begin{proof}
The first identity is well known and can be seen directly. The second identity follows from the bijection we defined before the lemma. We shall show the third identity. 

Any matrix $A\in \Xi_d$ satisfies the following condition:
\[
\sum_{i,  j\in [1, N]} a_{ij} = 2\sum_{i\in [1,n], j\in [1,N]} a_{ij} + 2 \sum_{j\in [1,n]} a_{n+1, j}  + a_{n+1, n+1} = 2d+1.
\]
So $a_{n+1, n+1}$ must be a positive odd number in $[1, 2d+1]$ and 
\[
\sum_{i\in [1,n], j\in [1,N]} a_{ij} +  \sum_{j\in [1,n]} a_{n+1, j} = \frac{2d+1-a_{n+1,n+1}}{2}.
\]
Thus
\[
\# \Xi_d = \sum_{l=0}^d \binom{2n^2+2n-1+d-l}{d-l}=\binom{2n^2+2n+d}{d}.
\]
The lemma follows.
\end{proof}

\subsection{Convolution algebras in action}
\label{sec:convolutionaction}

Let $v$ be an indeterminate and $\A =\Z[v,v^{-1}]$.   
We define 
$$
\Sj =\Axx, \qquad \Td =\Axy, \qquad \Hb =\Ayy
$$
to be the space of $O(D)$-invariant $\A$-valued functions on $X \times X$, $X \times Y$, and $Y \times Y$ respectively.
(Note by Lemma~\ref{lem:bijectionorbits} that the parametrizations of the $G$-orbits
are independent of the finite fields $\mathbb F_q$.)
For $A \in \Xi_d$, we denote by $e_A$ the characteristic function of the orbit $\Ob_A$. 
Then $\Sj$ is a free $\A$-module with a basis $\{e_A \mid A \in \Xi_d\}$. Similarly, $\Td$ and $\Hb$ are 
free $\A$-modules with bases parameterized by $\Pi$ and $\Sigma$, respectively. 

We define a convolution product $*$ on $\Sj$ as follows.
For a triple of matrices $(A, A', A'')$ in $\Xi_d \times \Xi_d \times \Xi_d$,  
we choose $(f_1, f_2) \in \Ob_{A''}$, and we let $g_{A,A',A'';q}$ be the number of $f\in X$ such that $(f_1, f) \in \Ob_{A}$ and $(f, f_2) \in \Ob_{A'}$.
A well-known property of the Iwahori-Hecke algebra implies that there exists a polynomial $g_{A,A',A''} \in \Z[v^2]$ such that
$g_{A,A',A'';q} =g_{A,A',A''}|_{v=\sqrt{q}}$ for every odd prime power $q$. 
We define the convolution product  on $\Sj$ by letting 
$$
e_A * e_{A'} =\sum_{A''} g_{A,A', A''} e_{A''}.
$$
Equipped with the convolution product, the $\A$-module $\Sj$ becomes an associative $\A$-algebra.  
(A completely analogous convolution product gives us an $\A$-algebra structure on
$\Hb$, which is well known to be the Iwahori-Hecke algebra of type $B_d$.)

An analogous convolution product (by regarding the triples $(A, A', A'')$ as  in $\Xi_d \times \Pi \times \Pi$ and $f\in Y$) gives us
a left $\Sj$-action on $\Td$; a suitably modified convolution gives us a right $\Hb$-action on $\Td$. These two actions commute
and hence we have obtained an $(\Sj, \Hb)$-bimodule structure on $\Td$. 
Denote 
$$\QQ\Sj =\Qq \otimes_\A \Sj,\quad
 \QQ\Hb =\Qq \otimes_\A \Hb, \quad
\QQ\Td =\Qq \otimes_\A \Td.
$$

\begin{rem}
\label{rem:Howe}
Let  us write $X_n$ for the variety $X$ of the $N$-step (isotropic) flags in $\mbb F_q^D$ and  write ${}^n\Sj$ for $\Sj$ for now (recall $N=2n+1)$. 
Let $M=2m+1$ be another odd positive integer.
Then via convolutions we can define a left action of ${}^n\Sj$ and a commuting right ${}^m\Sj$-action
 on $\A_{O(D)}(X_n\times X_m)$ which can be shown to form double centralizers. 
 This is a type $B$ variant of the geometric symmetric Howe duality of the type $A$ considered in \cite{W01}.
\end{rem}

\subsection{Geometric action of Iwahori-Hecke algebra}

For $1\leq j\leq d$, we define $T_j \in \Hb$ by 
\[
T_j (F, F') =
\begin{cases}
1, & \mbox{if} \,\, F_i = F_i' \quad \forall i\in [1, d]\backslash  \{j\} \text{ and } F_j\neq F_j';\\
0, & \mbox{otherwise}.
\end{cases}
\]
It is well known that $\Hb$ is isomorphic to the Iwahori-Hecke algebra of type $B_d$, which
is an $\A$-algebra generated by  $T_i$ $(1\le i \le d)$ subject to the relations:
$(T_i -v^2)(T_i +1) = 0$ for $1\le i\le d,$    
$T_i T_{i+1} T_i = T_{i+1} T_i T_{i+1} \text{ for } 1\le i < d-1,$
$T_i T_j = T_j T_i  \text{ for } |i-j| >1,$
$T_d T_{d-1} T_d T_{d-1}=T_{d-1} T_d T_{d-1} T_d.$

\begin{lem}
\label{Hecke-action}
The right $\Hb$-action on $\Td$,  
$$
\Td \times \Hb \longrightarrow \Td, \qquad
(e_{r_1 \dots r_d}, T_j ) \mapsto e_{r_1 \dots r_d} T_j,
$$
is given as follows.
For $1\leq j\leq d-1$, we have
\begin{equation}
\label{1}
e_{r_1 \dots r_d} T_j =
\begin{cases}
e_{r_1\dots r_{j-1} r_{j+1} r_jr_{j+2} \dots r_d} ,  &  \text{ if }r_j < r_{j+1};\\
v^2 e_{r_1\dots r_d}, &  \text{ if }r_j = r_{j+1};\\
(v^2-1) e_{r_1\dots r_d} + v^2 e_{r_1 \dots r_{j-1} r_{j+1} r_jr_{j+2}  \dots r_d}, & \text{ if } r_j> r_{j+1}.
\end{cases}\\
\end{equation}
Moreover, (recalling $r_{d+2} = N+1 - r_d$) we have
\begin{equation}
\label{2}
e_{r_1 \dots r_{d-1} r_{d}} T_d =
\begin{cases}
e_{r_1\dots r_{d-1}r_{d+2}} ,  &\text{ if } r_d < n+1;\\
v^2 e_{r_1\dots r_{d-1} r_d}, & \text{ if } r_d = n+1;\\
(v^2-1) e_{r_1\dots r_{d-1} r_d} + v^2 e_{r_1 \dots r_{d-1}r_{d+2}}, & \text{ if } r_d> n+1.
\end{cases}\\
\end{equation}
\end{lem}

\begin{proof}
It suffices to prove the formula for $v=\sqrt{q}$;
by definition we interpret the convolution product over $\mathbb F_q$ as
$e_{r_1 \dots r_d} T_j (V, F) =\sum_{F'\in Y} e_{r_1 \dots r_d}(V, F')  T_j (F', F)$. 

The above formula (\ref{1})  coincides with the one in  ~\cite[1.12]{GL92}, whose proof is also the same as in the  type $A$ case. We shall  prove (\ref{2}).
By definitions, we have
\begin{align*}
e_{r_1 \dots r_{d}} T_d (V, F) 
 &=\sum_{F'\in Y} e_{r_1 \dots r_{d}} (V, F') T_d(F', F)
 \\
 &=\sum_{F'\in Y: F'_d\neq F_d, F'_i=F_i,\forall i\in [1,d-1]} e_{r_1 \dots r_{d}} (V, F'). 
\end{align*}
From the above expression, we see that the value of $e_{r_1 \dots r_{d}} T_d$ at $(V, F)$ is zero if the associated sequence of $(V, F)$
is not the one listed in the formula.   It is also clear that the calculation is reduced to the case when $D=3$
by comparing the pairs $(V, F)$ with the pair obtained from (V, F) by intersecting with $F_{d+2}$ and modulo $F_{d-1}$,  
and in this case the formula  can be derived by a direct computation. 
(Note that the restriction of $Q$ to the quotient $F_{d+2}/F_{d-1}$ is again a non-degenerate form.)
\end{proof}

We set
\[
\tilde e_{r_1\cdots r_d} = v^{\#\{(c, c')| c, c'\in [1, d], c<c', r_c < r_{c'}\} + \epsilon} e_{r_1\cdots r_d}, 
\]
where
\[
\epsilon=
\begin{cases}
1, & \mbox{if} \;  r_d < n+1;\\
0, & \mbox{otherwise}.
\end{cases}
\]
The fomulas (\ref{1}) and (\ref{2}) can be rewritten as follows:
\begin{equation}
\label{3}
\tilde e_{r_1 \dots r_d} T_j =
\begin{cases}
v \tilde e_{r_1\dots r_{j-1} r_{j+1} r_j\dots r_d} ,  & \text{ if } r_j < r_{j+1};\\
v^2 \tilde e_{r_1\dots r_d}, &  \text{ if } r_j = r_{j+1};\\
(v^2-1) \tilde e_{r_1\dots r_d} + v \tilde  e_{r_1 \dots r_{j-1} r_{j+1} r_jr_{j+2}  \dots r_d}, & \text{ if } r_j> r_{j+1},
\end{cases}\\
\end{equation}
for $1\leq j\leq d-1$, and  
\begin{equation}
\label{4}
\tilde e_{r_1 \dots r_{d}} T_d =
\begin{cases}
v \tilde e_{r_1\dots r_{d-1} r_{d+2} } ,  &  \text{ if } r_d < n+1;\\
v^2  \tilde e_{r_1\dots r_d}, &  \text{ if }  r_d = n+1;\\
(v^2-1) \tilde e_{r_1\dots r_d} + v \tilde e_{r_1 \dots r_{d-1} r_{d+2}},  &  \text{ if } r_d> n+1.
\end{cases}\\
\end{equation}

\section{Structures of the Schur algebra $\Sj$}
 \label{sec:Schur}

In this section, we establish some fundamental multiplication formulas for the algebra $\Sj$ and
its action on $\Td$. Then we establish a monomial basis and a canonical basis for $\Sj$. 

\subsection{Relations for $\Sj$}
\label{sec:rel}

We shall use the notation $U\overset{a}{\subset}  W$ to denote that $U$ is a subspace of $W$ of codimension $a$ (for $a=1,2$). 
For $i\in [1, n]$, $a \in [1, n+1]$, and for $V, V' \in X$, we set 
\begin{align}
\E_i (V, V') &=
\begin{cases}
v^{-|V'_{i+1}/V'_i|}, &\mbox{if}\; V_i\overset{1}{\subset} V_i', V_j=V_{j}',\forall j\in [1,n]\backslash \{i\}; \\
0, &\mbox{otherwise.}
\end{cases}\\
\F_i (V, V') &=
\begin{cases}
v^{-|V'_i/V'_{i-1}|}, &\mbox{if}\; V_i\overset{1}{\supset} V_i', V_j=V_{j}',\forall j\in [1,n]\backslash \{i\}; \\
0, &\mbox{otherwise.}
\end{cases}\\
\mbf d_a^{\pm 1} (V, V') & =
\begin{cases}
v^{\mp(|V_a'/V_{a-1}'|)} , &\mbox{if}\; V=V';\\
0, & \mbox{otherwise.}
\end{cases}
\end{align}
We set $\mbf d_a =\mbf d_a^{+1}$. Clearly, $\E_i, \F_i,, \mbf d_a, \mbf d_a^{-1}$ lie in the $\A$-algebra $\Sj$. 

Let $\, \bar{} \,$ be the bar involution on $\A$ and $\Qq$ by sending $v\mapsto v^{-1}$. 
Let
$$
\llbracket r \rrbracket  = \displaystyle \frac{v^{r} - v^{-r}}{v - v^{-1}}, \qquad \text{ for } r\in \Z,
$$
be the (bar-invariant) quantum integer $r$.
In the following proposition we adopt the convention of dropping the product symbol $*$ to make the formulas more readable.

\begin{prop}
\label{Type-B-case-I-homomorphism}
The following relations hold in $\Sj$: for $i, j \in [1, n]$ 
\begin{align*}
\mbf d_{i} \mbf d_{i}^{-1} &=\mbf d_{i}^{-1} \mbf d_{i} =1, \displaybreak[0]\\
  \mbf d_{i}  \mbf d_{i} &=   \mbf d_{i}  \mbf d_{i}, \displaybreak[0]\\
\mbf d_{i} \E_{j} \mbf d_{i}^{-1} &
 = v^{\delta_{i, j} -\delta_{i, j+1} } \E_{j}, \displaybreak[0]\\
 \mbf d_{i} \F_{j}  \mbf d_{i}^{-1} &
= v^{-\delta_{i, j} + \delta_{i, j+1}  }\F_{j}, \displaybreak[0]\\
\E_{i} \F_{j} -\F_{j} \E_{i} &= \delta_{i,j} \frac{ \mbf d_{i} \mbf d_{i+1}^{-1}
 - \mbf d_{i}^{-1} \mbf d_{i+1}}{v-v^{-1}},        & &\text{if } i, j \neq n, \qquad \qquad \displaybreak[0]\\
 \E_{i}^2 \E_{j} +\E_{j} \E_{i}^2 &= \llbracket 2 \rrbracket   \E_{i} \E_{j} \E_{i},   &&\text{if }  |i-j|=1, \displaybreak[0]\\
\qquad \qquad  \F_{i}^2 \F_{j} +\F_{j} \F_{i}^2 &= \llbracket 2 \rrbracket  \F_{i} \F_{j} \F_{i}, &&\text{if } |i-j|=1,\displaybreak[0]\\
 \E_{i} \E_{j} &= \E_{j} \E_{i},     &&\text{if } |i-j|>1, \displaybreak[0]\\
 \F_{i} \F_{j}  &=\F_{j}  \F_{i},  &&\text{if } |i-j|>1,\displaybreak[0]\\
\mbf d_{n+1} \mbf d_{i} &= \mbf d_{i} \mbf d_{n+1}, \; \mbf d_{n+1} \mbf d_{n+1}^{-1} = \mbf d_{n+1}^{-1}\mbf d_{n+1} = 1, \tag{a}\displaybreak[0]\\
 \mbf d_{n+1} \E_i \mbf d_{n+1}^{-1}&= 
v^{-2\delta_{n, i} } \E_i, \; \mbf d_{n+1} \F_i \mbf d_{n+1}^{-1} = v^{2\delta_{n, i}} \F_i,  \tag{a$'$}\displaybreak[0] \\
%
%
 \E_{n}^2\F_{n} + \F_{n}\E_{n}^2
   &= \llbracket 2 \rrbracket  \Big(\E_{n}\F_{n}\E_{n}- 
     \E_{n} (  v \mbf d_{n} \mbf d_{n+1}^{-1}+
     v^{-1} \mbf d_{n}^{-1} \mbf d_{n+1} ) \Big),\tag{b}\\
 \F_{n}^2\E_{n} + \E_{n}\F_{n}^2
    &= \llbracket 2 \rrbracket  \Big(\F_{n}\E_{n}\F_{n}- 
     ( v\mbf d_{n} \mbf d_{n+1}^{-1} +  v^{-1} \mbf d_{n}^{-1} \mbf d_{n+1}  )\F_{n} \Big).\tag{c}
\end{align*}
\end{prop}

\begin{proof}
It suffices to prove the formulas when we specialize $v$ to $\texttt{v}\equiv \sqrt{q}$
and then perform the convolution products over $\mathbb F_q$. 

The relations above except the labeled ones are identical to the type $A$ case and hence are verified as in \cite{BLM}. 

Let us verify (b).
Without loss of generality, we assume that $n=1$.  
We have 
\begin{eqnarray*}
\E_1 \E_1 \F_1 (V, V') & =
\begin{cases}
\texttt{v}^{-(2D-3|V_1|-5)} (q+1) \frac{q^{D- 2(|V_1| +1)-1} -1}{q-1} , &\mbox{if }\; V_1 \overset{1}{\subset}  V_1';\\
\texttt{v}^{-(2D-3|V_1|-5)}(q+1), & \mbox{if } \; |V_1\cap V_1'|=|V_1|-1 = |V'_1|-2, \\
&\quad  \mbox{$V_1+V'_1$ is isotropic};\\
0, &\mbox{otherwise.}
\end{cases} 
\end{eqnarray*}

Indeed, by definition, we have
\[
\E_1 \E_1 \F_1 (V, V') = \texttt{v}^{-(2D-3|V_1|-5)} \#Z_1,
\]
where $Z_1$ consists of all pairs $(\tilde V_1, \hat V_1)$ of isotropic subspaces in $\mbb F_q^D$ with respect to
the form $Q$ such that $V_1\overset{1}{\subset} \tilde V_1 \overset{1}{\subset} \hat V_1 \overset{1}{\supset} V'_1$.
If $V_1\subset V'_1$, then the condition $V_1\subset \hat V_1$ is redundant. This implies that 
the number of choices for $\hat V_1$ is the same as the number of 
all isotropic lines in the space $(V'_1)^{\perp}/V'_1$ with
respect to the form induced from $Q$. The latter number is equal to
\[
\frac{q^{D- 2|V_1'| -1} -1}{q-1}  = \frac{q^{D- 2(|V_1| +1)-1} -1}{q-1}.
\]
For a fixed isotropic subspace $\hat V_1$ such that $\hat V_1 \overset{1}{\supset} V'_1$, there are clearly 
$(q+1)$ choices of $\tilde V_1$ such that $(\tilde V_1, \hat V_1) \in Z_1$. So we just obtain the formula
when $V_1\overset{1}{\subset} V'_1$.
If $|V_1\cap V_1'|=|V_1|-1 = |V'_1|-2$, then $V_1\not \subset V'_1$. This implies that 
if $(\tilde V_1, \hat V_1)\in Z_1$, then $\hat V_1= V_1+V'_1$, which must be isotropic. 
In particular, if $|V_1\cap V_1'|=|V_1|-1 = |V'_1|-2$ and $V_1+V'_1$ is isotropic, we see that 
$\#Z_1$ is the same as the number of subspaces $\tilde V_1$ such that $V_1\overset{1}{\subset} \tilde V_1\subset V_1+V'_1$. The latter number is clearly $q+1$. This implies the second entry in the above formula.
For the remaining case, the value $\E_1 \E_1 \F_1 (V, V')$ is  zero. So we obtain the above formula.

Moreover, we have 
\begin{eqnarray*}
\F_1 \E_1 \E_1 (V, V') & = 
\begin{cases}
\texttt{v}^{- (2D-3|V_1|-3 )} (q+1) \frac{q^{|V_1| } -1} {q-1} , & \mbox{if }\; V_1\overset{1}{\subset} V_1';\\
\texttt{v}^{- (2D-3|V_1|-3 )} (q+1), &\mbox{if } \; |V_1\cap V_1'|=|V_1|-1=|V'_1|-2; \\
0, & \mbox{otherwise.}\\
\end{cases}
\end{eqnarray*}

In fact, we have 
\[
\F_1 \E_1 \E_1 (V, V') =\texttt{v}^{- (2D-3|V_1|-3 )}  \# Z_2,
\]
where $Z_2$ is the set of all pairs $(\tilde V_1, \hat V_1)$ of isotropic subspaces such that
$V_1\overset{1}{\supset} \tilde V_1 \overset{1}{\subset} \hat V_1\overset{1}{\subset} V'_1$.
Since $\tilde V_1$ and $\hat V_1$ are contained in $V_1$ or $V'_1$, they are automatically isotropic.
It is straightforward to see that the value of $\F_1 \E_1 \E_1 (V, V')$ is zero unless the pair $(V, V')$ satisfies
the first or  second condition in the above formula.
If the first one, i.e., $V_1\overset{1}{\subset} V'_1$, is satisfied, then
the number of choices for $\tilde V_1$ is $\frac{q^{|V_1|}-1}{q-1}$ since the condition $\tilde V_1\subset V'_1$ 
holds automatically.
Fixing such $\tilde V_1$, we see that all possible choices of $\hat V_1$ are $(q+1)$, i.e., the number of lines in a two-dimensional vector space over $\mbb F_q$. So we have the formula under the first condition.
If the second condition in the formula is satisfied, then $\tilde V_1$ must be equal to $V_1\cap V'_1$, and it is 
clear that the factor $\hat V_1$ contributes to $(q+1)$ number of choices. So the formula holds under the second condition.
We thus obtain the above formula. 

Finally, we have
\begin{eqnarray*}
\E_1 \F_1 \E_1 (V, V') & = 
\begin{cases}
\texttt{v}^{-(2D-3|V_1|-4)} \left (\frac{q^{D-2|V_1|-1}-1}{q-1} + \frac{q^{|V_1|+1} - q}{q-1} \right ), &\mbox{if }\; V_1\overset{1}{\subset} V_1';\\
\texttt{v}^{-(2D-3|V_1|-4)}(q+1), & \mbox{if } \; |V_1\cap V_1'|=|V_1|-1=|V'_1|-2, \\
& \quad \mbox{$V_1+ V'_1$ is isotropic};\\
\texttt{v}^{-(2D- 3 |V_1|-4)} , &\mbox{if } \; |V_1\cap V_1'| = |V_1|-1=|V'_1|-2,\\
& \quad \mbox{$V_1+V'_1$ is not isotropic};\\
0, & \mbox{otherwise.}
\end{cases}
\end{eqnarray*}

This formula can be proved as follows.
By definition, there is 
\[
\E_1 \F_1 \E_1 (V, V') = \texttt{v}^{-(2D-3|V_1|-4)} \# Z_3,
\]
where $Z_3$ is the set of all pairs $(\tilde V_1, \hat V_1)$ of isotropic subspaces such that 
$V_1 \overset{1}{\subset} \tilde V_1 \overset{1}{\supset} \hat V_1 \overset{1}{\subset} V'_1$.
Assume now that $V_1\overset{1}{\subset} V'_1$. 
We set $Z_3' $ to be the subset of $Z_3$ such that $\hat V_1= V_1$. Let $Z''_3= Z_3-Z'_3$.
It is clear that $Z'_3$ is in bijection with the set of isotropic lines in $V_1^{\perp}/V_1$ with respect to the form induced from $Q$. So $\# Z'_3= \frac{q^{D-2|V_1|-1}-1}{q-1}$.
To count the set $Z''_3$, we introduce the auxiliary set $Z'''_3$ of all codimension  1 subspaces in $V_1$. 
Note that if $(\tilde V_1, \hat V_1) \in Z''_3$, then $\tilde V_1 = V_1+\hat V_1$. So $Z''_3$ is completely determined by
the choices of $\hat V_1$.
So there is a surjective map $Z''_3 \to Z'''_3$ given by $(\tilde V_1, \hat V_1) \mapsto V_1\cap \hat V_1$.
It is clear that this is a fiber bundle with fiber of $q$ elements. So we have
\[
Z''_3= q\cdot  \# Z'''_3 = q\cdot \frac{q^{|V_1|}-1}{q-1}= \frac{q^{|V_1|+1} - q}{q-1}.
\]
By combining the above analysis, we obtain the formula under the first condition.
Now we determine the value of $\E_1 \F_1 \E_1 (V, V')$ under the second condition.
It is clear that we have $V_1+\hat V_1=\tilde V_1$. This in return implies that 
$|V_1\cap V'_1|\geq |V_1\cap \tilde V_1| \geq |V_1|-1$. So we have $V_1\cap V'_1 \overset{1}{\subset} \hat V_1$.
Hence there are $(q+1)$ choices of $\hat V_1$, which is also the same as the number of elements in $Z_3$. 
Thus the formula is true under the second condition.
The proof of the formula under the third condition is similar to that for the second condition, except 
that we must assure that $V_1+\hat V_1$ is isotropic, which forces the choice of $\hat V_1$ to be only one.
So the formula holds under the third condition.
For the remaining case, it is routine to see that the value of $\E_1 \F_1 \E_1 (V, V')$ is zero.
The proof of the above formula is thus finished.

Now (b) follows from the above three formulas by a direct computation.

The involution $(V, V')\mapsto (V', V)$ defines  an $\A$-linear anti-automorphism $\tau$  on $\Sj$ such that
for any  $V, V' \in X$, 
\begin{align*}
\tau(\E_1) (V, V') &= \texttt{v}^{-(D-3|V_1|+1)} \F_1 (V, V'),\\
\tau(\F_1) (V, V') &=\texttt{v}^{D-3|V_1| -2} \E_1 (V, V').
\end{align*}
By applying the anti-automorphism $\tau$ to (b), we obtain (c). 
The verifications of (a) and (a$'$) are easy and will be skipped. 
\end{proof}

\subsection{Multiplication formulas}

For $i, j\in [1, N]$,  let $E_{ij}$ be the standard elementary matrix in $\mrm{Mat}_{N\times N}(\mbb N)$.
Let
\begin{equation} 
 \label{eq:Etheta}
E_{ij}^{\theta} = E_{i j} + E_{N+1-i, N+1-j}.
\end{equation}
The $(i, j )$-entry of $E_{ij}^{\theta}$ will be denoted by $\epsilon_{ij }^{\theta}$. Note that 
\[
\epsilon_{ij }^{\theta}=
\begin{cases}
2, & \text{ if } i=j=n+1;\\
1, & \text{ otherwise. }
\end{cases}
\]
Recall that the set  $\{e_{A}\mid A\in \Xi_d\}$   is  an $\A$-basis for $\Sj$. 
The following lemma is a counterpart of \cite[Lemma~3.2]{BLM}.

\begin{lem}
\label{BLM3.2}
\begin{itemize}
\item[(a)]  For $A, B\in \Xi_d$ such that $\ro(A) =\co(B)$ and  $B- E_{h, h+1}^{\theta}$ is a diagonal matrix for some $h\in [1,n]$, we have 
\[
e_B * e_A = \sum_{p\in [1, N], a_{h+1,p}\geq \epsilon_{h+1, p}^{\theta}} v^{2\sum_{j>p} a_{hj}} 
 \frac{v^{2(1+ a_{hp})} -1}{v^2-1} e_{A+ E_{hp}^{\theta} - E_{h+1,p}^{\theta} }.
\]

\item [(b)] For  $A, C\in \Xi_d$ such that $ \ro(A) =\co ( C)$ and  $C- E_{h+1, h}^{\theta}$ is a diagonal matrix for some $h\in [1,n]$, we have 
\[
e_C * e_A = \sum_{p\in [1, N], a_{hp}\geq 1 } v^{2\sum_{j<p} a_{h+1,j} } \frac{ v^{2(1 +a_{h+1,p})} -1} {v^2-1} e_{A-E_{hp}^{\theta}+ E_{h+1,p}^{\theta}}.
\]
\end{itemize}
\end{lem}

\begin{proof}
The proof for case (a) with $h\in [1,n]$ and for case (b) with $h \in [1, n-1]$ is essentially the same as  the proof of \cite[Lemma~3.2]{BLM}, 
and hence will not be repeated here.  We shall prove the new case when $h=n$ in (b) as follows. 
As before, the proof is further reduced to analogous results over finite fields by specializing $v$ to $\texttt{v}\equiv \sqrt q$.
Under the assumption of (b) and $h=n$, we have
\[
e_C * e_A = \sum_{p\in [1, N], a_{np}\geq 1 } \# G_p   \;e_{A-E_{np}^{\theta}+ E_{n+1,p}^{\theta}},
\]
where the set $G_p$ consists of all subspaces $S$ in $\mbb F_q^D$  determined by the following conditions:
\begin{itemize}
\item  $S$ is isotropic; 
\item $V_n\subset S$ and $| S/V_n |= 1$; 
\item $V_n\cap V_j' =S\cap V_j'$ for $j< p$ and $V_n \cap V_j'\neq S\cap V_j'$ for $j\geq p$; 
\item $(V, V')$ is a fixed pair of flags in $X$ whose associated matrix is $A-E_{np}^{\theta}+ E_{n+1,p}^{\theta}$.
\end{itemize}
This is obtained by an argument similar to ~\cite[\S3.1]{BLM}.
So the problem is reduced to compute the number $\# G_p$. 

First, we consider  the case when $h=n$ and $p\leq n$.  The situation is the same as \cite{BLM} when we observe that the subspace
$V_n+V_n^{\perp}\cap V_j'$ is  isotropic  if $V_n$ and $V_j'$ are isotropic. 

Next, we consider the case when $h=n$ and $p=n+1$. 
We set
\[
G_{n+1}' =\big\{T \; \mbox{isotropic} \mid V_n+ V_n^{\perp}\cap V_n' \overset{1}{\subset} T \subseteq V_n + V_n^{\perp}\cap V_{n+1}'
\big\}.
\]
It is clear that 
\[
\#G'_{n+1} = \frac{q^{a_{n+1,n+1} +2 -1}-1}{q-1}= \frac{q^{1+ a_{n+1,n+1} }-1}{q-1}.
\] 
We define a map $\Psi_{n+1}: G_{n+1} \to G'_{n+1}$ by $\Psi_{n+1}(S) = S+ S^{\perp}\cap V_n'$.  For a fixed $T \in G_{n+1}'$, we can identify a point $W$ in $\Psi_{n+1}^{-1}(T)$ with the vector space 
$V_n \oplus  \frac{V_n+V_n^{\perp}\cap V_n'} {V_n} \oplus L$ where $L$ is an isotropic subspace of dimension $1$ in $V_n+V_n^{\perp}\cap V_{n+1}'$ in an appropriate decomposition of $\mbb F_q^D$.
Under such an identification. we see that the fiber $\Psi_{n+1}^{-1}(T)$ is isomorphic to $ \frac{V_n+V_n^{\perp}\cap V_n'} {V_n}$.
So $\Psi_{n+1}$ gives us a vector bundle $G_{n+1}$ over $G'_{n+1}$ with rank equal to $|  \frac{V_n+V_n^{\perp}\cap V_n'} {V_n} | =\sum_{j<n+1} a_{n+1,j}$. 
We thus have  
\[
\# G_{n+1} = \# \Psi_{n+1}^{-1}(T) \cdot \# G'_{n+1} =  q^{\sum_{j<n+1} a_{n+1,j}}  \frac{q^{1+a_{n+1,n+1} }-1}{q-1}. 
\]
So the formula in (b) holds for $h=n$ and $p=n+1$.

Finally,  we consider the case when $h=n$ and $p\geq n+2$.  Let $G_p'$ be the set of all flags  $W=(W_i)_{1\leq i\leq n}$  in $\mbb F_q^D$ subject to the following conditions:
\begin{itemize}
\item $W_i$ is  isotropic for $1\leq i\leq n$ and  $V_n \subseteq W_1\subseteq W_2 \subseteq \cdots \subseteq W_n$.
\item $V_n+V_n^{\perp}\cap V_i' \subset W_i$ and $| \frac{W_i}{V_n+V_n^{\perp}\cap V_i' }| = 1$ for $1\leq i\leq N-p$.
\item $|V_n+V_n^{\perp}\cap V_i' |= |W_i|$ and $| \frac{W_i}{V_n+V_n^{\perp}\cap V_i' }| = 1$ for $N-p+1\leq i\leq n$.
\item $W_1 \not\subseteq V_n +V_n^{\perp}\cap V_{p-1}'$. 
\end{itemize}
We define a map
$
\Psi: G_p \longrightarrow G_p'
$
by $\Psi(S) = (S+S^{\perp}\cap V_i ')_{1\leq i\leq n} $. Let us fix a flag  $W=(W_i)_{1\leq i\leq n}$ in $G_p'$. Then the subspace $W_1$ can be rewritten as 
\begin{equation}
\label{W_1}
W_1 \simeq V_n \oplus  \frac{V_n+V_n^{\perp} \cap V_1'}{V_n} \oplus \langle w_1\rangle,
\end{equation}
where $w_1$ is a vector not contained in $V_n+ V_n^{\perp}\cap V_{p-1}'$ and $\langle w_1\rangle $ is the subspace spanned by $w_1$.
One can check that 
\begin{equation}
\label{fiber-phi}
\Psi^{-1}((W_i)_{1\leq i\leq n})\simeq \{ V_n\oplus \langle w_1+x\rangle | x\in  \frac{V_n+V_n^{\perp} \cap V_1'}{V_n}\}\simeq \mbb F_q^{a_{n+1, 1}},
\end{equation}
if  the two vector spaces in (\ref{W_1}) are identified. This implies that $\Psi$ is surjective and a vector bundle of fiber dimension $a_{n+1, 1}$. 

Let $I_p$ be the set  of all flags $U=(U_{N-p+1}\subseteq \cdots \subseteq U_n)$  subject to the following conditions:
\begin{itemize}
\item $V_n+ V_n^{\perp} \cap V_{N-p}' \subset U_{N-p+1} \subseteq V_n+V_n^{\perp}\cap V_{N-p+1}'$ and $|\frac{V_n+V_n^{\perp}\cap V_{N-p+1}'}{U_{N-p+1}} |=1$;
\item  $V_n+ V_n^{\perp}\cap V_{i-1}' \not \subseteq U_i \subset V_n+V_n^{\perp} \cap V_{i} $ and $|\frac{V_n+V_n^{\perp}\cap V_{i}'}{U_i}| =1$ for $N-p+2\leq i\leq n$.
\end{itemize}
We stratify $G_p'$ as
\[
G_p' =\bigsqcup_{U\in I_p }  G'_{p, U}, \quad
G'_{p, U} =\{ W\in G_p' | W_i\cap (V_n+V_n^{\perp}\cap V_j) = U_i, \; \forall N-p+1\leq i\leq n\}.
\]
Inside $G'_{p, U}$, the subspace $W_n$ is subject to the conditions:
\[
U_n\subset W_n \subset U_n^{\perp}, \quad W_n \neq V_n + V_n^{\perp}\cap V_n' 
\quad \mbox{and}\quad 
W_n \not \subseteq V_n+V_n^{\perp}\cap V_{n+1}'. 
\]
The number of choices for such a $W_n$ is 
\[
\frac{q^{a_{n+1,n+1} +1}-1}{q-1} - 1 -  q \frac{ q^{a_{n+1,n+1}-1}-1}{q-1} = q^{a_{n+1, n+1}}.
\]
Fixing $W_n$, we see that the number of choices for $W_{n-1}$ is $q^{|U_n/U_{n-1}|}=q^{a_{n+1,n}}$. Inductively, we have 
\begin{equation}
\label{p-U}
\# G_{p, U}' =\prod_{1\leq i\leq n} q^{a_{n+1, i+1}}.
\end{equation}

We now consider the index set $I_p$. The number of choices for $U_{N-p+1}$ is 
$\frac{q^{1+a_{n+1, N-p+1}}-1}{q-1}$. Fixing $U_{N-p+1}$, we see that the number of choices for $U_{N-p+2}$ is 
$q^{a_{n+1, N-p +2}}= q^{a_{n+1,p}}$. Inductively, we conclude that 
\begin{equation}
\label{I-p}
\# I_p =\frac{q^{1+ a_{n+1,p}} -1}{q-1} \prod_{n+2\leq i\leq p-1} q^{a_{n+1, i}} . 
\end{equation}
By putting together (\ref{fiber-phi}), (\ref{p-U}) and (\ref{I-p}), we see immediately that 
\[
\# G_p = q^{a_{n+1, 1}} \# G_p' = q^{a_{n+1, 1}} \# I_p \# G_{p, U}' = \frac{q^{1+ a_{n+1,p}} -1}{q-1} \prod_{1 \leq i < p} q^{a_{n+1, i}}.
\]
This finishes the proof of the lemma.
\end{proof}

We set, for $a \in \mathbb{Z}$ and $b \in \mathbb N$,
\[
\begin{bmatrix}
a\\
b
\end{bmatrix}
=\prod_{1\leq i\leq b} \frac{v^{2(a-i+1)}-1}{v^{2i}-1}, \quad \text{ and } \quad [a] = \begin{bmatrix}
a\\
1
\end{bmatrix}.
\]
We have the following  multiplication formulas for the algebra $\Sj$, which is an analogue of \cite[Lemma 3.4(a1),(b1)]{BLM}.

\begin{prop} 
\label{BLM3.4}
Suppose that $h\in [1, n]$ and $R\in \mbb N$. 
\begin{enumerate}
\item[(a)]  For $A, B\in \Xi_d$ such that $B-RE_{h, h+1}^{\theta}$ is diagonal and $\ro(A)=\co(B)$, we have 
\begin{equation}
\label{B}
e_B * e_A = \sum_{t} v^{ 2 \sum_{j>u} a_{hj} t_u} \prod_{u=1}^N 
\begin{bmatrix}
a_{h u} + t_u\\ t_u
\end{bmatrix} 
\, e_{A+ \sum_{u=1}^N t_u( E^{\theta}_{hu} - E^{\theta}_{h+1,u})},
\end{equation}
where $t=(t_1,\ldots, t_N)\in \mbb N^N$ with  $\sum_{u=1}^N t_u=R$ such that
\[
\begin{cases}
t_u  \leq a_{h+1,u}, &\text{ if } h<n,\\ 
 t_u+t_{N+1-u} \leq a_{h+1,u}, &\text{ if } h=n.
\end{cases}
\]

\item[(b)] For $A, C\in \Xi_d$ such that $C-RE_{h+1, h}^{\theta}$   is diagonal and $\co ( C) =\ro(A)$, we have 
\begin{equation}
\label{h<n}
e_C * e_A = \sum_{t} v^{ 2 \sum_{j<u} a_{h+1, j} t_u} \prod_{u=1}^N 
\begin{bmatrix}
a_{h+1, u} + t_u\\
 t_u
\end{bmatrix} \, e_{A- \sum_{u=1}^N t_u( E^{\theta}_{hu} - E^{\theta}_{h+1,u})}, \; \mbox{for} \; h< n,
\end{equation}
and for $h=n$,
\begin{equation}
\label{h=n}
\begin{split}
e_C * e_A &= \sum_{t} v^{ 2 \sum_{j<u} a_{n+1, j} t_u}  v^{2\sum_{N+1-j<u<j} t_ut_j + \sum_{u>n+1} t_u(t_u-1)} \prod_{u<n+1} 
\begin{bmatrix}
a_{n+1, u} + t_u\\
 t_u
\end{bmatrix} \\
& \cdot \prod_{u>n+1} 
\begin{bmatrix}
a_{n+1,u} +t_u+t_{N+1-u}\\
 t_u
\end{bmatrix}
 \prod_{i=0}^{t_{n+1}-1} \frac{[a_{n+1,n+1}+1+2i]}{[i+1]} \, e_{A- \sum_{u=1}^N t_u( E^{\theta}_{nu} - E^{\theta}_{n+1,u})},
\end{split}
\end{equation}
where $t=(t_1,\ldots, t_N)\in \mbb N^N$ such that $\sum_{u=1}^N t_u=R$ and $ t_u \leq a_{hu}$.
\end{enumerate}
(Note that the above coefficients are in $\A$ since $a_{n+1,n+1}$ is an odd integer.)
\end{prop}

\begin{proof}
We shall prove (\ref{h=n}) for $h=n$ by induction on $R$ in detail.  It is clear that (\ref{h=n}) holds for $R=1$ by Lemma ~\ref{BLM3.2}(b). 
Let us write $C_R$ instead of $C$ in (b) to indicate the dependence on $R$.
Similarly we let $C_{R+1} \in \Xi_d$ be such that $C_{R+1}-(R+1)E_{h+1, h}^{\theta}$  is diagonal and $\co ( C_{R+1}) =\ro(A)$,
and let $C_1$ be such that $C_{1}- E_{h+1, h}^{\theta}$  is diagonal and $\co ( C_{1}) =\ro(C_R)$. 
We have by Lemma ~\ref{BLM3.2}(b) again that
\begin{equation}
\label{eq:1R}
e_{C_1} * e_{C_R} =[R+1] e_{C_{R+1}}.
\end{equation}
We write $A(t)= A- \sum_{u=1}^N t_u( E^{\theta}_{n,u} - E^{\theta}_{n+1,u})$  and $G_{A, t}$ the coefficient of $A(t)$ in (\ref{h=n}) for $t=(t_1,\ldots, t_N)$.
So we have
\[
e_{C_1} * e_{C_R}  * e_A =\sum_{t, s} G_{A, t} G_{A(t), s}  e_{A(t+s)},
\]
where the sum runs over all $(s, t)$ such that $\sum t_u=R$ and $\sum s_u=1$. 
By a direct computation, for any $r$ such that $\sum r_u=R+1$, we have
\[
\frac{1}{[R+1]} \sum_{t+s = r} G_{A,t} G_{A(t),s}
= \frac{1}{[R+1]} \sum_{s} v^{2 \sum_{j<s' } t_j } [t_{s'}] G_{A, r} = G_{A,r},
\]
where $s'$ is the unique nonzero position in $s$. The formula (\ref{h=n}) follows. 
The proofs of (\ref{B}) and (\ref{h<n}) are similar and will be skipped.
\end{proof}

\subsection{The $\Sj$-action on $\Td$}

Recall we have an $\A$-basis $\{e_{r_1 \cdots r_d} \mid 1\le r_1,\ldots, r_d \le N\}$ for $\Td$, and
there is a bijection (see  \eqref{bij:rr}) between these $d$-tuples $r_1\cdots r_d$ and $\mbf r =(r_1, \ldots, r_D)$ subject to 
$r_c+ r_{D+1-c} =N+1$.   
A (simpler) variant of the proof of Lemma ~\ref{BLM3.2} gives us the following proposition.

\begin{prop}
The left $\Sj$-action on $\Td$ via the convolution product 
\[
\Sj \times \Td \longrightarrow \Td
\]
is given as follows: for $1\le i \le n$, 
\begin{align}
\E_i e_{r_1\cdots r_d} = v^{-\# \{1 \leq k \leq D\vert r_k=i+1\}} 
 \sum_{\substack{\{1 \leq p \leq D \vert r_p=i\}}} v^{2\#\{1 \leq j <p \vert r_j=i+1\}} e_{r'_1\cdots r'_d},
\end{align}
where 
$\mbf r' = (r_1', \dots, r_D')$ for each $p$ with $r_p =i$  satisfies $r_s'=r_s$ (for $s\neq p, D+1-p$), $r'_p =i+1$, and $r'_{D+1-p} =N-i$; 
\begin{align}
\F_i e_{r_1\cdots r_d} = v^{-\#\{ 1 \leq k \leq D\vert r_k=i\}} \sum_{\substack{\{1 \leq p \leq D \vert r_p=i+1\}}} v^{2\#\{p< j \leq D\mid r_j=i\}} e_{r''_1\cdots r''_d},
\end{align}
where $\mbf r'' = (r_1'', \dots, r_D'')$ for each $r_p =i+1$ satisfies  $r''_s=r_s$ (for $s\neq p$, $D+1-p$), $r''_p=i$, and $r''_{D+1-p} =N+1-i$;
and 
$
\mbf d_a^{\pm} e_{r_1\cdots r_d}
=v^{\mp \#\{1\le  j\le D | r_j=a\} }
e_{r_1\cdots r_d}.
$
\end{prop}

\subsection{A standard basis}

Recall that $\mathcal O_A$ is the associated $O(D)$-orbit of $A$. We are interested in computing its dimension 
over the algebraic closure $\overline{\mbb F}_q$. 
We first recall that  the dimension of $O(D)$ is ${D(D-1)}/{2}$. Next we shall compute the stabilizer of a point $(V, V')$ in $\mathcal O_A$. We decompose the vector space 
$\overline{\mbb F}_q^D$ into $\overline{\mbb F}_q^D=\oplus_{1\leq i, j\leq N} Z_{ij}$ such that 
\[
V_a =\oplus_{i\leq a, j\in [1,N]} Z_{ij} 
\quad \mbox{and} \quad
V_b' =\oplus_{i\in [1,N], j\leq b} Z_{ij},\quad \forall a, b \in [1, N].
\]
With respect to  a fixed basis in $Z_{ij}$ 
and the lexicographic order for  the set $\{(i,j)| i, j\in [1,N]\}$, 
the bilinear form $Q$ can be chosen so that its associated matrix  is anti-diagonal and an identity block matrix on each anti-diagonal position. 
The Lie algebra  of the stabilizer $G_{V, V'}$ of the point $(V, V')$ in $O(D)$ is then  the space of all linear maps $x_{(i,j),(k,l)}: Z_{ij}\to Z_{kl}$ satisfying the following conditions.
\begin{itemize}
\item[(a)] $x_{(i,j),(k, l)} = 0$ unless $i\geq k$ and $j\geq l$.
\item[(b)] $x_{(i,j), (k, l)} =-\ ^t\!x_{ (N+1-k, N+1-l), (N+1-i, N+1-j)}, \quad \forall i,j, k, l\in [1,N]$.
\end{itemize}
Note that the  condition (a) is obtained as in  ~\cite[2.1]{BLM}, while the condition (b) is from the choice of $Q$. 
From (b), we see that $x_{(i,j), (k, l)}=-\ ^t\!x_{(i,j), (k, l)}$ if and only if $i+k=N+1$ and $j+l=N+1$.
So  the dimension of the stabilizer $G_{V, V'}$ is
\[
\sum_{\substack{i\geq k, \,  j\geq l \\ i+k < N+1 }} a_{ij} a_{kl} 
+\sum_{\substack{i\geq k, \, j\geq l \\ i+k = N+1, \,  j+l < N+1 }} a_{ij} a_{kl} 
+\sum_{\substack{i\geq n+1, j\geq n+1 }} a_{ij} (a_{ij}-1)/2.  
\]
Summarizing, we have proved the following. 

\begin{lem}
\label{BLM2.2}
The dimension of $\mathcal O_A$, denoted by $d(A)$, is given by
\[
d(A) = 
\sum_{\substack{i< k \;\mrm{\small{or}}\; j<l  \\ i+k<N+1}} a_{ij}a_{kl}
+\sum_{\substack{i< k \;\mrm{\small{or}}\; j<l  \\ i+k=N+1,\, j+l < N+1}} a_{ij}a_{kl} 
+\sum_{\substack{i< n+1\; \mrm{\small{or}} \;  j< n+1  }} a_{ij} (a_{ij}-1)/2.  
\]
\end{lem}
\noindent (Here the condition $\substack{i< k \;\mrm{\small{or}}\; j<l  \\ i+k<N+1}$ means that 
$
i<k , i+k<N+1$ or $ 
i\geq k, j<l, i+k<N+1.$)

Denote by $r(A)$ the dimension of the image of $\mathcal O_A$ under the first projection $X\times X\rightarrow X$.
Note that $r(A) =d(B)$, the dimension of the orbit $\mathcal O_B$, where $B$ is a diagonal matrix such that $b_{ii}=\sum_{j} a_{ij}$. 
By applying Lemma ~\ref{BLM2.2} to the matrix $B$, we have 
\begin{align}
\label{r(A)}
r(A) =\sum_{i<k, i+k<N+1} a_{ij} a_{kl}  +\sum_{i<n+1} a_{ij} a_{il} /2 -\sum_{i<n+1}  a_{ij}/2.
\end{align}
By Lemma ~\ref{BLM2.2} and (\ref{r(A)}),  we have 
\begin{equation}
\label{d-r}
d(A)-r(A) =\sum_{\substack{i>  k, \;  j<l  \\ i+k<N+1}} a_{ij}a_{kl}
+\sum_{\substack{i<n+1\;\mrm{\small{or}}\; j<N+1-l \\j< l}} a_{ij}a_{il}
+\sum_{i\geq n+1> j} a_{ij}(a_{ij}-1)/2.
\end{equation}

We set
\begin{align}
\label{[A]}
[A] = [A]_d = v^{-d(A) + r(A)} e_A, \quad \forall A\in \Xi_d.
\end{align}
(The notation $[A]_d$ will only be used when it is necessary to indicate the dependence on $d$.)
Then $\{[A] \mid A\in \Xi_d\}$ forms an $\A$-basis for $\Sj$, which we call a {\em standard basis} of $\Sj$. 

\begin{rem}
 \label{BLM3.10}
It follows by the same argument as for \cite[Lemma~3.10]{BLM}
that the assignment $[A]\mapsto [\, ^t\!A]$ defines an $\A$-linear 
anti-automorphism on $\Sj$.
\end{rem}

The following is a reformulation of the multiplication formulas for $\Sj$ in Proposition~\ref{BLM3.4}.

\begin{thm}
\label{BLM3.4b}
\begin{enumerate}
\item [(a)] Under the assumptions in Proposition ~\ref{BLM3.4}(a),  we have 
\[
[B] * [A] =\sum_{t} v^{\beta(t)} \prod_{u=1}^N \overline{
\begin{bmatrix}
a_{hu} + t_u\\
 t_u
\end{bmatrix}
} \Big [A+\sum_{1\le u\le N}  t_u (E^{\theta}_{hu}-E^{\theta}_{h+1,u}) \Big ],
\]
where  $t$ is summed over as in Proposition ~\ref{BLM3.4}(a) and 
\begin{equation}\label{eq:lem:betat}
\beta(t) =\sum_{j\leq l} a_{h l} t_j -\sum_{j<l} a_{h+1,l}t_j + \sum_{j< l} t_jt_l
+ \delta_{h,n} \Big(\sum_{\substack{j< l \\ j+l <N+1}} t_j t_l + \sum_{j< n+1}\frac{ t_j(t_j+1)}{2} \Big).
\end{equation}

\item [(b)] Under the assumptions in Proposition ~\ref{BLM3.4}(b), we have
\[
[C] * [A]
=\sum_{t} v^{\beta'(t)} \prod_{u=1}^N \overline{
\begin{bmatrix}
a_{h+1, u} + t_u\\
 t_u
\end{bmatrix}
} \Big [A-\sum_{1\le u\le N}  t_u (E^{\theta}_{hu}-E^{\theta}_{h+1,u}) \Big ], \forall h<n,
\]
where $t$ is summed over as in Proposition ~\ref{BLM3.4}(b) and
\begin{equation}\label{eq:lem:betat'}
\beta'(t)
=\sum_{j\geq l} a_{h+1,l}t_j
-\sum_{j> l} a_{h l} t_j  
+ \sum_{j>l} t_jt_l .
\end{equation}
For $h=n$, we have  
\begin{equation}
 \begin{split}
&[C] * [A]=\sum_{t}v^{\beta''(t)}
\prod_{u<n+1}\overline{ 
\begin{bmatrix}
a_{n+1,u}+t_u+t_{N+1-u}\\
 t_u
\end{bmatrix}
}
\prod_{u> n+1}\overline{
\begin{bmatrix}
 a_{n+1,u}+t_u\\
t_u
\end{bmatrix}}
\\
&\cdot \prod_{i=0}^{t_{n+1}-1}
\frac{\overline{[a_{n+1,n+1}+1+2i]}}
{\overline{[i+1]}}
\Big [A-\sum_{1\le u\le N}  t_u(E^{\theta}_{nu}-
E^{\theta}_{n+1,u})\Big ],
  \end{split}
\end{equation}
where 
\begin{equation}\label{eq:lem:betat''}
\beta''(t)
=\sum_{j\geq l} a_{h+1,l}t_j
-\sum_{j> l} a_{h l} t_j  
+ \sum_{j>l} t_jt_l 
- \sum_{\substack{j< l \\j+l< N+1}} t_j t_l - \sum_{j< n+1} \frac{t_j(t_j-1)}{2} +\frac{R(R-1)}{2}.
\end{equation}
\end{enumerate}
\end{thm}

\begin{proof}
By Proposition ~\ref{BLM3.4}, we have 
\[
\beta(t) =
d(X)-r(X) - (d(A)-r(A)) - (d(B)-r(B)) +
2\sum_{j>u} a_{hj} t_u + 2\sum_u a_{hu} t_u,
\]
where 
\[
X= A+\sum_u t_u (E^{\theta}_{hu}-E^{\theta}_{h+1,u}).
\]
By direct computations, we have $d(B) - r(B) = \sum_{j, u} a_{hj} t_u$.
Then by a lengthy calculation, we have
\begin{align*}
d(X)&-r(X)  - (d(A)-r(A))\\
&=
\sum_{j<l} a_{hl}t_j - \sum_{j< l} a_{h+1, l}t_j +\sum_{j<l} t_jt_l 
+\delta_{h,n} ( \sum_{j<l, j+l<N+1} t_jt_l+ \sum_{j<n+1} \frac{t_j(t_j+1)}{2}).
\end{align*}
So we obtain the formula of $\beta(t)$.
The computations for $\beta'(t)$ and $\beta''(t)$ are similar. 
\end{proof}

\subsection{A monomial basis}

We say that $A\leq B$ if $\mathcal O_A\subseteq \overline{\mathcal O}_B$ over $\overline{\mbb F}_q$.
This defines a partial order $\leq$ in $\Xi_d$.
Following ~\cite[3.5]{BLM},  we define a second  partial order $\preceq$ on $\Xi_d$ by declaring
$A\preceq B$ if and only if 
\begin{align}
\sum_{r\leq i; s\geq j} a_{rs} \leq \sum_{r\leq i; s\geq j} b_{rs}, & \quad \forall  i<j,\label{rs-a}\\
\sum_{r\geq i; s\leq j} a_{rs} \leq \sum_{r\geq i; s\leq j} b_{rs}, &\quad \forall  i>j. \label{rs-b}
\end{align}
Note that (\ref{rs-b}) is redundant, since it can be deduced from (\ref{rs-a}) and $a_{ij}=a_{N+1-i,N+1-j}$.
Since the Bruhat orders on Weyl groups of type $A$ and $B$ are compatible with each other, 
the next result  follows immediately from ~\cite[3.5]{BLM}.  

\begin{lem}
\label{BLM3.6}
If $A\leq B$ for $A, B\in \Xi_d$, then we have $A\preceq B$.
\end{lem}
We introduce a partial order $\sqsubseteq$ on $\Xi_d$ as follows: for $A, A' \in \Xi_d$, we say that 
\begin{equation}
\label{eq:order}
\text{$A' \sqsubseteq A$ 
if and only if $A'\preceq A$,  $\ro(A')=\ro(A)$ and $\co(A') =\co(A)$.}
\end{equation}
 We write $A' \sqsubset A$ if $A' \sqsubseteq A$ and $A' \neq A$.

In the expression  ``$M+$ lower terms'' below, the ``lower terms'' represents a linear combination of elements strictly less than $M$ 
with respect to the partial order $\sqsubseteq$.

\begin{lem}
\label{BLM3.8}
Let $R$ be a positive integer.

\begin{enumerate}
\item[(a)]  Suppose that $A\in \Xi_d$ satisfies  one of  the following conditions: 
\begin{align*}
&a_{hj}=0, \quad \forall j\geq k; \; a_{h+1,k}=R, a_{h+1,j}=0,\quad \forall j>k, \quad \mbox{if} \; h\in [1,n); \;\mbox{or}\\
&a_{nj}=0, \quad \forall j\geq k; \; a_{n+1,k}=R, a_{n+1, j}=0,\quad \forall j> k,\quad \mbox{if} \; h=n, k\in (n+1, N]; \; \mbox{or}\\
&a_{nj}=0, \quad \forall j\geq n+1; \; a_{n+1,n+1}=2R+a, a_{n+1, j}=0,\quad \forall j> n+1,\quad \mbox{if} \; h=n, k=n+1,
\end{align*} 
for some odd integer $a$.
Let $B$ be the matrix such that $B-RE^{\theta}_{h,h+1}$ is diagonal and $\co (B) =\ro(A)$. Then  
\[
[B] * [A] = [M] + \mbox{lower terms}, \quad \mbox{where}\; M= A+ R(E_{h,k}^{\theta}- E_{h+1, k}^{\theta}).
\]

\item [(b)] Suppose that  $A\in \Xi_d$ satisfies one of the following conditions:
\begin{align*}
&a_{hj}=0,\quad \forall j < k,  a_{hk}=R; \;a_{h+1,j} =0,\quad \forall j\leq k,\quad\mbox{if} \; h\in [1,n);\; \mbox{or}\\
 & a_{nj}=0,\quad\forall j< k, a_{nk}=R; \; a_{n+1,j}=0,\quad\forall j\leq k, \quad \mbox{if}\; h=n, k\in[1,n].
\end{align*}
Let $C\in \Xi_d$ be a matrix such that $C-R E^{\theta}_{h+1,h} $ is diagonal and $\co (C ) =\ro(A)$. Then
\[
[C] * [A] = [M] + \mbox{lower terms}, \quad \mbox{where}\; M= A- R(E_{h, k}^{\theta}- E_{h+1,k}^{\theta}).
\]
\end{enumerate}
\end{lem}

\begin{proof}
Observe that $\beta(t)=\beta'(t)=0$ in Proposition ~\ref{BLM3.4} for $t$ such that $t_k=R$ and $0$, otherwise. The lemma follows from 
the same  argument as that of ~\cite[3.8]{BLM} by using again that the partial order $\preceq$ is compatible with the analogous one in ~\cite[3.5]{BLM}.
\end{proof}

\begin{thm}
\label{BLM3.9}
For any $A \in \Xi_d$, we have 
\begin{align}
\begin{split}
\label{3.9a}
\prod_{1\leq j\leq h<i\leq N} [D_{i,h,j}+ a_{ij} E_{h+1, h}^{\theta}] & = [A]+\mbox{lower terms}\\
(\text{this element in } \Sj  \text{ will} & \text{  be denoted by } m_A),
\end{split}
\end{align}
where the product in $(\Sj, *)$ is taken in the following order: $(i, h, j)$ proceeds $(i',h',j')$ if and only if 
$i<i'$, or $i=i'$, $j<j'$, or $i=i'$, $j=j'$, $h>h'$; 
the diagonal matrices $D_{i,h,j} \in \text{Mat}_{N\times N}(\mathbb N)$ are uniquely determined by $\ro(A)$ and $\co(A)$. 
Moreover, the product has $\frac{N(N^2-1)}{6}$ terms.
\end{thm}

\begin{proof}
The proof is a slight modification of the proof of \cite[Proposition~3.9]{BLM} by using Lemma ~\ref{BLM3.8}.
One just needs to be cautious when $h> n$. In this case, $E_{h+1,h}^{\theta}= E_{N-h,N+1-h}^{\theta}$, from which one uses Lemma ~\ref{BLM3.8}(a).

Let us explain the proof in more details for the $n=2$ (i.e., $N=5$) case. We start with a diagonal matrix $D$  such that  $\ro(D)=\co(D)=\co(A)$.
We must fill in each off diagonal entries with the desired number. 
Since all matrices involved satisfy the property that the entries of $(i,j)$ and $(N+1-i, N+1-j)$ are the same, we only need to fill in all the entries below the diagonal. 
We do it by multiplying repetitively from the left a certain matrix.
Of course, we always need to have a leading term in each step. To make this work, we use Lemma ~\ref{BLM3.8} and  fill the entries below the diagonal  from bottom to top and from right to left,  which is exactly the order stated in the theorem.

To this end, we shall first fill in the $(5,4)$-entry . We multiply $D_{5, 4, 4} + a_{54} E^{\theta}_{54}$ by $D$  
where $a_{ij}$ is the $(i,j)$-entry of $A$ and $D_{5,4,4}$ is a diagonal matrix such that  $\co (D_{5, 4, 4} + a_{54} E^{\theta}_{54}) = \ro (D)$. In particular,
\[
[D_{5, 4, 4} + a_{54} E^{\theta}_{54}] * [D]= [D_{5, 4, 4} + a_{54} E^{\theta}_{54}],
\]
with the right hand side of the desired number at the  $(5,4)$-entry.
Next, we fill the  $(5,3)$-entry. We multiply the above product by $[ D_{5, 3, 3} + a_{53} E^{\theta}_{43}]$ from the left. 
Since $E_{43}^{\theta}=E_{23}^{\theta}$,
we obtain by Lemma ~\ref{BLM3.8}(a) that the leading term of the resulting product is of the form:
\[
\begin{pmatrix}
* & a_{12} & 0 & 0 & 0\\
0 & * & a_{13} & 0 & 0 \\
0 & 0 & * & 0 & 0\\
0 & 0 & a_{53} & * & 0\\
0 & 0 & 0 & a_{54} & *
\end{pmatrix}
\]
To bring down $a_{53}$ to the desired position, we multiply the above matrix by $[D_{5, 4, 3} + a_{53} E_{54}^{\theta}]$ from the left. 
By Lemma ~\ref{BLM3.8}(b), the leading term for the resulting product is 
 \[
\begin{pmatrix}
* & a_{12} & a_{13}& 0 & 0\\
0 & * & 0 & 0 & 0 \\
0 & 0 & * & 0 & 0\\
0 & 0 & 0 & * & 0\\
0 & 0 & a_{53} & a_{54} & *
\end{pmatrix}.
\]
Summarizing, in order to put $a_{53}$ in the $(5,3)$-entry, we need the piece
\[
[D_{5, 4, 3} + a_{53} E_{54}^{\theta}] *
[ D_{5, 3, 3} + a_{53} E^{\theta}_{43}].
\]
We can now apply repetitively the above procedure to the rest of the entries in the prescribed order. 
In particular, for $a_{52}$, we need the piece
\[
[D_{5,4,2}+ a_{52} E_{54}^{\theta}] * [D_{5, 3,2} + a_{52} E_{43}^{\theta}] *
[D_{5, 2, 2} + a_{52} E_{32}^{\theta}].
\]
For $a_{51}$, we need the piece
\[
[D_{5, 4, 1}+ a_{51} E_{54}^{\theta}] * [D_{5,3,1} + a_{51} E_{43}^{\theta}] * [ D_{5,2,1}+ a_{51} E_{32}^{\theta} ] * [D_{5,1,1} +a_{51} E_{21}^{\theta}].
\]
For $a_{43}$, we need the piece
\[
[D_{4,3,3}+ a_{43} E_{43}^{\theta}].
\]
For $a_{42}$, we need the piece
\[
[D_{4,3,2} + a_{42} E_{43}^{\theta}] * [ D_{4,2,2} + a_{42} E_{32}^{\theta}].
\]
For $a_{41}$, we need the piece
\[
[D_{4,3,1} + a_{41} E_{43}^{\theta}] * [D_{4,2,1} + a_{41} E_{32}^{\theta}] * [D_{4,1,1} + a_{41} E_{21}^{\theta}].
\]
For $a_{32}$, we need
\[
[D_{3,2,2} + a_{32} E_{32}^{\theta}].
\]
For $a_{31}$, we need
\[
[D_{3,2,1} + a_{31} E_{32}^{\theta} ] * [ D_{3,1,1} + a_{31} E_{21}^{\theta}].
\]
For $a_{21}$, we need 
\[
[D_{2,1,1} + a_{21} E_{21}^{\theta}].
\]
By putting the pieces together, we have the theorem for $n=2$ and the general case follows in the same pattern.

It follows by Remark~\ref{rem:sameorder} below  that the number of the terms in the product (\ref{3.9a}) is half of what is in ~\cite[3.9(a)]{BLM}.
\end{proof}

\begin{rem}
 \label{rem:sameorder}
The ordering of the product \eqref{3.9a} coincides with the one used in \cite[Theorem 13.24
]{DDPW}. Namely, if the superscript $\theta$ in (\ref{3.9a}) is dropped, 
the left hand side becomes exactly  the second half of  a similar product in \cite[Theorem 13.24]{DDPW}
(which is basically  \cite[3.9(a)]{BLM}). 
As explained  in \cite[Notes for \S13.7, pp.589]{DDPW} and the reference therein,  the ordering of the products adopted 
in  \cite[Theorem 13.24]{DDPW} is not the same as the one used in \cite[3.9(a)]{BLM}, but
the resulting products are the same.
\end{rem}

Then by Theorem~\ref{BLM3.9} the transition matrix from $\{m_A \mid A \in \Xi_d\}$ to the standard basis
$\{[A] \mid A \in \Xi_d\}$ is unital triangular, and hence
$\{m_A \mid A \in \Xi_d\}$ forms an $\A$-basis of $\Sj$, which we call a {\em monomial basis} of $\Sj$.

The following lemma follows by definitions. 
\begin{lem}
\label{lem:fed}
For $i \in [1,n], a\in [1,n+1]$,  we have the following identities in $\Sj$:
\[
\F_i=\sum_B [B],
\quad
\E_i =\sum_C [C],
\quad
\mbf d_a = \sum_{D} v^{-D_{aa}} [D]
\]
where the sums are over $B, C, D \in \Xi_d$ such that $B-E^{\theta}_{i,i+1}$,   $C- E^{\theta}_{i+1,i}$, and $D$ are diagonal,  respectively.
\end{lem}

The following corollary of Theorem ~\ref{BLM3.9} is now immediate by applying Lemma~\ref{lem:fed} and a standard Vandermonde-determinant-type argument.

\begin{cor}
The  $\E_i$, $\F_i$, $\mbf d_i^{\pm 1}$ and $\mbf d_{i+1}^{\pm 1}$ for $i\in [1,n]$ generate the $\Qq$-algebra $\QQ \Sj$.
\end{cor}
Bearing in mind the presentations of the standard $v$-Schur algebras of type $A$ (cf. \cite{DDPW}),
we expect a presentation of the $\Qq$-algebra $\QQ\Sj$ with generators given in the above corollary
subject to relations in Proposition~\ref{Type-B-case-I-homomorphism} together with the following 
additional relations:
\begin{align*}
\mbf d_{n+1} \mbf d_n^2\cdots \mbf d_1^2 &= v^{-D},\\
(\mbf d_i -1)(\mbf d_i -v^{-1}) (\mbf d_i - v^{-2}) \cdots (\mbf d_i - v^{-d} )& =0,
\quad \forall
i\in[1,n],\\
(\mbf d_{n+1} - v^{-1}) \cdots (\mbf d_{n+1} - v^{-D}) &=0.
\end{align*}

\subsection{A canonical basis}
\label{subset:CBonS}

Let $\IC_A$, for $A\in \Xi_d$, be  the shifted intersection complex associated with the closure of  the orbit $\mathcal O_A$ 
such that the restriction of $\IC_A$ to $\mathcal O_A$ is the constant sheaf on  $\mathcal O_A$.
Since $\IC_A$ is $O(D)$-equivariant, the stalks of the $i$-th cohomology sheaf of $\IC_A$  at different points in $\mathcal O_{A'}$ (for $A' \in \Xi_d$) are isomorphic.
Let $\mathscr H^i_{\mathcal O_{A'}} (\IC_A)$ denote the stalk of the $i$-th cohomology group of $\IC_A$ at any point in $\mathcal O_{A'}$. 
We  set
\begin{align}
 \label{{A}}
\begin{split}
P_{A', A} &=\sum_{i\in \mbb Z} \dim \mathscr H^i_{\mathcal O_{A'}} (\IC_A) \; v^{i-d(A) +d(A')},
 \\
\{ A\} &= \sum_{A'\leq A} P_{A', A} [A'].
\end{split}
\end{align}
When it is necessary to indicate the dependence on $d$ of $\{A\}$, we will sometimes write $\{A\}_d $ for $\{A\}$. 
By the properties of intersection complexes, we have
\begin{equation}
 \label{eq:v-}
P_{A, A} =1, \qquad P_{A', A} \in v^{-1} \mbb N[v^{-1}] \; \mbox{if}  \; A' < A.
\end{equation}
As in \cite[1.4]{BLM}, we have an anti-linear bar involution $\bar\ : \Sj \to \Sj$ such that 
\[
\bar v= v^{-1}, \qquad \overline{\{A\}} =\{A\}, \quad \forall A\in \Xi_d.
\]
In particular, we have
\[
\overline{[A]} =\sum_{A'\leq A} c_{A', A} [A'], \quad \mbox{where}\; c_{A, A} =1, c_{A', A}\in \mathbb Z[v, v^{-1}].
\]
Then $\mbf B_d^\jmath := \{ \{A\} \mid A\in \Xi_d\}$ forms an $\A$-basis for $\Sj$, called a {\em canonical basis}. 

The approach to the canonical basis for $\Sj$ above follows \cite{BLM}, and it can also be done following an alternative algebraic approach developed by 
Du (see \cite{Du92}). 
 
\subsection{An inner product}
\label{sec:inner Sj}

We set
\[
d_A= d(A)-r(A).
\]
Then, recalling ${}^t A$ denotes the transpose of $A$ we have
\begin{equation}
\label{d_A}
2(d_A- d_{ \ \! ^t\!A}) = \frac{1}{2} \Big (\sum_{i=1}^N \ro(A)_i^2 -\co(A)_i^2 \Big ) - \frac{1}{2} \left ( \ro(A)_{n+1} -\co(A)_{n+1} \right ) .
\end{equation}

Given $A, A' \in \Xi_d$, fix any element $L'$  in $ X_{\co(A)}$ and set
$
X^{L'}_{{}^t\! A}=\{L\mid (L',L)\in \mathcal O_{^t\!A}\}.
$
A standard argument shows that $\# X^{L'}_{{}^t\! A}$ is the specialization at $v=\sqrt{q}$
of a Laurent polynomial $f_{A,A'}(v) \in \Z[v, v^{-1}]$. 
Following McGerty \cite{Mc12}, we define a bilinear form
\[
(\cdot , \cdot)_D: \Sj\times \Sj \longrightarrow  \mathbb Q(v)
\]
by  
\[
(e_A, e_{A'})_D = \delta_{A,A'} v^{2(d_A-d_{ ^t\! A})}  f_{A,A'}(v).
\]
By using (\ref{d_A}) and arguing in exactly the same manner as that of \cite[Proposition~3.2]{Mc12}, we have the following.

\begin{prop}
\label{Mc12-3.2}
$([A] e_{A_1}, e_{A_2})_D= v^{d_A-d_{^t\!A}} (e_{A_1}, [\ \!  ^t\! A] e_{A_2})_D$,
for all $A, A_1, A_2\in \Xi_d$.
\end{prop}

We record the following useful consequence of Proposition ~\ref{Mc12-3.2}. 

\begin{cor}
Let $i \in [1,n]$, $a\in [1,n+1]$, and let $A_1, A_2\in \Xi_d$. Then we have
\begin{enumerate}
\item[(a)] $( \E_i e_{A_1}, e_{A_2})_D = ( e_{A_1}, v {\mbf d}_i {\mbf d}_{i+1}^{-1}   \F_i e_{A_2})_D$;
\item[(b)]  $(\F_i e_{A_1}, e_{A_2})_D = (e_{A_1}, v^{-1} \E_i {\mbf d}_i^{-1} {\mbf d}_{i+1}   e_{A_2})_D$;
\item[($\mbox{c}$)]  $(\mbf d_a e_{A_1},e_{A_2})_D = (e_{A_1}, \mbf d_a e_{A_2})_D$.
\end{enumerate}
\end{cor}

The same argument as for \cite[Lemma~3.5]{Mc12} now gives us the following.

\begin{prop}
 \label{prop:AA0}
\begin{itemize}
\item [(a)] $([A], [A])_D \in 1 + v^{-1} \mbb Z[v^{-1}]$, for any $A\in \Xi_d$.
\item [(b)]  $([A],[A'])_D=0$ if $A\neq A'$.
\end{itemize}
\end{prop}

The next proposition follows from Proposition~\ref{prop:AA0} together with the definition and property of the canonical basis given in \eqref{{A}} and \eqref{eq:v-}.

\begin{prop}
\begin{itemize}
\item [(a)] 
The canonical basis $\mbf B_d^\jmath$ satisfies the almost orthonormality, i.e., 
 $(\{A\}, \{A'\})_D \in \delta_{A,A'} + v^{-1} \mbb Z[v^{-1}]$, for any $A, A'\in \Xi_d$.
\item [(b)]  The signed canonical basis $(- \mbf B_d^\jmath) \cup \mbf B_d^\jmath$ of the $\A$-module $\Sj$
is characterized by the almost orthonormal property (a) together with the bar-invariance.
\end{itemize}
\end{prop}

\section{The algebra $\Kj$ and its identification as a coideal algebra}
 \label{sec:qalg}

In this section we construct an $\A$-algebra $\Kj$ out of $\Sj$ from a stabilization procedure.
We then show that $\Kj$ is  isomorphic to an integral form of a modified coideal algebra $\Ujdot$.
The canonical bases for $\Kj$ and $\Ujdot$ are constructed.
A geometric realization of the $(\Uj, \Hb)$-duality is established. 

\subsection{Stabilization}
\label{sec:stab}

Below we shall imitate \cite[\S 4]{BLM} to develop a stabilization procedure to construct a limit $\A$-algebra $\Kj$ out
of $\Sj$ as $d$ goes to $\infty$. As the constructions are largely the same as {\em loc. cit.}, we will be sketchy.

Recall $N=2n+1$. We introduce the set
\begin{align}
 \label{eq:txi}
 \begin{split}
\txi =\{ A =(a_{ij}) \in \mbox{Mat}_{N\times N} (\mbb Z) \mid  &\; a_{ij} \geq 0\; (i\neq j), 
\\
& a_{n+1,n+1} \in 2\mathbb Z+1, a_{ij} = a_{N+1-i, N+1-j} \;(\forall i,j) \}.
\end{split}
\end{align}
Let $\Kj$ be the free $\A$-module with an $\A$-basis given by the symbols $[A]$, for $A\in \txi$.
Also set 
\begin{equation}
 \label{eq:txid}
\txi^{\text{diag}} =\{A \in \txi\mid A \text{ is diagonal}\},
\qquad 
\Xi :=\cup_{d\ge 0} \Xi_d.
\end{equation}

For $A \in \txi$, setting ${}_{2p} A :=A+2pI$  we have ${}_{2p} {A}\in \Xi$ for integers $p\gg 0$. 
Given matrices $A_1, A_2, \ldots, A_f \in \txi$ (with the same total sum of entries), 
one shows exactly as  in \cite[4.2]{BLM} that there
exists $Z_i \in \txi$ $(1\le i \le m)$ such that 
\begin{equation}
\label{eq:algS2p}
[{}_{2p} A_1] * [{}_{2p} A_2] * \cdots * [{}_{2p} A_f] =\sum_{i=1}^m G_i(v, v^{-2p}) [{}_{2p} Z_i],
\end{equation}
where $G_i(v,v')$ lies in a subring $\mathscr R_1$ of $\Q(v)[v']$ as defined in \cite[4.1]{BLM}.  
This allows us to define a unique structure of associative $\A$-algebra (without unit) on $\Kj$ (with product
denoted by $\cdot$) such that
\begin{equation}
\label{eq:algK}
[A_1] \cdot [A_2] \cdot \ldots \cdot  [A_f] =\sum_{i=1}^m G_i(v, 1) [ Z_i].
\end{equation}

From the above stabilization procedure, the multiplication formula in Theorem~\ref{BLM3.4b}(a) leads to the following.
For any $A, B\in \txi$ such that $\ro(A) =\co (B) $ and $B-RE_{h, h+1}^{\theta}$ is diagonal for some $h\in [1,n]$, we have
\begin{align}
\label{BLM4.6a}
[B] \cdot [A] =\sum_{t} v^{\beta(t)} \prod_{u=1}^N \overline{
\begin{bmatrix}
a_{hu} + t_u\\
 t_u
\end{bmatrix}
} \left  [A+\sum_u t_u (E^{\theta}_{hu}-E^{\theta}_{h+1,u}) \right ],
\end{align}
where  $\beta(t)$ is defined in \eqref{eq:lem:betat} and $t=(t_1,\ldots, t_N)\in \mbb N^N$ 
such that $\sum_{u=1}^N t_u=R$ and $ t_u  \leq a_{h+1,u}$ for $u\neq h+1$ and $h<n$
or $t_u+t_{N+1-u} \leq a_{n+1,u}$ for $u\neq n+1$ and $h=n$.

Similarly, we obtain the following multiplication formula from the one in Theorem~\ref{BLM3.4b}(b)
via the above stabilization procedure.
For any   $A, C\in \txi$ such that $\ro(A) = \co(C)$ and $C-RE_{h+1,h}^{\theta}$ is diagonal for some $h\in [1,n-1]$, we have
\begin{align}
\label{BLM4.6b}
[C] \cdot [A]
=\sum_{t} v^{\beta'(t)} \prod_{u=1}^N \overline{
\begin{bmatrix}
a_{h+1, u} + t_u\\
 t_u
\end{bmatrix}
} \left  [A-\sum_u t_u (E^{\theta}_{hu}-E^{\theta}_{h+1,u}) \right ],
\end{align}
where $\beta'(t)$ is defined in \eqref{eq:lem:betat'} and $t=(t_1,\ldots, t_N)\in \mbb N^N$ 
such that $\sum_{u=1}^N t_u=R$ and $0\leq t_u  \leq a_{h,u}$ for $u\neq h$.
For $h=n$, we have 
\begin{equation}
 \label{BLM4.6b'}
 \begin{split}
&[C]\cdot [A]=\sum_{t}v^{\beta''(t)}
\prod_{u<n+1}\overline{ 
\begin{bmatrix}
a_{n+1,u}+t_u+t_{N+1-u}\\
 t_u
\end{bmatrix}
}
\prod_{u> n+1}\overline{
\begin{bmatrix}
a_{n+1,u}+t_u\\
t_u
\end{bmatrix}}
\\
&\qquad\qquad\qquad\qquad \cdot \prod_{i=0}^{t_{n+1}-1}
\frac{\overline{[a_{n+1,n+1}+1+2i]}}
{\overline{[i+1]}}
\left[A-\sum_{u=1}^Nt_u(E^{\theta}_{nu}-
E^{\theta}_{n+1,u})\right],
  \end{split}
\end{equation}
where $\beta''(t)$ is defined in \eqref{eq:lem:betat''} and 
$t=(t_1,\ldots, t_N)\in \mbb N^N$ such that $\sum_{u=1}^N t_u=R$ and $0\leq t_u  \leq a_{n,u}$ for $u\neq n$.

We extend the partial order $\sqsubseteq$ on $\Xi_d$ to $\txi$  by the same definition \eqref{eq:order} with $A, A' \in \txi$. 
Now it follows by Theorem ~\ref{BLM3.9} that
\begin{align}
\begin{split}
\label{BLM4.6c}
\prod_{1\leq j\leq h<i\leq N} [D_{i,h,j}+ a_{ij} E_{h+1, h}^{\theta}] &= [A]+\sum_{A' \sqsubset  A}  \gamma_{A', A}[A'], 
\quad \text{ for } \gamma_{A',A}\in \A
\\
(\text{this element in } & \Kj \text{ will be denoted by } \tM_A),
\end{split}
\end{align}
where the product is taken in the same order as specified in Theorem ~\ref{BLM3.9} and the notation $D_{i,h,j}$ can be found therein. 
Then $\{ \tM_A \mid A \in \txi\}$ forms an $\A$-basis (called the {\em monomial basis}) of $\Kj$. 
Summarizing, we have established the following.

\begin{prop}
\label{BLM4.5}
The formula \eqref{eq:algK} endows $\Kj$ a structure of associative $\A$-algebra (without unit).
Moreover, this $\A$-algebra structure on $\Kj$ is characterized by the multiplication formulas \eqref{BLM4.6a}--\eqref{BLM4.6b'}.   
\end{prop}

\subsection{A canonical basis of $\Kj$}
\label{sec:CBKj}

Just as in \cite[4.3,  4.5(b)]{BLM}, we have an anti-linear bar involution on $\Kj$ induced from the ones on $\Sj$ (as $d$ goes to $\infty$). 
More explicitly, one shows the following stabilization phenomenon of the bar involutions on $\Sj$:
\[
\overline{ [{}_{2p}A]} =\sum_{i=1}^m H_i(v, v^{-2p})  [{}_{2p}T_i]   \quad \forall p \gg 0. 
\]
where $H_i(v, v')\in \mathbb Q(v)[v', v'^{-1}]$. 
Following {\em loc. cit.}, we obtain an anti-linear bar involution $\ \bar\ : \Kj \to \Kj$ defined by 
\[
\overline {[A]} =\sum_{i=1}^m H_i(v,1) [T_i].
\]
In particular, we have 
\[
\overline{[A]} = [A] +\sum_{A': A'\sqsubseteq A, A'\neq A} \tau_{A',A} [A'], \quad \text{ for }  \tau_{A', A}\in\A.
\]
By a standard argument (see, e.g., ~\cite[24.2.1]{Lu93}), we have the following.

\begin{thm}
\label{BLM4.7}
There exists a unique $\A$-basis $\Bj = \{ \{A\} \mid A \in \tilde{\Xi}\}$ for $\Kj$ such that 
\begin{align*}
\overline{\{A\}} &= \{A\}, \\
\{A\} &=[A] +\sum_{A'\sqsubset A}  \pi_{A',A} [A'], \quad \text{ for }  \pi_{A', A} \in v^{-1} \mbb Z[v^{-1}].
\end{align*}
 The basis $\Bj$ is called  the $canonical$ $basis$ of $\Kj$.
\end{thm}


\subsection{Definition of $\Ujdot$}
\label{sec:Ujdot}

The algebra $\U^{\jmath}$ is defined to be the associative algebra over $\mathbb Q(v)$ 
generated by $\ibe{i}$, $\ibff{i}$, $\ibd{a}$, $\ibd{a}^{-1}$, $i = 1, 2, \dots, n$, $a = 1, 2, \dots, n+1$
subject to the following relations, for $i, j = 1, 2, \dots, n$, $a,b = 1, 2, \dots, n+1$:
\begin{eqnarray}
\label{align:star} 
\left\{
\begin{array}{rll}
 \ibd{a} \ibd{a}^{-1} &= \ibd{a}^{-1} \ibd{a} =1, & \\
 \ibd{a} \ibd{b} &=  \ibd{b}  \ibd{a}, & \\
 \ibd{a} \ibe{j} \ibd{a}^{-1} &= v^{-\delta_{a, j+1}-\delta_{N+1-a, j+1} +\delta_{a, j} } \ibe{j}, &  \\
 \ibd{a} \ibff{j} \ibd{a}^{-1} &= v^{-\delta_{a, j} + \delta_{a, j+1} + \delta_{N+1-a, j+1} } \ibff{j}, &  \\
 \ibe{i} \ibff{j} -\ibff{j} \ibe{i} &= \delta_{i,j} \frac{\ibd{i}\ibd{i+1}^{-1}
 -\ibd{i}^{-1}\ibd{i+1}}{v-v^{-1}},         &\text{if } i, j \neq n,  \\
 \ibe{i}^2 \ibe{j} +\ibe{j} \ibe{i}^2 &= \llbracket 2\rrbracket  \ibe{i} \ibe{j} \ibe{i},   &\text{if }  |i-j|=1, \\
\ibff{i}^2 \ibff{j} +\ibff{j} \ibff{i}^2 &= \llbracket 2\rrbracket  \ibff{i} \ibff{j} \ibff{i}, &\text{if } |i-j|=1, \\
 \ibe{i} \ibe{j} &= \ibe{j} \ibe{i},     &\text{if } |i-j|>1,  \\
 \ibff{i} \ibff{j}  &=\ibff{j}  \ibff{i},  &\text{if } |i-j|>1, \\
  \ibff{n}^2\ibe{n} + \ibe{n}\ibff{n}^2
    & = \llbracket 2\rrbracket \Big(\ibff{n}\ibe{n}\ibff{n}- ( v \ibd{n}\ibd{n+1}^{-1}+
    v^{-1}\ibd{n}^{-1}\ibd{n+1} )\ibff{n} \Big),  & \\
 \ibe{n}^2\ibff{n} + \ibff{n}\ibe{n}^2
   &=  \llbracket 2\rrbracket \Big(\ibe{n}\ibff{n}\ibe{n}- \ibe{n} ( v\ibd{n}\ibd{n+1}^{-1} + v^{-1}\ibd{n}^{-1}\ibd{n+1}  )  \Big). &
   \end{array}
   \right.
\end{eqnarray}

We also write $\ibe{i} = \ibff{N+1-i}$ and $\ibff{i} = \ibe{N+1-i}$ for $ n+1< i \leq N$. 
Denote by $^{0}\Uj$ the $\Qq$-subalgebra of $\Uj$ generated by $\ibd{a}^{\pm 1}$ for all $a = 1, 2, \dots , n+1$.

\begin{rem}
\label{rem:sln}
The $\Qq$-subalgebra of $\Uj$ generated by $\ibe{i}$, $\ibff{i}$, $\ibd{i}^{\pm 1}$, $\ibd{i+1}^{\pm 1}$, $i = 1, 2, \dots, n-1$
is naturally isomorphic to the quantum group $\U(\gl(n))$.
The algebra $\Uj$ and the quantum group $\U(\gl(N))$ form the quantum symmetric pair $(\U(\gl(N)), \Uj)$.
This is a variant of  the quantum symmetric pair associated with the quantum group $\U(\mathfrak{sl}(N))$;
see \cite{Le02, K14} and \cite[Section~6]{BW}. The algebra $\Uj$ (and $\Ui$ in later sections) also appeared independently in \cite{ES13}. 
The relation between the algebra $\Uj$ defined here and the algebra of the same notation defined in \cite[Secion~6]{BW} 
is just the usual $\mathfrak{sl}$ vs $\mathfrak{gl}$ relation and
can be described as follows: 
the generators $\ibe{i}$ and $\ibff{i}$ match the generators with $\ibe{\alpha_{n-i +  \frac{1}{2}}}$ and $\ibff{\alpha_{n-i +  \frac{1}{2}}}$ respectively, 
while 
\[
\ibd{a}\ibd{a +1}^{-1}=
\begin{cases}
k^{}_{\alpha_{n-a  +  \frac{1}{2}}}, &\text{ if } 1\le a < n ;\\
vk_{\alpha_{ \frac{1}{2}}}, &\text{ if } a = n.
\end{cases}
\]
The convention on the generators $\ibd{a}$ (and in particular on $d_{n+1}$) is made to better match the geometric construction.
\end{rem}

\begin{lem}\label{lem:barUj}\text{(cf. \cite[Lemma~6.1]{BW})}
The algebra $\Uj$ has an anti-linear bar involution, denoted by 
$\bar{\phantom{x}}$, such that $\overline{d_a} = d^{-1}_a$, $\overline{\ibe{i}} = \ibe{i}$, and $\overline{\ibff{i}}= \ibff{i}$ for $i = 1,  \dots, n$ and $a =1,  \dots, n+1$.
\end{lem}

Denote by $\U(\gl(N))$ the quantum general linear Lie algebra of rank $N$, which is $\Qq$-algebra generated by $E_{i}$, $F_i$, $K^{\pm 1}_i$, and
$K^{\pm 1}_{i+1}$, for $i \in [1, N-1]$
subject to a standard set of relations 
which can be found as part of the relations \eqref{align:star} in different notations.
(Recall from Remark~\ref{rem:sln} that the quantum group $\U(\gl(n))$ is naturally a subalgebra of $\Uj$ with the corresponding subset of relations
from \eqref{align:star}.)

\begin{prop}\label{prop:embedding}
There is an injective $\Qq$-algebra homomorphism $\jmath : \U^{\jmath} \rightarrow \U(\gl(N))$ given by, for all $ i = 1, \dots, n$,
\begin{align*}
  	\ibd{i} & \mapsto K^{-1}_i K^{-1}_{N+1-i}, &
	\ibe{i} &\mapsto  F_i  + K^{-1}_{i}K_{i+1}E_{N-i},\\
	\ibd{n+1} & \mapsto vK^{-2}_{n+1}, &
	\ibff{i} &\mapsto  E_i K^{-1}_{N-i}  K_{N+1-i}+ F_{N-i}.
	\end{align*}
	
\end{prop}

\begin{proof}
This is a $\mathfrak{gl}(N)$-variant of \cite[Proposition~6.2]{BW}; also see \cite{K14, ES13}. 
\end{proof}

Following \cite{BLM} and \cite[\S23.1]{Lu93}, we shall similarly define the modified  
quantum algebra $\Ujdot$ from  $\U^{\jmath}$, where the unit of  $\U^{\jmath}$ is replaced by a collection of orthogonal idempotents. 
We shall denote by $\lambda = \text{diag} (\lambda_1, \dots, \lambda_N)$ a matrix in $\txid$.
For $\lambda, \lambda' \in \txid$, we set
\[
{}_{\lambda} \U^{\jmath}_{\lambda'} = \U^{\jmath}/ \big(\sum^{n+1}_{
a=1} (\ibd{a}-v^{-\lambda_a}) \U^{\jmath}  + \sum^{n+1}_{a=1} \U^{\jmath}( \ibd{a} - v^{-\lambda'_a}) \big).
\]

Let $\pi_{\lambda, \lambda'} : \U^{\jmath} \rightarrow {_{\lambda} \U^{\jmath}_{\lambda'}} $ be the canonical projection. Set 
\[
\Ujdot = \bigoplus_{{\lambda}, {\lambda'} \in \txid} {_{\lambda} \U^{\jmath}_{\lambda'} }.
\]
Let $D_{\lambda} : = \pi_{\lambda,\lambda}(1)$. 
Following \cite[23.1]{Lu93}, $\Ujdot$ is naturally an associative $\Qq$-algebra containing $D_{\lambda}$ as orthogonal idempotents,
and $\Ujdot$ is naturally a $\U^{\jmath}$-bimodule. 
In particular, we have 
\[
\Ujdot = \sum_{{\lambda} \in \txid} \Uj D_{\lambda} = \sum_{{\lambda} \in \txid} D_{\lambda}\Uj.
\]
By construction, the following relations hold: 
\begin{align*}
D_{\lambda} D_{\lambda'} &= \delta_{\lambda, \lambda'}D_{\lambda}, \displaybreak[0]\\
\ibd{a} D_{\lambda}  &= D_{\lambda} \ibd{a} = v^{-\lambda_a} D_{\lambda}.\displaybreak[0]
\end{align*}

%
%
%
%
%
%
%

For any $I = (i_1, i_2, \dots, i_j)$ with $i_k \in \{1, 2, \dots, n, n+2, \dots, N\}$,
 we set $\mathcal E_I = \ibe{i_1}\ibe{i_2}\dots\ibe{i_j}$, where $\ibe{\emptyset} = 1$. 
 Following \cite[Proposition~6.2]{K14}, there exists a collection of such indices $I$, denoted by $\mathbb{I}$, 
 such that $\{ \mathcal E_I  \mid I \in \mathbb{I}\}$ forms a basis of the free $^{0}\Uj$-module $\U^{\jmath}$. 
 Therefore the elements $\mathcal E_I  D_{\lambda}$ ($I \in \mathbb{I}, {\lambda} \in \txid$) form a basis of the $\Qq$-vector space $\Ujdot$.

\subsection{A presentation of $\Ujdot$}
\label{sec:presentUjdot}

Given  $\lambda\in \txid$, we introduce the short-hand notation
$\la -\alpha_i =\la + E^{\theta}_{i+1,i+1} -E^{\theta}_{i,i}$ and $\la +\alpha_i =\la- E^{\theta}_{i+1,i+1} +E^{\theta}_{i,i}$, for $1\le i \le n$. 
We define $\mathbf{A}$ to be the $\Qq$-algebra generated by the symbols $D_{\lambda}$, $\ibe{i}D_{\lambda}$, 
$D_{\lambda}\ibe{i}$, $\ibff{i}D_{\lambda}$ and $D_{\lambda}\ibff{i}$, for $i = 1, \dots, n$ and ${\lambda} \in \txid$, 
subject to the following relations \eqref{align:doublestar}, for $i,j = 1, \dots, n$ and $\la, \la' \in \txid$:
\begin{eqnarray}
\label{align:doublestar}  
\left\{
 \begin{array}{rll}
x D_{\lambda}  D_{\lambda'} x' &= \delta_{\la, \la'} x D_{\lambda}  x',
 \quad \text{ for } x, x'\in \{1, \ibe{i}, \ibe{j}, \ibff{i}, \ibff{j}\}, &  \\
\ibe{i} D_{\lambda} &= D_{\lambda -\alpha_i}   \ibe{i}, & \\ 
\ibff{i} D_{\lambda} &=  D_{\lambda +\alpha_i}  \ibff{i},  & \\
 \ibe{i}D_{\lambda}\ibff{j} &=  \ibff{j}  D_{\lambda-\alpha_i -\alpha_j}   \ibe{i}, &\text{ if } i \neq j,  \\
 \ibe{i} D_{\lambda} \ibff{i} &=\ibff{i} D_{\lambda -2\alpha_i}   \ibe{i}  + \llbracket \lambda_{i+1} - \lambda_{i} \rrbracket  D_{\lambda-\alpha_i},    & \text{ if } i \neq n,
 \\
(\ibe{i}^2 \ibe{j} +\ibe{j} \ibe{i}^2) D_{\lambda} &= \llbracket 2\rrbracket \ibe{i} \ibe{j} \ibe{i} D_{\lambda}, & \text{ if }  |i-j|=1, \\
( \ibff{i}^2 \ibff{j} +\ibff{j}\ibff{i}^2) D_{\lambda} &= \llbracket 2\rrbracket \ibff{i} \ibff{j} \ibff{i}D_{\lambda} , &\text{ if } |i-j|=1,\\
 \ibe{i} \ibe{j} D_{\lambda} &= \ibe{j} \ibe{i} D_{\lambda} ,   & \text{ if } |i-j|>1, \\
\ibff{i} \ibff{j} D_{\lambda}  &=\ibff{j} \ibff{i} D_{\lambda} ,   & \text{ if } |i-j|>1, \\
 (\ibff{n}^2 \ibe{n} - \llbracket 2\rrbracket  \ibff{n}\ibe{n} \ibff{n} + \ibe{n} \ibff{n}^2) D_{\lambda} 
 &= -\llbracket 2\rrbracket  \Big(v^{\lambda_{n+1} - \lambda_{n} -2} + v^{\lambda_n - \lambda_{n+1}+2} \Big)\ibff{n} D_{\lambda}, & \\
(\ibe{n}^2 \ibff{n}  - \llbracket 2\rrbracket  \ibe{n}\ibff{n}\ibe{n} +\ibff{n}\ibe{n}^2) D_{\lambda}
   &= - \llbracket 2\rrbracket  \Big(v^{\lambda_{n+1} - \lambda_{n}+1} + v^{\lambda_n - \lambda_{n+1}-1}\Big) \ibe{n} D_{\lambda}. &
   \end{array}
    \right.
\end{eqnarray}
To simplify the notation, we shall write $x_1D_{\lambda^1} x_2 D_{\lambda^2} \cdots x_l D_{\lambda^l} = x_1 x_2 \cdots x_l D_{\lambda^l}$, 
if the product is not zero; in this case such $\lambda^1$, $\lambda^2$, $\dots$, $\lambda^{l-1}$ are unique.

\begin{prop}\label{prop:presentationUjdot}
We have an isomorphism of $\Qq$-algebras $ \Ujdot \cong \mbf A$ by identifying the generators in the same notation.
That is, the relations \eqref{align:doublestar} 
are the defining relations of the $\Qq$-algebra $\Ujdot$.
\end{prop}

\begin{proof}
By definitions of $\Ujdot$ and $\U^\jmath$ with relations \eqref{align:star} in \S\ref{sec:Ujdot}, 
we see that $\Ujdot$ satisfies the same relations  \eqref{align:doublestar} as for $\mbf A$.
Hence there is a  surjective algebra homomorphism  $\mathbf{A} \rightarrow \Ujdot$, sending generators to generators in the same notation. 
Following the presentation of the algebra $\mathbf{A}$ in \eqref{align:doublestar}, we see that the algebra $\mathbf{A}$ has a natural $\U^{\jmath}$-bimodule structure, 
where the actions of $\ibe{i}$, $\ibff{i} \in \U^{\jmath}$ are defined in the obvious way,
and the action of $\ibd{a}$ on the idempotents $D_{\lambda}$ is given by 
$
D_{\lambda} \ibd{a}= \ibd{a} D_{\lambda} = v^{-\lambda_a}D_{\lambda}.
$
As a left or right $\Uj$-module, $\mathbf{A}$ is generated by the idempotents $D_{\lambda}$, for $\lambda \in \txid$. Hence we have
$
\mbf A = \sum_{\lambda \in \txid} \U^{\jmath} D_{\lambda}.
$

Recall that $\{ \mathcal E_I  \mid I \in \mathbb{I}\}$ forms a basis of the free $^{0}\Uj$-module $\U^{\jmath}$. Therefore we have 
$
\mbf A = \sum_{I \in \mathbb{I}, \lambda \in \txid} \Qq \mathcal E_I  D_{\lambda}.
$
Since  $\{\mathcal E_I D_{\lambda} \mid I \in \mathbb{I},  {\lambda} \in \txid\}$ forms a $\Qq$-basis of $\Ujdot$, these elements in the same notation in $\mbf A$
must be linearly independent, and hence the homomorphism
$\mbf A \rightarrow \Ujdot$ is an isomorphism.
\end{proof}

\subsection{Isomorphism  ${}_\A\Ujdot \cong \Kj$}
\label{sec:isom j}

The following theorem provides a geometric realization of $\Ujdot$ thanks to the geometric nature of $\Kj$.

\begin{thm}\label{thm:UisoK}
We have an isomorphism of $\Qq$-algebras $\aleph : \Ujdot \rightarrow  \QQ\Kj$ which sends
\begin{align*}
%
\ibe{i} D_{\lambda} \mapsto [ D_\lambda - E^{\theta}_{i,i} + E^{\theta}_{i+1, i} ] ,\quad
\ibff{i} D_{\lambda} &\mapsto [ D_\lambda - E^{\theta}_{i+1,i+1} + E^{\theta}_{i,i+1} ],
\end{align*}
and $D_{\lambda} \mapsto [D_\lambda],$ for all $i = 1, \dots, n$, and  $\la \in \txid$.
\end{thm}

\begin{proof}
Via a direct computation we can check using  the multiplication formulas \eqref{BLM4.6a}--\eqref{BLM4.6b'}
that the relations  \eqref{align:doublestar} for $\Ujdot$ are satisfied by
the images of $D_\la, \ibe{i} D_\la, \ibff{i}D_\la$ as specified in the lemma.
Since the relations \eqref{align:doublestar}
are defining relations for $\Ujdot$ by Proposition~ \ref{prop:presentationUjdot}, we conclude that $\aleph$ is an algebra homomorphism.  

It remains to show $\aleph$ is a linear isomorphism. Set $\varepsilon_{n+1} =(0,\ldots, 0,1,0,\ldots,0) \in \mathbb  N^{2n+1}$, where $1$ is in the $(n+1)$st position.
Note that ${\varepsilon_{n+1}}$, when regarded as in $\txid$, is simply the diagonal matrix $E_{n+1,n+1}$.
Also set
\begin{align*}
\tilde{\Theta} &=
\{ A =(a_{ij}) \in \mbox{Mat}_{N\times N} (\mbb Z) \mid a_{ij} \geq 0\; (i\neq j)\},
 \\
\tilde{\Theta}^- &=\{ A\in \tilde{\Theta} \mid a_{ij} =0\; (i<j), \co(A)=0 \},
 \\
\txi^- &=\{ A\in \txi \mid   \co(A)= \varepsilon_{n+1} \}.
\end{align*}
The diagonal entries of a matrix $A'\in \tilde{\Theta}^-$ (respectively, $A \in \txi^-$)
are completely determined by its strictly lower triangular entries.
Hence, there is a natural bijection $\tilde{\Theta}^- \longleftrightarrow \txi^-$, which sends $A'\in \tilde{\Theta}^-$ to $A \in \txi^-$
such that the strictly lower triangular parts of $A$ and $A'$ are identical. 

Denote by $\K$ the $\A$-algebra in the type $A$ stabilization of \cite{BLM}, and $\QQ\K =\Qq \otimes_\A \K$.
The quantum group $\U = \U(\mathfrak{gl}(N))$ has a triangular decomposition $\U =\U^-\U^0\U^+$. 
Denote by $\Udot =\dot\U(\mathfrak{gl}(N))$  the modified quantum group of $\mathfrak{gl}(N)$ with idempotents ${\bf 1}_\la$ and
denote its $\A$-form by ${}_\A\Udot$.
It was shown in \cite{BLM} that there exists a $\Qq$-algebra isomorphism
$\aleph^a: \Udot \rightarrow \QQ\K$ (entirely parallel to the homomorphism 
$\aleph: \Ujdot \rightarrow \QQ\Kj$ above). Here and below we will add superscript $a$ to indicate type $A$ and distinguish from the notations already
used in the type $B$ setting of this paper.
Recall $\K$ has a monomial $\A$-basis $\{{}^a\tM_{A'}  \mid A'   \in \tilde{\Theta}\}$ which is given in \cite[4.6(c)]{BLM} (without such notation or terminology),
entirely parallel to the monomial basis for $\Kj$ which we defined in \S\ref{sec:stab}. 
Via the isomorphism $\aleph^a$,  
the monomial $\A$-basis $\{{}^a\tM_{A'}  \mid A'   \in \tilde{\Theta}^-\}$ for the $\A$-submodule $\aleph^a ({}_\A \U^-{\bf 1}_0)$ of $\K$
corresponds to a monomial  $\A$-basis $\{{}^u\tM_{A'}  \mid A'   \in \tilde{\Theta}^-\}$ for the $\A$-submodule ${}_\A\U^-{\bf 1}_0$ of ${}_\A\Udot$, 
where $ \aleph^a ( {}^u\tM_{A'}) = {}^a\tM_{A'}$.

Note that we may regard that $e_i, f_i \in \Uj$ have ``leading terms" $F_i, F_{N-i}$ by Proposition~\ref {prop:embedding}. 
By \cite{Le02} (see \cite[Proposition 6.2]{K14}), 
replacing the leading terms $F_i$, $F_{N-i}$ by $e_i$ and $f_i$ respectively,
we obtain a monomial $\Qq$-basis $\{\tilde\tM_{A}  \mid A  \in \txi^-\}$ for $\Ujdot D_{\varepsilon_{n+1}}$
from the monomial basis $\{{}^u\tM_{A'}  \mid A'   \in \tilde{\Theta}^-\}$ for $\U^-{\bf 1}_0$. 
The homomorphism $\aleph  : \Ujdot \rightarrow \QQ\Kj$ restricts to 
a $\Qq$-linear map $\aleph |_{\varepsilon_{n+1}}: \Ujdot  D_{\varepsilon_{n+1}}\rightarrow \QQ\Kj [D_{\varepsilon_{n+1}}]$, which sends
$\tilde\tM_{A}$ to $\tM_{A}$ for $A\in \txi^-$  (using Remark~\ref{rem:sameorder}
and the definition  \eqref{BLM4.6c} of a monomial basis element as a product of generators). 
Hence $\aleph |_{\varepsilon_{n+1}}$ is a $\Qq$-linear isomorphism. 
This leads to a $\Qq$-linear isomorphism
 $\aleph |_\la : \Ujdot  D_\la\rightarrow \QQ\Kj [D_\la]$ (which is a restriction of $\aleph$),
for any $\la \in \txid$, via the following commutative diagram:
\begin{eqnarray*}
\begin{CD}
\Ujdot  D_{\varepsilon_{n+1}}
 @>\aleph |_{\varepsilon_{n+1}}>>
\QQ\Kj [D_{\varepsilon_{n+1}}]  \\
 @V\sharp_\la VV @V \sharp_\la VV \\
\Ujdot  D_\la
 @>\aleph|_\la>>
\QQ\Kj [D_\la]
  \end{CD}
 \end{eqnarray*}
 Here $\sharp_\la : \Kj [D_{\varepsilon_{n+1}}] \rightarrow \Kj [D_\la]$ is a $\Qq$-linear isomorphism which
 sends a monomial basis element $\tM_{A +\varepsilon_{n+1}}$ to $\tM_{A+\la}$, and
 $\sharp_\la : \Ujdot  D_{\varepsilon_{n+1}} \rightarrow \Ujdot D_\la$ is defined accordingly. 
 
 Putting $\aleph|_\la$  together, we have shown that $\aleph  : \Ujdot \rightarrow \QQ\Kj$ is an isomorphism.
\end{proof}

The bar involution on $\Uj$ (given in Lemma~\ref{lem:barUj}) induces a compatible bar involution on $\Ujdot$, denoted also by $\bar{\phantom{x} }$, 
which fixes all 
the generators $D_\la, \ibe{i} D_\la, \ibff{i} D_\la$. 

\begin{cor}
The homomorphism $\aleph$ intertwines the bar involutions on $\Ujdot$ and  on $\QQ\Kj$,  i.e., $\aleph(\bar{u}) = \overline{\aleph(u)}$, for $u\in \Uj$.
\end{cor}

\begin{proof}
The corollary follows by checking when $u$ is a generator of $\Ujdot$.
\end{proof}

We define an $\A$-subalgebra of $\Ujdot$ by ${_{\A}\Ujdot} : = \aleph^{-1}(\Kj)$. Clearly we have ${_{\A}\Ujdot} \otimes_{\A} \Qq = \Ujdot$.

\begin{cor}
The integral form $_{\A}\Ujdot$ is a free $\A$-submodule of $\Ujdot$ and it is stable under the bar involution.
\end{cor}

The isomorphism $\aleph: {}_\A \Ujdot \rightarrow \Kj$ allows us to transport 
 the canonical basis $\Bj$ for $\Kj$ to a canonical basis for ${}_\A \Ujdot$ (and for $\Ujdot$).
Introduce the divided powers  $\ibe{i}^{(r)} = \ibe{i}^r/ \llbracket r \rrbracket!$ and $\ibff{i}^{(r)} = \ibff{i}^r/ \llbracket r \rrbracket!$, for $r\ge 1$,
where $\llbracket r\rrbracket! =\llbracket r\rrbracket \cdots \llbracket 1\rrbracket$.

\begin{prop}
\label{prop:powergen}
The isomorphism $\aleph  : \Ujdot \rightarrow \QQ\Kj$ satisfies, for $r \ge 1$,
\[
\aleph(\ibe{i}^{(r)}D_{\lambda}) = [D_{\lambda} - rE^{\theta}_{i,i} + rE^{\theta}_{i+1,i}] \quad \text{ and } \quad  \aleph(\ibff{i}^{(r)}D_{\lambda}) = [D_{\lambda} - rE^{\theta}_{i+1,i+1} + rE^{\theta}_{i,i+1}].
\]
Moreover, the $\A$-algebra $_{\A}\Ujdot$ is generated by $\ibe{i}^{(r)}D_{\lambda}$ and $\ibff{i}^{(r)}D_{\lambda}$ for various $i$, $r$ and $\lambda$. 
\end{prop}

\begin{proof}
We shall only show the first identity, as the second one is entirely similar. 
Following the multiplication formula \eqref{BLM4.6a},  we have  
\[
[D_{\lambda'} - E^{\theta}_{i,i} + E^{\theta}_{i+1,i}] [D_{\lambda''} - (r-1)E^{\theta}_{i,i} + (r-1)E^{\theta}_{i+1,i}]  =  [r][D_{\lambda} - rE^{\theta}_{i,i} + rE^{\theta}_{i+1,i}],
\]
where $\text{ro}(D_{\lambda'} - E^{\theta}_{i,i} + E^{\theta}_{i+1,i}) = \text{ro}(D_{\lambda} - rE^{\theta}_{i,i} + rE^{\theta}_{i+1,i})$ and $\text{co}(D_{\lambda''} - (r-1)E^{\theta}_{i,i} + (r-1)E^{\theta}_{i+1,i}) = \text{co}(D_{\lambda} - rE^{\theta}_{i,i} + rE^{\theta}_{i+1,i})$.
The first statement follows. 

The second statement follows from Theorem~\ref{BLM3.9} and the definition of $_{\A}\Ujdot$.
\end{proof}

\subsection{Homomorphism from $\Kj$ to $\Sj$}
 \label{sec:homom j}
 
Now we establish a precise relation between the algebras $\Kj$ and $\Sj$. 

\begin{prop}
 \label{prop:homom KSj}
There exists a unique surjective $\A$-algebra homomorphism $\phi_d: \Kj \rightarrow \Sj$ such that,
for $R \ge 0$, $i \in [1,n]$ and $D \in \txid$,
\begin{align*}
\phi_d( [D+ R E^{\theta}_{i,i+1}]) &= 
\begin{cases}
[D+ R E^{\theta}_{i,i+1}], &\text{ if } D+ R E^{\theta}_{i,i+1} \in \Xi_d;\\
0, &\text{ otherwise;}
\end{cases}\\
\phi_d( [D+ R E^{\theta}_{i+1,i}]) &= 
\begin{cases}
[D+ R E^{\theta}_{i+1,i}], &\text{ if } D+ R E^{\theta}_{i+1,i} \in \Xi_d;\\
0, &\text{ otherwise. }
\end{cases}
\end{align*}
Hence we have a surjective $\A$-algebra homomorphism $\phi_d\circ \aleph: \Ujdot \rightarrow \Sj$. 
\end{prop}

\begin{proof}
The existence of a homomorphism of $\Qq$-algebras $\phi_d: \QQ\Kj \rightarrow \QQ\Sj$ 
(or equivalently, a homomorphism $\phi_d\circ \aleph: \QQ \Ujdot \longrightarrow \QQ\Sj$) given by
the above formulas with $R=0,1$ 
follows immediately from Proposition~\ref{Type-B-case-I-homomorphism} and the presentation of $\Ujdot$ (and of $\QQ\Kj$)
from Proposition~\ref{prop:presentationUjdot} and Theorem~\ref{thm:UisoK}. The surjectivity and uniqueness of such $\aleph$ are clear.

It follows by the 
multiplication formulas for $\Kj$ and for $\Sj$ that $\phi_d$ matches the ``divides powers", as indicated by the formulas in the proposition.
Now the proposition follows by noting that these ``divided powers" (which correspond to the divided powers in ${}_\A\Ujdot$ by Proposition~\ref{prop:powergen})
generate $\Kj$ and $\Sj$. 
\end{proof}

\subsection{A geometric duality}
 \label{sec:duality j}

Recall the embedding $\jmath: \U^{\jmath} \rightarrow \U(\gl(N))$ in Proposition~\ref{prop:embedding}.
 Let $\VV$ be the natural representation of $\U(\gl(N))$ with the standard basis $\{\texttt{v}_{1}, \ldots, \texttt{v}_N\}$. 
 Then $\VV^{\otimes d}$ becomes a $\U(\gl(N))$-module via the coproduct:
\begin{align*}
\Delta: \U(\gl(N)) &\longrightarrow \U(\gl(N)) \otimes \U(\gl(N)),\\
E_{i} &\mapsto 1 \otimes E_{i} + E_{i} \otimes K_{i}K^{-1}_{i+1},\\
F_{i} &\mapsto F_{i} \otimes 1 + K^{-1}_{i}K_{i+1} \otimes F_{i},\\
K_a & \mapsto K_a \otimes K_a, 
\end{align*}
for  $i = 1, \dots, n$ and $a = 1, \dots, n+1$.
Then $\VV$ and $\VV^{\otimes d}$ are naturally $\Uj$-modules via the embedding $\jmath : \Uj \rightarrow \U(\gl(N))$. 
Following \cite[23.1]{Lu93}, $\VV^{\otimes d}$ becomes a $\Ujdot$-module as well. 

We write $ \texttt{v}_{r_1\dots r_d} = \texttt{v}_{r_1}\otimes \cdots \otimes \texttt{v}_{r_d}$. 
There is a right action of the Iwahori-Hecke algebra $\Hb$  on the $\Qq$-vector space $\VV^{\otimes d}$ as follows: 
\begin{equation}
  \label{eq:Hb1}
 \texttt{v}_{r_1 \dots r_d} T_j =
\begin{cases}
v  \texttt{v}_{r_1\dots r_{j-1} r_{j+1} r_j r_{j+2}\dots r_d} ,  & \text{ if } r_j < r_{j+1};\\
v^2  \texttt{v}_{r_1\dots r_d}, &  \text{ if } r_j = r_{j+1};\\
(v^2-1)  \texttt{v}_{r_1\dots r_d} + v  \texttt{v}_{r_1 \dots r_{j-1} r_{j+1} r_jr_{j+2}  \dots r_d}, & \text{ if } r_j> r_{j+1},
\end{cases}\\
\end{equation}
for $1\leq j\leq d-1$, and  
\begin{equation}
 \label{eq:Hb2}
 \texttt{v}_{r_1 \dots r_{d-1} r_{d}} T_d =
\begin{cases}
v  \texttt{v}_{r_1\dots r_{d-1} r_{d+2} } ,  &  \text{ if } r_d < n+1;\\
v^2   \texttt{v}_{r_1\dots r_{d-1} r_d}, &  \text{ if }  r_d = n+1;\\
(v^2-1)  \texttt{v}_{r_1\dots r_{d-1} r_d} + v  \texttt{v}_{r_1 \dots r_{d-1} r_{d+2}},  &  \text{ if } r_d> n+1,
\end{cases}\\
\end{equation}
where $r_{d+2} = N+1 - r_{d-1}$.
The $(\Ujdot, \Hb)$-duality established in \cite[Theorem~6.26]{BW},
states that we have commuting actions of  $\Ujdot$ and $\QQ\Hb$ on $\VV^{\otimes d}$
which form double centralizers. We caution the reader that the conventions of the algebras $\Uj$ and $\Hb$ formulated above
are chosen to fit with the geometric counterpart (see Proposition~\ref{prop:same duality} below) and they differ from those in {\em loc. cit.}

Recall the right action of $\Hb$ on $\Td$ from Lemma~\ref{Hecke-action}.
The left action of $\Sj$ on $\Td$ given in \S\ref{sec:convolutionaction} is lifted to a left action of $\Kj$ on $\Td$ via the homomorphism $\phi_d: \Kj \rightarrow \Sj$.  
We define a $\Qq$-vector space isomorphism:
\begin{align}
 \label{eq:Omega}
\begin{split}
\Omega:  \VV^{\otimes d}   &\longrightarrow \QQ\Td \\
\texttt{v}_{r_1} \otimes \texttt{v}_{r_2} \otimes \cdots \otimes \texttt{v}_{r_d}   & \mapsto  \tilde{e}_{r_1r_2\cdots r_d}.
\end{split}
\end{align}
The following provides a geometric realization of the $(\Ujdot, \Hb)$-duality established in \cite[Theorem~6.26]{BW}.
 
\begin{thm}
 \label{thm:same duality}
We have the following commutative diagram of double centralizing actions
under the identification
$\Omega:  \VV^{\otimes d}  \longrightarrow \QQ\Td$:
\begin{eqnarray*}
\begin{array}{ccccc}
  \QQ\Kj &\circlearrowright   
 & \QQ\Td & \circlearrowleft  \; & \QQ\Hb 
  \\
\aleph \uparrow\;\; & & \Omega \uparrow & & \parallel
 \\
\Ujdot    & \circlearrowright  
& \VV^{\otimes d} &  \circlearrowleft  \; & \QQ\Hb 
\end{array}
 \end{eqnarray*}
\end{thm}
A general consideration (cf. \cite{P09}) also shows that the actions of $\QQ\Sj$ and $\QQ\Hb$ on $\QQ\Td$ satisfy the double centralizer property
(under the assumption that $n \ge d$).  

\begin{rem}
The $\imath$canonical basis of the $\Uj$-module $\VV^{\otimes d}$ was constructed in \cite{BW},
and it does not coincide with Lusztig's canonical basis of $\VV^{\otimes d}$ (regarded as a $\U(\mathfrak{gl}(N))$-module).
An $\imath$canonical  basis on $\Td$ can also be defined via a standard intersection complex construction (similar to the one
on $\Sj$). 
These two $\imath$canonical bases coincide under the isomorphism $\Omega$. 
\end{rem}

\begin{rem}
The (geometric) symmetric Howe duality in Remark~\ref{rem:Howe} can also afford an algebraic formulation.
Note that an (algebraic) skew Howe duality was formulated in \cite{ES13} though the double centralizer
property was not proved therein (actually this property can be proved easily using  $\imath$Schur duality 
and a deformation argument). 
\end{rem}

\section{A second geometric construction in type $B$}
 \label{sec:2B}

In this section, we provide a realization of a second Schur algebra $\Si$ (which is a subalgebra of $\Sj$) in the geometric framework
of type $B$. We establish a generating set and prove a number of  relations of $\Si$. 
The algebra $\Si$ contains a subtle new generator $\bt$, which makes $\Si$ quite different from $\Sj$. 
We show that the monomial and canonical bases of $\Sj$ restrict to monomial and canonical bases of its subalgebra $\Si$.

\subsection{The setup} 
\label{sec:setup-i}

Recall $\Xi_d, \Pi$ from \S\ref{odd-orth} and $\txi$ from \eqref{eq:txi}.
Let ${}^\imath\Xi_d$ (and respectively, ${}^\imath \txi$) be  a subset of $\Xi_d$ (and respectively, $\txi$)
consisting of matrices $A$ in $\Xi_d$ (and respectively, in $\txi$)
such that all the entries in the $(n+1)$st row and in the $(n+1)$st column are $0$ except $a_{n+1, n+1}=1$.
Also set 
$$
{}^\imath \txi^{\text{diag}} =\{A \in {}^\imath\txi\mid A \text{ is diagonal}\},
\quad 
{}^\imath\Xi_d^{\text{diag}} =\{A \in {}^\imath\Xi_d\mid A \text{ is diagonal}\}.
$$ 
Similarly, we denote by ${}^\imath \tilde{\Pi}$ the subset of $\Pi$ consisting of matrices $B =(b_{ij}) \in \Pi$
such that all the entries in the $(n+1)$st row and in the $(d+1)$st column are $0$ except $b_{n+1,d+1}=1$.
This extra condition makes a matrix in ${}^\imath \Xi_d$,  ${}^\imath \txi$ or ${}^\imath \Pi$ look like 
\[
\begin{pmatrix}
* & \cdots & * & 0 & * & \cdots & * \\
*& \cdots & * & 0  & * & \cdots & *\\
\cdots\\
0 & \cdots & 0 & 1 & 0 & \cdots &0\\
\cdots \\
* & \cdots & * & 0 & * & \cdots & * \\
*& \cdots & * & 0  & * & \cdots & *
\end{pmatrix}.
\]

Recall $X$ from \S\ref{odd-orth}. We consider the following subset of $X$:
\[
{}^\imath X =\{ V =(V_i) \in X \mid V_n \; \mbox{is maximal isotropic}\}.
\]
The products ${}^\imath X\times {}^\imath X$ and ${}^\imath X\times Y$ are invariant under the diagonal action of $O(D)$.
Note that $V_n$ being maximal isotropic for $V=(V_i) \in {}^\imath X$ (and $V_{n+1} =V_n^{\perp}$) implies that
$V_{n+1}/V_n$ is one-dimensional. Therefore  Lemma~\ref{lem:bijectionorbits} readily
leads to the following.

\begin{lem}
The bijections in Lemma~\ref{lem:bijectionorbits}  induce the following bijections: 
\begin{equation*}
O(D) \backslash {}^\imath X\times {}^\imath X \longleftrightarrow {}^\imath \Xi_d, 
\quad \mrm{and}\quad
O(D) \backslash {}^\imath X\times Y \longleftrightarrow {}^\imath \Pi. 
\end{equation*}
\end{lem}

The following is a counterpart of  Lemma~\ref{Case-I-dimensions} and we skip the similar argument.

\begin{lem}
We have 
\[
\#\; {}^\imath \Xi_d = \binom{2n^2 +d-1}{d},
\quad \mrm{and}\quad
\#\; {}^\imath \Pi = (2n)^d.
\]
\end{lem}

We define 
$$
\Si =\Axxi, \qquad \Tdi =\Axyi 
$$
to be the space of $O(D)$-invariant $\A$-valued functions on ${}^\imath X \times \ \! {}^\imath X$ and ${}^\imath X \times Y$, respectively.
Following \S\ref{sec:convolutionaction},  under the convolution product $\Si$ is an $\A$-subalgebra of $\Sj$,
$\Tdi$ is a left $\Si$-submodule and also a right $\Hb$-submodule of $\Td$,
and we obtain commuting actions of $\Si$ and $\Hb$ on $\Tdi$.
In particular, $\Si$ is a free $\A$-module with a basis $\{e_A \mid A \in{}^\imath \Xi_d\}$
and with a standard basis $\{[A] \mid A \in{}^\imath \Xi_d\}$ (inherited from their counterparts in $\Sj$ by restriction).
Denote 
$$\QQ\Si =\Qq \otimes_\A \Si,\qquad
\QQ\Tdi =\Qq \otimes_\A \Tdi.
$$

\subsection{Relations for $\Si$}

We can still define the elements
$\E_i, \F_i$ and $\mbf d_a$ in $\Si$ for $i\in [1, n-1]$ and $a\in [1,n]$ as done for $\Sj$ in 
\S\ref{sec:rel}.  However, the elements $\E_n$ and $\F_n$ defined for $\Sj$
will no longer make sense here. Instead, 
we define a new element $\T \in \Si$ by setting, for  $V, V' \in {}^\imath X$, 
\begin{align}
 \label{eq:T}
\T(V, V') =
\begin{cases}
v^{- (|V'_n|-|V'_{n-1}|)}, & \mbox{if} \; |V_n\cap V_n'|\geq d-1, V_j=V_j', \forall j\in [1,n-1];\\
0, & \mbox{otherwise}.
\end{cases}
\end{align}

\begin{rem}
\label{rem:T}
One checks that, for  $V, V' \in {}^\imath X$, (and recall that $|V_{n+1}'/V_{n}'| =1$), 
\begin{equation} 
 \label{eq:Tfe}
\T (V, V') = \F_n \E_n(V, V') -  \delta_{V, V'} \left \llbracket |V_n'/V_{n-1}'|-|V_{n+1}'/V_{n}'| \right \rrbracket.
\end{equation}
\end{rem}

\begin{prop} 
\label{Type-B-case-II-homomorphism}
The elements $\E_i, \F_i$, $\mbf d_i^{\pm 1}$, $\mbf d_{i+1}^{\pm 1}$  for $i\in [1,n-1]$, and $\T$ in $\Si$  satisfy the following relations:
\begin{align}
 \text{the } & \text{defining relations of } \mbf U(\mathfrak{gl} (n)) \text{ for }
\E_i, \F_i, \mbf d_i^{\pm 1}, \mbf d_{i+1}^{\pm 1}, \forall i\; (\text{see Remark } \ref{rem:sln}); \tag{a}\\
 \mbf d_i \T =&  \T \mbf d_i, \quad \forall i\in [1, n]; \tag{b}\\
 \E_i \T =&  \T \E_i, \quad \F_i \T= \T \F_i,  \quad \forall i\in [1, n-2]; \tag{c}\\
\E_{n-1}^2  \T & - \llbracket 2\rrbracket  \E_{n-1} \T \E_{n-1} + \T \E_{n-1}^2 =0;  \tag{d}\\
 \F_{n-1}^2 \T & - \llbracket 2\rrbracket  \F_{n-1} \T \F_{n-1} + \T \F_{n-1}^2=0; \tag{e}\\
 \T^2 \E_{n-1}  &- \llbracket 2\rrbracket  \T\E_{n-1} \T + \E_{n-1} \T^2 = \E_{n-1}; \tag{f}\\
 \T^2 \F_{n-1}  &- \llbracket 2\rrbracket  \T\F_{n-1} \T + \F_{n-1} \T^2 = \F_{n-1}. \tag{g}
\end{align}
\end{prop}

\begin{proof}
It suffices to prove the formulas when we specialize $v$ to $\texttt{v}\equiv \sqrt{q}$
and then perform the convolution products over $\mathbb F_q$. 

The relations (a), (b) and  ($\text{c}$)  are clear. 
The remaining relations can be reduced to Proposition ~\ref{Type-B-case-I-homomorphism} by using \eqref{eq:Tfe}. 
Below let us give a direct proof. 

We  now prove (d). 
Without loss of generality, we may and shall assume that $n=2$. 
A direct computation yields
\begin{eqnarray*}
&\E_1^2 \T (V, V') =
\begin{cases}
\texttt{v}^{-3(d-|V_1|) +6} (\texttt{v} +\texttt{v}^{-1}), & \mbox{if} \; V_1\overset{2}{\subset} V_1' \subset V_2, |V_2\cap V_2'| \geq d-1,\\
0, & \mbox{otherwise};
\end{cases}
\end{eqnarray*}

\begin{eqnarray*}
&\T\E_1^2(V, V') =
\begin{cases}
\texttt{v}^{-3(d-|V_1|)+4}  (\texttt{v} +\texttt{v}^{-1}), & \mbox{if} \; V_1\overset{2}{\subset} V_1', |V_2\cap V_2'|\geq d-1, \\
0, & \mbox{otherwise};
\end{cases} 
\end{eqnarray*}

\begin{eqnarray*}
& \E_1\T\E_1 (V, V') =
\begin{cases}
\texttt{v}^{-3(d- |V_1| ) +5}  (\texttt{v} +\texttt{v}^{-1}) , & \mbox{if}\; V_1\overset{2}{\subset} V_1' \subseteq V_2\cap V_2', |V_2\cap V_2'|\geq d-1, \\ 
\texttt{v}^{-3(d-|V_1|) + 4}, & \mbox{if}\; V_1\overset{2}{\subset} V_1'\not \subseteq V_2\cap V_2', |V_2\cap V_2'| \geq d-1, \\
0, & \mbox{otherwise}.
\end{cases}
\end{eqnarray*}
The relation (d) follows.  

Similarly, the relation (f) is obtained by  the following computations:
\begin{eqnarray*}
&\T^2 \E_1(V, V') =
\begin{cases}
\texttt{v}^{-3(d-|V_1|) + 1} \frac{q^{d-|V_1| +1} -1}{q-1} , & \mbox{if}\; V_1\overset{1}{\subset} V_1' , V_2 = V_2', \\
\texttt{v}^{-3(d-|V_1|+1} (q+1) , & \mbox{if}\; V_1\overset{1}{\subset} V_1', |V_2\cap V_2'| = d-1 \;\mbox{or}\; d-2,\\
0, & \mbox{otherwise}; 
\end{cases}
\end{eqnarray*}

\begin{eqnarray*}
& \E_1\T^2 (V, V') =
\begin{cases}
\texttt{v}^{-3(d-|V_1|) + 3} \frac{q^{d-|V_1|} -1}{q-1}, & \mbox{if}\; V_1\overset{1}{\subset} V_1', V_2=V_2', \\
\texttt{v}^{-3(d-|V_1|) +3} (q+1) , & \mbox{if} \; V_1\overset{1}{\subset} V_1' \subset V_2, |V_2\cap V_2'| = d-1\;\mbox{or}\; d-2, \\
0, &\mbox{otherwise};
\end{cases} 
\end{eqnarray*}

\begin{eqnarray*}
& \T\E_1\T (V, V') =
\begin{cases}
\texttt{v}^{-3(d-|V_1|) +2} \frac{q^{d-|V_1|} -1}{q-1}, & \mbox{if} \; V_1\overset{1}{\subset} V_1', V_2=V_2', \\
\texttt{v}^{-3(d-|V_1|) +2} (q+1), & \mbox{if}\; V_1\overset{1}{\subset}V_1' \subset V_2, |V_2\cap V_2'|= d-1\; \mbox{or}\; d-2, \\
\texttt{v}^{-3(d-|V_1|) +2}, &\mbox{if}\; V_1\overset{1}{\subset} V_1' \not \subset V_2, |V_2\cap V_2'| = d-1\;\mbox{or}\; d-2, \\
0, &\mbox{otherwise}. 
\end{cases}
\end{eqnarray*}
Finally, the involution $(V, V') \mapsto (V', V)$ on ${}^\imath X\times {}^\imath X$ induces an anti-involution $\sigma$ on $\Si$ such that $\sigma (\E_1)= \texttt{v}^{d+1} \F_1$,
$\sigma(\mbf d_1) =\mbf d_1$, $\sigma(\mbf d_2) =\mbf d_2$,  and $\sigma(\T) =\T$.
This implies that the relations (e) and (g) follow from (d) and (f), respectively.
\end{proof}

\subsection{A sheaf-theoretic description of $\T$}

We shall now give a geometric interpretation of the  element $\T$. 
To do this, we need to divert  from the previous  setting over a finite field to a setting over the complex field $\mbb C$.
Let ${}^\imath X(\mbb C)$ and $Y(\mbb C)$  be the algebraic varieties defined over $\mbb C$, analogous to ${}^\imath X$ and $Y$, respectively.
Let $S(\tau)$ be the closed subvariety in ${}^\imath X(\mbb C) \times {}^\imath X(\mbb C)$ defined by 
\[
S(\tau) =\left\{ (V, V') \in {}^\imath X(\mbb C) \times {}^\imath X(\mbb C) \mid |V_n\cap V_n'| \geq d-1, V_j=V_j' ,\;  \forall j\in [1, n-1] \right \}.
\]
It is clear that  $S(\tau)$ is the subvariety corresponding to  the support of $\T$. 

\begin{lem}
The variety $S(\tau)$ is  rationally smooth.
In particular, the constant sheaf on $S(\tau)$ is a semisimple complex. 
\end{lem}

\begin{proof}
The trick is to reduce the proof  to a problem in the Schubert varieties in $Y(\mbb C)$ and then apply the known results of their singular loci. 
Without loss of generality, we assume that $n=1$.
We only need to show that each connected component in $S(\tau)$ is rationally smooth.  Let us fix a connected component, say $C$,   in  $S(\tau)$. 
This is further  reduced to show the rational smoothness of  the inverse image, say $C'$, in $Y(\mbb C)\times Y(\mbb C)$ of $C$  under the projection from 
$Y(\mbb C)\times Y(\mbb C)$ to a suitable connected component of ${}^\imath X(\mbb C)\times {}^\imath X(\mbb C)$.
Fix a flag $F$  in $Y(\mbb C)$, and let $C''=C' \cap \{F\} \times Y(\mbb C)$. 
It is finally reduced to show that $C''$ is rationally smooth. Observe that $C''$ is a Schubert variety in $Y(\mbb C)$. 
It is indexed by the permutation $\tau = (d-1, \cdots 1, \overline d)$  
in the  Weyl group of type $B_d$ in the notation in \cite{BL00} and \cite{B98}. 
This permutation  does not contain any bad pattern in ~\cite[13.3.3]{BL00}, hence implies that $C''$ is rationally smooth
(see \cite[A1]{KL79}) by \cite[8.3.16]{BL00}.
\end{proof}

A direct consequence of this lemma is that $\T$ is a shadow of  a semisimple complex. 
Its idempotent components are shadows of simple perverse sheaves, up to shifts.

\subsection{Generators for $\Si$}
 \label{sec:multi Si}

Recall $\Si$ is an $\A$-subalgebra of $\Sj$. Fix any $A\in {}^\imath \Xi_d$.  
By  (\ref{3.9a}), $m_A \in \Sj$ can be generated by elements of the form $[D_{i,h, j} + a_{ij} E^{\theta}_{h+1,h}]$.  
Note that  elements of the form $[D_{i,h,j}+ a_{ij} E^{\theta}_{n, n+1}]$ or  $[D_{i,h,j}+ a_{ij} E^{\theta}_{n+1, n}]$ in general do not lie in $\Si$. 
However,  we have the following key observation:
\begin{align}
\text{\em 
The ``twin product" } [D_{i,h,j}+ a_{ij} E^{\theta}_{n, n+1}] *[D_{i,h,j}+ a_{ij} E^{\theta}_{n+1, n}]  
\text{\em 
  is   an element in  }   \Si. 
 \notag
 \\
 \text{\em 
 The monomial basis element $m_A$ given in Theorem~\ref{BLM3.9} is always a product  
}
\label{twin}
\\
\text{\em of such  twin products  together with
$[D_{i,h, j} + a_{ij} E^{\theta}_{h+1,h}] \in \Si$ for $h\neq n, n+1$. 
}
\notag
 \end{align}
Indeed, thanks to $A\in {}^\imath \Xi_d$,  the $(n+1)$-th entry of the  row vector of $D_{i,h,j}+ a_{ij} E^{\theta}_{n, n+1}$ is 1
and so $(D_{i,h,j})_{n+1,n+1} =1$. 
By Theorem~\ref{BLM3.4b}, for $R>0$ and $D$ with $D_{n+1,n+1} =1$, we have 
\begin{align}
 \label{enfn}
\begin{split}
[D+ &R E^{\theta}_{n,n+1} ]  * [D+R E^{\theta}_{n+1,n}] 
\\
=&[D+ RE_{n, n+2}] + \sum^R_{i=1} v^{\beta(i)} \overline{\begin{bmatrix}
D_{n,n} + i\\
 i
\end{bmatrix}
} 
\big[D+ i E^{\theta}_{n,n}   + (R-i) E^{\theta}_{n, n+2}  \big]
\end{split}
\end{align}
where 
$
\beta(i) = D_{n,n} i -  {i(i+1)}/{2}.
$

Thus we  have $m_A \in \Si$ for any $A\in {}^\imath \Xi_d$.
For example,  for $n=2$, $m_A$ is a product of $20$ terms in  (\ref{3.9a}) 
which simplifies to 14 terms thanks to $A \in {}^\imath \Xi_d$ as follows,
and it actually lies  in $\Si$:
\begin{align*}
& [D_{2,1,1} + a_{21} E^{\theta}_{21}] 
 * \big( [D_{4,3,1} + a_{41} E^{\theta}_{43}] *[ D_{4,2,1} + a_{41} E^{\theta}_{32} ] \big) * [D_{4,1,1} +a_{41} E^{\theta}_{21}] \\
&* \big( [D_{4,3,2} + a_{42} E^{\theta}_{43}] *[ D_{4,2,2} + a_{42} E^{\theta}_{32}] \big) 
* [D_{5,4,1} + a_{51} E_{54}^{\theta} ] \\
&* \big( [D_{5, 3, 1} + a_{51} E^{\theta}_{43}]* [ D_{5,2,1} + a_{51} E^{\theta}_{32}] \big) *  [D_{5,1,1} + a_{51} E^{\theta}_{21}] 
 * [D_{5, 4, 2} + a_{52} E^{\theta}_{54}]\\
&* \big([D_{5, 3, 2}+ a_{52} E^{\theta}_{43} ] *[D_{5, 2, 2} + a_{52} E^{\theta}_{32}] \big) 
* [D_{5,4, 4} + a_{54} E^{\theta}_{54}].
\end{align*}
We summarize the above discussions as follows.

\begin{prop}
 \label{prop:Si}
For any $A\in {}^\imath \Xi_d$, we have $m_A \in \Si$. Moreover, $\{m_A \mid A\in {}^\imath \Xi_d\}$ forms a monomial $\A$-basis for $\Si$. 
\end{prop}

\begin{prop}
\label{S_2-generator}
The $\A$-algebra $\Si$ is generated by the elements $[D+R E^{\theta}_{n,n+2} ]$, $[D+ R E^{\theta}_{i,i+1}]$,  $[D+ RE^{\theta}_{i+1,i}]$
where  $1\le i \le n-1$,  $R\in[0,d]$ and $D\in {}^\imath \Xi_{d-R}^{\text{diag}}$.
\end{prop}

By using Lemma  ~\ref{BLM3.8}, we have
\[
[D+ R E^{\theta}_{n, n}+ E^{\theta}_{n,n+2}] *[ D +E^{\theta}_{n, n} + R E^{\theta}_{n, n+2}] = [D+ (R+1) E^{\theta}_{n, n+2}] + \mbox{lower terms}.
\]
Hence Proposition~\ref{S_2-generator} implies the following.

\begin{cor}
 \label{cor:gen-i-Q}
The $\mbb Q(v)$-algebra $\QQ\Si$ is generated by the elements  $[D']$ with  $D' \in {}^\imath \Xi_{d}^{\text{diag}}$, and
the elements
$[D+  E^{\theta}_{i,i+1}]$,  $[D+ E^{\theta}_{i+1,i}]$,
$[D+ E^{\theta}_{n,n+2} ]$,  where $1\le i \le n-1$ and $D \in {}^\imath \Xi_{d-1}^{\text{diag}}$.
\end{cor}

Recall in \S\ref{subset:CBonS}, we have defined the canonical basis $\mbf B_d^\jmath = \{ \{A\} \mid A\in \Xi_d\}$ 
for $\Sj$. The bar involution on $\Si$ is identified with the restriction from
 the bar involution on $\Sj$ via the inclusion $\Si \subset \Sj$ by the geometric construction.
 It follows from Proposition~\ref{prop:Si} (and recall every monomial basis element is bar invariant) that
 $\{A\} \in \Si$, for $A\in {}^\imath  \Xi_d$.
In particular, we have the following.

\begin{thm}
The canonical basis for $\Si$ is given by $\mbf B_d^\jmath \cap \Si = \{ \{A\} \mid A\in {}^\imath  \Xi_d\}$. 
\end{thm}

Hence all the results on the inner product and canonical basis for $\Sj$ 
in \S\ref{subset:CBonS} and \S\ref{sec:inner Sj} make sense by restriction to the subalgebra $\Si$. 
We also obtain a geometric realization of the second $\imath$Schur duality (on the Schur algebra level) as
$\QQ\Si \circlearrowright   
  \QQ\Tdi  \circlearrowleft  \;  \QQ\Hb.
$
To avoid much repetition, we will explain this duality in some detail in Appendix~A,
 where the Schur algebra $\Si$ is replaced by a modified coideal algebra.

\section{Convolution algebras from geometry of type $C$}
 \label{sec:typeC}

In this section, we formulate analogous constructions and results in type $C$. 
This could have been done in a completely analogous way as before, but to avoid much repetition we choose some short cuts to reduce
the considerations quickly to the type $B$ counterparts. 

\subsection{A first formulation}
\label{symplectic-I}

We fix the following data in this subsection:
\begin{itemize}
\item A pair of positive integers $(n, d)$ such that $N= 2n$ and $D=2d$ in Section ~\ref{prelim};

\item A non-degenerate skew-symmetric bilinear form $Q: \mbb F_q^D\times \mbb F_q^D \longrightarrow  \mbb F_q$.
\end{itemize}

Let $Sp(D)$ be the symplectic subgroup of $GL(D)$ consisting of all elements $g$ such that 
$Q(gu, gu') = Q(u, u')$, for $u, u'\in \mbb F_q^D$.
We can define the sets $X_{\mbf C_d}$, $Y_{\mbf C_d}$, $\Xi_{\mbf C_d}$,  $\Pi_{\mbf C_d}$ 
and $\Sigma_{\mbf C_d}$ in formally the same way (in notations $N$ and $D$) 
as the sets $X$, $Y$, $\Xi_d$, $\Pi$ and $\Sigma$ in Section ~\ref{odd-orth}, respectively.  
For example, $X_{\mbf C_d}$ is the set of $N$-step flags $(V_i)_{0\le i \le N}$ in $\mbb F_q^D$ satisfying $V_{N-i} =V_i^\perp$, and 
in particular, $V_n = V_n^{\perp}$ is Lagrangian. 

We have the following  analogue of Lemmas~ \ref{lem:bijectionorbits} and \ref{Case-I-dimensions},  whose proof is skipped.

\begin{lem}
\begin{enumerate}
\item[(a)]  
We have
$$\# \Sigma_{\mbf C_d} = 2^d \cdot d!, \quad  \# \Pi_{\mbf C_d}  = (2n)^d,
\quad \text{ and }\; \#\Xi_{\mbf C_d} = \binom{2n^2+ d-1 }{d}.
$$

\item [(b)]  We have canonical bijections
$Sp(D) \backslash X_{\mbf C_d}\times X_{\mbf C_d} \leftrightarrow \Xi_{\mbf C_d}$, 
  \; $Sp(D) \backslash X_{\mbf C_d}\times Y_{\mbf C_d} \leftrightarrow \Pi_{\mbf C_d}$,
         and  $Sp(D)\backslash Y_{\mbf C_d}\times Y_{\mbf C_d} \leftrightarrow \Sigma_{\mbf C_d}$. 
\end{enumerate}
\end{lem}

Following \S\ref{sec:convolutionaction} we define 
$$
\Sic =\Axxc, \quad \Tdc =\Axyc, \quad \Hb =\Ayyc
$$
to be the space of $Sp(D)$-invariant $\A$-valued functions on $X_{\mbf C_d} \times X_{\mbf C_d}$, $X_{\mbf C_d} \times Y_{\mbf C_d}$, and $Y_{\mbf C_d} \times Y_{\mbf C_d}$ respectively.
(Note the superscript $\imath$ here instead of $\jmath$ is used!)
As before, $\Sic$ is endowed with an $\A$-algebra structure via a convolution product. 
Also, the convolution algebra $\Ayyc$ is known to be canonically isomorphic to the Iwahori-Hecke algebra $\Hb$ of type $B_d$, and
so there is no ambiguity of notation above.

Associated to $B \in \Pi_{\mbf C_d}$, we have a sequence of integers $r_1, \ldots, r_D$ as defined in \eqref{eq:rc} which
satisfies the same bijections \eqref{bij:rr}. 
We shall denote the characteristic function of the $Sp(D)$-orbit $\Ob_B$ by $e_{r_1 \dots r_d}$.
The following is an analogue of Lemma~\ref{Hecke-action}, whose similar proof is skipped.

 \begin{lem}
  \label{lem:Hecke action}
The right $\Hb$-action on $\Tdc$
$$
\Tdc \times \Hb \longrightarrow \Tdc, \qquad
(e_{r_1 \dots r_d}, T_j ) \mapsto e_{r_1 \dots r_d} T_j,
$$
is given as follows.
For $1\leq j\leq d-1$, we have
\begin{equation}
\label{5}
e_{r_1 \dots r_d} T_j =
\begin{cases}
e_{r_1\dots r_{j-1} r_{j+1} r_jr_{j+2}\dots r_d},  & \text{ if }  r_j < r_{j+1};\\
v^2 e_{r_1\dots r_d}, & \text{ if }   r_j = r_{j+1};\\
(v^2-1) e_{r_1\dots r_d} + v^2 e_{r_1 \dots r_{j-1} r_{j+1} r_j r_{j+2} \dots  r_d}, &\text{ if }   r_j> r_{j+1}.
\end{cases}
\end{equation}
 Moreover, (recalling $r_{d+1} = N+1 - r_d$) we have
\begin{equation}
\label{6}
e_{r_1 \dots r_{d-1}  r_{d}} T_d =
\begin{cases}
e_{r_1\dots r_{d-1} r_{d+1}} ,  & \text{ if }  r_d < r_{d+1};\\
v^2 e_{r_1\dots r_{d-1} r_d}, & \text{ if }   r_d = r_{d+1};\\
(v^2-1) e_{r_1\dots r_{d-1} r_d} + v^2 e_{r_1 \dots r_{d-1} r_{d+1}}, & \text{ if }  r_d> r_{d+1}.
\end{cases}
\end{equation}
\end{lem}

We set
\[
\tilde e_{r_1\cdots r_d} = v^{\#\{(c, c')| c, c'\in [1, d+1], c<c', r_c < r_{c'}\}  } e_{r_1\cdots r_d}.
\]
The formulas (\ref{5}) and (\ref{6}) can be rewritten as follows.
For $1\leq j\leq d-1$, 
\begin{equation}
\label{7}
\tilde e_{r_1 \dots r_d} T_j =
\begin{cases}
v \tilde e_{r_1\dots r_{j-1} r_{j+1} r_j r_{j+2}\dots r_d} ,  &\text{if }    r_j < r_{j+1};\\
v^2 \tilde e_{r_1\dots r_d}, & \text{if }   r_j = r_{j+1};\\
(v^2-1)\tilde  e_{r_1\dots r_d} + v \tilde  e_{r_1 \dots r_{j-1} r_{j+1} r_j r_{j+2} \dots  r_d}, &\text{if }   r_j> r_{j+1}; 
\end{cases}
\end{equation}
\begin{equation}
\label{8}
\tilde e_{r_1 \dots r_{d-1} r_d} T_d =
\begin{cases}
v \tilde e_{r_1\dots r_{d-1} r_{d+1}} ,  & \text{ if }  r_d < r_{d+1};\\
v^2 \tilde e_{r_1\dots r_{d-1} r_d}, & \text{ if }   r_d = r_{d+1};\\
(v^2-1) \tilde e_{r_1\dots r_{d-1} r_d} + v\tilde  e_{r_1 \dots r_{d-1} r_{d+1}}, &\text{ if }   r_d> r_{d+1}.
\end{cases}
\end{equation}

\begin{rem}
The formulas (\ref{7}) and (\ref{8}) essentially coincide with the one given in ~\cite{G97} in an opposite ordering. 
(Note that the presentation of Iwahori-Hecke algebras in ~\cite{G97} is somewhat different from ours via $T_i$.) This shows that 
the $\A$-algebra
$\Sic$ is isomorphic to the hyperoctahedral Schur algebra in ~\cite{G97}. 
\end{rem}

\subsection{A variation}
\label{symplectic-II}

We fix the following data in this subsection:
\begin{itemize}
\item A pair of positive integers $(n, d)$ such that $N= 2n+1$ and $D=2d$ in Section ~\ref{prelim}; 

\item A non-degenerate skew-symmetric bilinear form $Q: \mbb F_q^D\times \mbb F_q^D \longrightarrow \mbb F_q$.
\end{itemize}
(Note the only difference from the data in Section ~\ref{symplectic-I} is that $N=2n+1$ now.)
We can define analogous sets as those in Section~\ref{symplectic-I}. 
Let ${}'X_{\mbf C_d}$ be the set defined formally as $X_{\mbf C_d}$ in Section ~\ref{symplectic-I} with now $N=2n+1$. 
We keep exactly the same $Y_{\mbf C_d}$ as in Section ~\ref{symplectic-I}.  
Following \S\ref{sec:convolutionaction},
we can define an $\A$-algebra $\Sjc := \A_{Sp(D)} ({}'X_{\mbf C_d}\times {}'X_{\mbf C_d})$ with a convolution product;
Similarly, we have a commuting left $\Sjc$-action and a right $\Hb$-action on $\Tdc' := \A_{Sp(D)} ({}'X_{\mbf C_d}\times Y_{\mbf C_d})$.
As before (see  \eqref{eq:rc} and \eqref{bij:rr}), $\Tdc'$ has a basis given by the characteristic functions $e_{r_1\cdots r_d}$, where the $d$-tuples $r_1\cdots r_d$
are in bijection with the $Sp(D)$-orbits in ${}'X_{\mbf C_d}\times Y_{\mbf C_d}$. 

There is a natural inclusion $X_{\mbf C_d} \subset {}'X_{\mbf C_d}$,  which identifies 
a flag $(\cdots \subseteq V_n \subseteq \cdots)$   with $(\cdots \subseteq V_n \subseteq V_n \subseteq \cdots)$.

\begin{lem}
The $\A$-algebra $\Sic$ 
           is naturally a subalgebra of  $\Sjc$ induced      
           by the inclusion $X_{\mbf C_d} \subset {}'X_{\mbf C_d}$.
\end{lem}

The next lemma follows by similar arguments for Lemmas~\ref{Case-I-dimensions} and \ref{Hecke-action}.
\begin{lem} 
\label{symplectic-I-II}
\begin{enumerate}
\item[(a)]   
We have 
$$
\#(Sp(D) \backslash {}'X_{\mbf C_d}\times {}'X_{\mbf C_d}) = \binom{2n^2+ 2n +d}{d},
\qquad \#(Sp(D) \backslash {}'X_{\mbf C_d}\times Y_{\mbf C_d} )=( 2n+1)^d.
$$

\item[(b)]  The right $\Hb$-action on $\Tdc'$ is given by the formulas (\ref{5})-(\ref{8}).
\end{enumerate}
\end{lem}

\subsection{Type $C$ vs  type $B$}

By Lemmas ~\ref{Case-I-dimensions}, \ref{Hecke-action}, and  ~\ref{symplectic-I-II}, we have a right $\Hb$-module isomorphism
$
\QQ\Td 
 \overset{\simeq}{\longrightarrow}  \QQ^{\mbf C} \Td
$
by sending $e_{r_1\cdots r_d}$ to the element in the same notation. 
This isomorphism induces  an algebra isomorphism 
\[
\phi_1: \End_{\Hb} ( \Td) 
 \overset{\simeq}{\longrightarrow} \End_{\Hb} \big( \Tdc'  
  \big).
\]
Earlier (see Theorem~\ref{thm:same duality} and \cite{P09}) we have obtained an $\A$-algebra isomorphism 
\[
\phi_2:  \Sj \longrightarrow \End_{\Hb} (\Td).
\]
Similarly, we have an $\A$-algebra isomorphism 
\[
\phi_3:  \Sjc \longrightarrow \End_{\Hb} (\Tdc').
\]
The following proposition allows us to reduce the type $C$ case to the type $B$ case. 

\begin{prop}
\label{symplectic-I-II-b} We have natural $\A$-algebra isomorphisms
 $\Sj \cong \Sjc$ and  $\Si \cong \Sic$.
\end{prop}

\begin{proof}
The first isomorphism is given by $\phi_3^{-1}\phi_2\phi_1$, and the second isomorphism can be obtained similarly. 
\end{proof}

Actually, the above isomorphisms are canonical in the sense they match various bases; we treat one case below in some detail. 

\begin{prop}
 The isomorphism $\phi_3^{-1}\phi_2\phi_1: \Sj \rightarrow \Sjc$ sends the basis element $e_A$, for $A \in \Xi_d$,  to the element $e_{A-E_{n+1,n+1}}$. 
\end{prop}

\begin{proof}
When $A$ is a diagonal matrix, the statement holds by definition.
When $A- E^{\theta}_{h,h\pm 1}$ is diagonal, the statement holds by checking directly that the action of $e_{A-E_{n+1,n+1}} \in \Sjc$ on $e_{r_1\cdots r_d}$ 
is the same as the action of $e_A \in \Sj$ on $e_{r_1 \cdots r_d}$. 
Note that the counterpart of Lemma ~\ref{BLM3.2} for $\Sjc$ holds with the same structure constants.  
The proof is  basically the same as that of Lemma ~\ref{BLM3.2} with $\epsilon^{\theta}_{n+1,n+1}$ defined to be 1 
with the extra care that the number of isotropic lines in $\mathbb F_q^D$
with respect to the skew-symmetric form is $\frac{q^D-1}{q-1}$. 
The proposition for general $A$ follows from this by induction. 
\end{proof}

\newpage
\appendix
\section{A geometric setting for the coideal algebra $\Uidot$  and compatibility of canonical bases
  \\
(by H. Bao, Y. Li, and W. Wang)
}

In this Appendix, we construct an $\A$-algebra $\Ki$ from the family of $\imath$Schur algebras $\Si$ as $d$ varies.
Since it is difficult to see directly the stabilization of $\Si$ whose new generators admit rather implicit multiplication formulas,
we construct $\Ki$ as a subalgebra of another limit algebra $\Kj_>$ in a two-step process. We then show that $\Ki$
is isomorphic to a subalgebra of a quotient of $\Kj$ and also isomorphic to a quotient of a subalgebra of $\Kj$.
All standard, monomial, and canonical bases of $\Ki$ and $\Kj$ (and the intermediate algebras)
are shown to be compatible. We then show that $\Ki$ is isomorphic to the modified coideal algebra $\Uidot$ associated to $\gl(N-1)$, for $N=2n+1$,
and obtain a geometric realization of the $(\Uidot, \Hb)$-duality.
Finally, we construct a surjective homomorphism $\phi_d^\imath: \Ki \rightarrow \Si$ (which is compatible
with the homomorphism $\phi_d: \Kj \rightarrow \Sj$ in \S \ref{sec:homom j}), and show that both homomorphisms $\phi_d$ and $\phi_d^\imath$
send canonical basis elements to canonical basis elements or zero.

\subsection{A limit algebra $\Kj_>$ and a subalgebra $\Ki$} 
 \label{subsec:Ki}

We set $N=2n+1$, and introduce the following subsets of $\txi$:
\begin{align*}
\txi_<  &= \{A =(a_{ij})  \in \tilde{\Xi} \mid  a_{n+1,n+1} <0\},
 \\
\txi_> &= \{A =(a_{ij})  \in \tilde{\Xi} \mid  a_{n+1,n+1} >0\}.
\end{align*} 

Recalling $I$ is the identity $N\times N$-matrix, for any $N\times N$-matrix $A$ and $p\in \Z$ we set 
$$\breve I = I -E_{n+1,n+1},
\qquad 
{}_{\breve p}A = A + p \breve{I}. 
$$
For $A \in \txi_>$, we  have ${}_{\breve p}A \in \Xi$ for $p \gg 0$.   

\begin{lem} \label{cor:Kjstab}
Given $A_{1}$, $A_2, \dots,  A_f \in \txi_>$,
there exist $Z_i \in \txi_>$ and $G_i (v, v') \in  \Qq[v',v'^{-1}]$ $(1 \le i \le m$, for some $m)$ such that 
\begin{equation}  \label{eq:Gi}
[{}_{\breve p}A_1] \ast [{}_{\breve p}A_2] \ast \ldots \ast [{}_{\breve p}A_f] = \sum^m_{i=1}G_i(v, v^{-p} )[{_{\breve p}Z_i}], \quad \text{ for all even integers } p \gg 0.
\end{equation}
\end{lem}

\begin{proof}
Though the proof is entirely similar to the proof of \cite[Proposition 4.2]{BLM} (where ${}_pA=A+pI$ is used), we provide some details to assure a reader
that the modification from ${}_pA$ to ${}_{\breve p}A$ does not cause extra problems.
It suffices to prove the lemma for $f=2$, as the general case follows by induction on $f$.
We shall use Theorem~\ref{BLM3.4b} and the notations therein. 

Let $1 \le h \le n$. We assume first that  $A_1 - RE^{\theta}_{h,h+1}$ is diagonal. Let $A_2= A = (a_{i,j})$. 
Let $T$ be the set of all $t = (t_1, t_2, \dots, t_{N}) \in \mathbb{N}^{N}$ such that $\sum^N_{i=1} t_i = R$ and 
\[
\begin{cases}
t_{u} \le a_{h+1,u }, &\text{ if } h < n \text{ and } u \neq h+1;\\
t_{u} +t_{N+1-u} \le  a_{h+1,u}, &\text{ if } h=n.
\end{cases}
\]
 For each $t \in T$, we define 
\[
G_{t}(v, v') = v^{\beta(t)} \prod^N_{\stackrel{u = 1}{u \neq h}}
\overline{
\begin{bmatrix}
a_{h,u} + t_{u}\\
t_{u}
\end{bmatrix}
}
\,\prod^{t_h}_{i=1}
\frac{v^{-2(a_{h,h}+t_h-i+1)}{v'}^2-1}{v^{-2i}-1}
({v'})^{-\delta_{h,n}\cdot \sum_{j < n+1}t_j}.
\]
Then by Theorem~\ref{BLM3.4b}, we have 
\[
[{}_{\breve p}A_1] \ast [{}_{\breve p}A] = \sum_{t \in T} G_t(v, v^{-p}) \Big[{}_{\breve p} \Big(A + \sum_{1\le u\le N} 
t_u(E^{\theta}_{h,u}- E^{\theta}_{h+1,u})\Big)\Big]\quad \text{ for all even } p \gg 0.
\]
Next we assume that  $A_1 - RE^{\theta}_{h+1,h}$ is diagonal. Let $A_2= A = (a_{i,j})$. 
Let $T$ be the set of all $t = (t_1, t_2, \dots, t_{N}) \in \mathbb{N}^{N}$ such that $\sum^N_{i=1} t_i = R$ and $t_{u} \le a_{h,u }$  for all $u \neq h$. 

For each $t \in T$, if $h \neq n$, we define 
\[
G_t(v, v') = v^{\beta'(t)} \prod^N_{\stackrel{u=1}{u \neq h+1}}
\overline{
\begin{bmatrix}
a_{h+1,u}+ t_u\\
t_u
\end{bmatrix}
}
\prod^{t_{h+1}}_{i=1} \frac{v^{-2(a_{h+1,h+1}+t_{h+1}-i+1)}(v')^2 -1}{v^{-2i}-1};
\]
if $h=n$, we define
\begin{align*}
G_t(v, v')=&
v^{\beta''(t)}
\prod_{u<n+1}\overline{ 
\begin{bmatrix}
a_{n+1,u}+t_u+t_{N+1-u}\\
 t_u
\end{bmatrix}
}
\prod_{u> n+1}\overline{
\begin{bmatrix}
 a_{n+1,u}+t_u\\
t_u
\end{bmatrix}}
\\
&\cdot \prod_{i=0}^{t_{n+1}-1}
\frac{\overline{[a_{n+1,n+1}+1+2i]}}
{\overline{[i+1]}} \cdot (v')^{\sum_{j \ge n+1}{t_j}}. 
\end{align*}
By Theorem~\ref{BLM3.4b} we have 
\[
[{}_{\breve p}A_1] \ast [{}_{\breve p}A] = \sum_{t \in T} G_t(v, v^{-p}) \Big[{}_{\breve p} \Big(A + \sum_{1\le u\le N} 
t_u(E^{\theta}_{h+1,u}- E^{\theta}_{h,u})\Big)\Big]\quad \text{ for even } p \gg 0.
\]

Therefore we have proved so far that 
\[
[{}_{\breve p}A_1] \ast [{}_{\breve p} A_2] = \sum^m_{i=1}G_{i}(v,v^{-p}) [{}_{\breve p} Z_i] \quad \text{ for even } p \gg0,
\]
where either $A_1 - RE^{\theta}_{h,h+1}$ or $A_1 - RE^{\theta}_{h+1,h}$ $(1 \le h\le n)$ is diagonal. 
The proof of this formula for general $A_1 \in \txi_>$ 
follows from the special cases which we have just  established by the same induction as in \cite[pp.668]{BLM}.
\end{proof}

We shall introduce an $\A$-algebra $\Kj_>$ below.
\begin{cor}\label{cor:Kjtilde}
Let $\Kj_>$  be a free $\mathcal{A}$-module with a basis $[A]$ for $A \in \txi_>$.
Then there is a unique structure of associative $\mathcal{A}$-algebra on $\Kj_>$ in which the product 
$[A_1] \cdot [A_2] \cdot \ldots \cdot [A_f]$ is given by  $\sum^m_{i=1}G_i(v, 1 )[Z_i]$ (in the notation of \eqref{eq:Gi}). 
\end{cor}
We remark the similarity of the stabilization procedure above leading to the algebra $\Kj_>$ with the one in type $D$ given in \cite{FL14}. 

Following the definition we obtain the following multiplication formula in the algebra $\Kj_>$.
For any $A, B\in \txi_>$ such that $\ro(A) =\co (B) $ and $B-RE_{h, h+1}^{\theta}$ being diagonal for some $h\in [1,n]$, we have
\begin{align}
\label{BLM4.6aa}
[B] \cdot [A] =\sum_{t} v^{\beta(t)} \prod_{u=1}^N \overline{
\begin{bmatrix}
a_{hu} + t_u\\
 t_u
\end{bmatrix}
} \Big [A+\sum_{1\le u\le N}  t_u (E^{\theta}_{hu}-E^{\theta}_{h+1,u}) \Big],
\end{align}
where  $\beta(t)$ is defined in \eqref{eq:lem:betat} and $t=(t_1,\ldots, t_N)\in \mbb N^N$ 
such that $\sum_{u=1}^N t_u=R$ and $ t_u  \leq a_{h+1,u}$ for $u\neq h+1$ and $h<n$
or $t_u+t_{N+1-u} \leq a_{n+1,u}$ for $h=n$.

Similarly for any   $A, C\in \txi_>$ such that $\ro(A) = \co(C)$ 
and $C-RE_{h+1,h}^{\theta}$ being diagonal for some $h\in [1,n-1]$, we have
\begin{align}
 \label{BLM4.6bb}
[C] \cdot [A]
=\sum_{t} v^{\beta'(t)} \prod_{u=1}^N \overline{
\begin{bmatrix}
a_{h+1, u} + t_u\\
 t_u
\end{bmatrix}
} \Big [A-\sum_{1\le u\le N}  t_u (E^{\theta}_{hu}-E^{\theta}_{h+1,u}) \Big ],
\end{align}
where $\beta'(t)$ is defined in \eqref{eq:lem:betat'} and $t=(t_1,\ldots, t_N)\in \mbb N^N$ 
such that $\sum_{u=1}^N t_u=R$ and $0\leq t_u  \leq a_{h,u}$ for $u\neq h$.
For $h=n$, we have 
\begin{equation}
 \label{BLM4.6bb2}
 \begin{split}
&[C]\cdot [A]=\sum_{t}v^{\beta''(t)}
\prod_{u<n+1}\overline{ 
\begin{bmatrix}
a_{n+1,u}+t_u+t_{N+1-u}\\
 t_u
\end{bmatrix}
}
\prod_{u> n+1}\overline{
\begin{bmatrix}
a_{n+1,u}+t_u\\
t_u
\end{bmatrix}}
\\
&\qquad\qquad\qquad\qquad \cdot \prod_{i=0}^{t_{n+1}-1}
\frac{\overline{[a_{n+1,n+1}+1+2i]}}
{\overline{[i+1]}}
\Big[A-\sum_{1\le u\le N} t_u(E^{\theta}_{nu}-
E^{\theta}_{n+1,u})\Big],
  \end{split}
\end{equation}
where $\beta''(t)$ is defined in \eqref{eq:lem:betat''} and 
$t=(t_1,\ldots, t_N)\in \mbb N^N$ such that $\sum_{u=1}^N  t_u=R$ and $0\leq t_u  \leq a_{n,u}$ for $u\neq n$.

An entirely analogous argument as for the construction of the monomial basis for $\Kj$ (see Theorem ~\ref{BLM3.9} and \eqref{BLM4.6c}),
now using \eqref{BLM4.6aa},
\eqref{BLM4.6bb} and \eqref{BLM4.6bb2}, proves the following.

\begin{prop}
The $\A$-algebra $\Kj_>$ admits a monomial basis $\{\tM_A \mid A\in \txi_>\}$.
\end{prop}
In other words, the monomial basis element $\tM_A$, for $ A\in \txi_>$, is formally given by the same formula as in  \eqref{BLM4.6c} though
with a different multiplication. 


Let $\Ki$ be the $\A$-submodule of $\Kj_>$ generated by $[A]$, for $A \in \txii$ (recall $\txii$ from \S\ref{sec:setup-i}).
Since any matrix $A \in \txi_>$ with $\co(A)_{n+1} =\ro(A)_{n+1} =1$ must lie in $\txii$, we conclude from 
\eqref{BLM4.6aa}--\eqref{BLM4.6bb2} that $\Ki$ is a subalgebra of $\Kj_>$. By construction, the monomial basis of $\Kj_>$
restricts to a monomial basis $\{\tM_A\mid A\in \txii\}$ for $\Ki$.

By a construction entirely analogous to \cite[4.3, 4.5]{BLM},  the $\A$-algebra $\Kj_>$ admits a natural bar involution $\overline{\phantom{v}}$, which
restricts to an involution on the subalgebra $\Ki$. 
Moreover, the bar involution satisfies the following property: for $A\in \txi_>$ (respectively, $A\in \txii$), we have
$$
\overline{[A]} =[A] + \sum_{A' \sqsubset A} \tau_{A',A} [A'],
$$
where $A'$ runs over $\txi_>$  (respectively, $\txii$).  Note $\overline{\tM_A} =\tM_A$ for all $A$.
A standard argument like \cite[24.2.1]{Lu93} implies the existence of a canonical basis $\{ \{A\} \mid A\in \txi_>\}$
for $\Kj_>$, which restrict to a canonical basis $\Bi := \{ \{A\} \mid A\in \txii\}$ for $\Ki$.
(The defining properties of the canonical basis can be found in Theorem~\ref{BLM4.7}.)

\begin{rem}
The $\A$-submodule $H$ of $\Kj$ spanned by $[A]$ for $A\in \txii$ is not a subalgebra of $\Kj$ (and thus 
we cannot define the algebra $\Ki$ naively to be $H$ though they have the same size). 
Indeed, one has the following example by using repeatedly Theorem~\ref{BLM3.4b}
(for $a,b \in \Z$ with $b>0$):
\begin{align*}
& \left [
\begin{matrix}
a + b -1  & 0 & 1\\
0 & 1 & 0 \\
1 & 0 & a + b -1 
\end{matrix} 
\right ]
 \cdot 
\left [
\begin{matrix}
a & 0 & b\\
0 & 1 & 0 \\
b & 0 & a
\end{matrix} 
\right ] \\
&=
v^{-a} ( v^b - v^{-b}) 
\left [
\begin{matrix}
a & 0 & b\\
0 & 1 & 0 \\
b & 0 & a
\end{matrix} 
\right ]
+ v^b \overline{[b+1]} 
\left [
\begin{matrix}
a -1 & 0 & b +1\\
0 & 1 & 0 \\
b +1 & 0 & a -1
\end{matrix} 
\right ]\\
& \qquad+ v^{b-1} \overline{[a +1]}
\left [
\begin{matrix}
a +1 & 0 & b -1\\
0 & 1 & 0 \\
b - 1 & 0 & a +1
\end{matrix} 
\right ]
+ v^{-a + b -1} ( 1 - v^{-2}) 
\left [
\begin{matrix}
a & 1 & b-1\\
1 & - 1 & 1 \\
b -1 & 1 & a
\end{matrix} 
\right ].
\end{align*}
The last term on the right hand side of the above equation is clearly not in $H$,
and hence $H$ is not a subalgebra.  
\end{rem}

\subsection{ $\Ki$ as a subquotient of  $\Kj$}
 \label{subsec:Kij}

Let $\mathcal J$ be the $\A$-submodule of $\Kj$ spanned by $[A]$ for all $A\in \txi_<$. 

\begin{lem}
\label{lem:idealJ}
The subspace $\mathcal J$  is a two-sided ideal of $\Kj$.
\end{lem}

\begin{proof}
Assume we know for now that $\mathcal{J}$ is a left ideal of $\Kj$. Since $\mathcal{J}$ 
is clearly  invariant under the anti-involution $[A] \mapsto v^{- d_A + d_{{}^tA}} [{}^t A]$,
it follows that $\mathcal{J}$ is also a right ideal of $\Kj$. 

It remains to show that $\mathcal{J}$ is a left ideal of $\Kj$. To that end it suffices to show that
$[B] \cdot [A] \in \mathcal J,$ 
for $A \in \txi_<$ and $B\in \txi$ such that
$B - RE^{\theta}_{h,h+1}$ or $B - RE^{\theta}_{h+1,h}$ is diagonal, for some $1\le h \le n$ and $R \ge 0$ 
(since such $[B]$ form a generating set of the algebra $\Kj$). 


Let us first assume  $B - RE^{\theta}_{h,h+1}$ is diagonal. Then by the multiplication formula in Theorem~\ref{BLM3.4b}(a), 
the terms arising in $[B] \cdot [A]$ are $[A+\sum_u t_u (E^{\theta}_{hu}-E^{\theta}_{h+1,u})]$, where $t_u\ge 0$. 
Clearly the $(n+1,n+1)$-entry of $[A+\sum_u t_u (E^{\theta}_{hu}-E^{\theta}_{h+1,u})]$ does not exceed $a_{n+1,n+1}$
(and recall $a_{n+1,n+1}<0$ by assumption $A\in \txi_<$). Hence
we have $[A+\sum_u t_u (E^{\theta}_{hu}-E^{\theta}_{h+1,u})] \in \mathcal J$ and $[B] \cdot [A] \in \mathcal J$.

Now assume $B - RE^{\theta}_{h+1,h}$ is diagonal. The same argument as above  using Theorem~\ref{BLM3.4b}(b)
shows that $[B] \cdot [A] \in \mathcal J$ unless $h =n$.  If $h=n$, by Theorem~\ref{BLM3.4b}(b) (where $C$ is replaced by $B$ here) we have 
\begin{equation}
 \label{BLM4.6b2}
 \begin{split}
&[B]\cdot [A]=\sum_{t}v^{\beta''(t)}
\prod_{u<n+1}\overline{ 
\begin{bmatrix}
a_{n+1,u}+t_u+t_{N+1-u}\\
 t_u
\end{bmatrix}
}
\prod_{u> n+1}\overline{
\begin{bmatrix}
a_{n+1,u}+t_u\\
t_u
\end{bmatrix}}
\\
&\qquad\qquad\qquad\qquad \cdot \prod_{i=0}^{t_{n+1}-1}
\frac{\overline{[a_{n+1,n+1}+1+2i]}}
{\overline{[i+1]}}
\left[A-\sum_{u=1}^Nt_u(E^{\theta}_{nu}-
E^{\theta}_{n+1,u})\right].
  \end{split}
\end{equation}
The $(n+1, n+1)$-entry of $\left[A-\sum_{u=1}^Nt_u(E^{\theta}_{nu}-
E^{\theta}_{n+1,u})\right]$ is $a_{n+1, n+1} + 2t_{n+1}$ (which is an odd integer). 
If $a_{n+1, n+1} + 2t_{n+1} > 0$, then the coefficient of the term $\left[A-\sum_{u=1}^Nt_u(E^{\theta}_{nu}-E^{\theta}_{n+1,u})\right]$ 
must be zero since  (here we recall $a_{n+1,n+1} < 0$)
\[
\prod_{i=0}^{t_{n+1}-1}
\frac{\overline{[a_{n+1,n+1}+1+2i]}}
{\overline{[i+1]}} = 0. 
\]
Therefore, by \eqref{BLM4.6b2}, we have $[B] \cdot [A] \in \mathcal J$.

The lemma is proved.
\end{proof}

Recall the monomial basis $\{\tM_A\mid A\in \txi\}$ of $\Kj$ defined in \eqref{BLM4.6c}.

\begin{lem}
\label{monomial-J}
For $A\in \txi_<$, 
we have 
\begin{enumerate}
\item 
$\tM_A \in \mathcal J$;

\item
$\tM_A  = [A] + \sum_{\stackrel{A'\in \txi_<}{A' \sqsubset A}} T_{A, A'} [A']$, for $T_{A, A'} \in \A$. 
\end{enumerate}
Hence $\J$ admits a monomial basis $\{\tM_A \mid A\in \txi_<\}$, and the quotient $\A$-algebra $\Kj/\J$
admits a monomial basis $\{\tM_A +\J \mid A\in \txi_>\}$.
\end{lem}
\begin{proof}
Let us write $\tM_A = [X_1] \cdot [X_2] \cdot \ldots \cdot [X_m]$, where $[X_i]$ denote the divided power factors in \eqref{BLM4.6c}. 
By construction we have 
\begin{equation}
 \label{eq:XX}
 [X_r] \cdot \ldots \cdot [X_m] = [A_r] + \sum_{A'_r \sqsubset A_r} C_{A_{r}, A_r'} [A_r']=\tM_{A_r}, 
 \end{equation}
for $A_r \in \txi$ and $1 \le r \le m$. 
The monomial $\tM_{A}$ is constructed in a way such that we fill the entries of $A$ (along the  sequence of matrices $X_m=A_{m}, A_{m-1}, \ldots, A_1=A$) 
below the diagonal from bottom to top and from right to left. Therefore, when we have filled exactly the lower-triangular $(i,j)$-th entries  with $n+1< i$ and $j < i$
 the resulting matrix $A_\ell = (a^\ell_{ij})$, for some unique $\ell$, satisfies
\begin{align}
\label{eq:Al}
a^\ell_{ij} =
\begin{cases}
0, &\text{ if } j < i \le n+1;\\
a_{ij}, &\text{ if } n+1< i \text{ and }  j < i.
\end{cases}
\end{align}
Recall we have $a_{n+1,n+1}<0$ by the assumption that $A\in \txi_<$.
By \eqref{eq:XX}, we have  $\co(X_m)_{n+1} = \co(A)_{n+1} = \co(A_r)_{n+1}$ (for all $r$, in particular for $r=\ell$), 
and thus
\begin{align*}
a^\ell_{n+1,n+1} &=\co(A_\ell)_{n+1} - 2\sum_{i>n+1} a^\ell_{i,n+1}
\\
& =\co(A)_{n+1} - 2\sum_{i>n+1} a_{i,n+1}
= a_{n+1,n+1} < 0. 
\end{align*}
Hence it follows by \eqref{eq:XX} and \eqref{eq:Al} that
\begin{align*}
\ro(X_\ell)_{n+1} 
= \ro (A_\ell)_{n+1} = a^\ell_{n+1, n+1} < 0.
\end{align*} 
Since all entries in the $(n+1)$st row of $X_\ell$  except the diagonal one
$(X_\ell)_{n+1, n+1}$ are non-negative, we conclude that $(X_\ell)_{n+1, n+1}<0$, i.e., $X_\ell \in \txi_<$. 
It follows by Lemma~\ref{lem:idealJ} that $\tM_A = [X_1] \cdot \ldots \cdot [X_m] \in \mathcal{J}$. This proves (1).
The remaining statements follow from this.
\end{proof}

The following  can be rephrased in the terminology of Lusztig \cite{Lu93} that
$\J$ is a based submodule (or ideal) of $\Kj$. 
\begin{prop}
\label{CB-J}
\begin{enumerate}
\item
The ideal 
$\mathcal J$ admits a canonical basis $\Bj \cap \mathcal J =\{ \{A\} \mid A \in \txi_<\}$. 

\item
The quotient $\A$-algebra $\Kj / \mathcal J$ admits a canonical basis $\{ \{A\} +\J \mid A \in \txi_>\}$. 
\end{enumerate}
\end{prop}

\begin{proof}
It follows by Lemma~\ref{monomial-J} that $\mathcal J$ is a bar invariant submodule of $\Kj$ and $\mathcal J$ admits a canonical basis
$C_A$ for $A\in \txi_<$, where $C_A$ satisfies $\overline{C_A} =C_A$ and $C_A \in [A] + \sum_{A'\sqsubset A} v^{-1}\Q[v^{-1}] [A'].$ 
Since $C_A \in \J \subset \Kj$ and the canonical basis element $\{A\}$ satisfies the same characterization,
it follows by uniqueness that $C_A =\{A\}$. This proves (1), and (2) follows directly from (1). 
\end{proof}

The precise relation between $\Kj_>$ and $\Kj$ is given as follows. 

\begin{prop}\label{KtildeK}
We have an isomorphism of $\mathcal{A}$-algebras 
$\sharp: \Kj / \mathcal{J}  \longrightarrow \Kj_>$, which sends
$[A] +\J \longmapsto 
[A]$, for $A\in \txi_>$. 
Moreover, the isomorphism $\sharp$ matches the corresponding monomial bases and canonical bases. 
\end{prop}

\begin{proof}
We shall denote $[A]' =[A]+\J$ in this proof. 
It is clear that $\sharp$ is a linear isomorphism. 
It remains to prove $\sharp$ is an algebra homomorphism, i.e., the structure constants with respect the standard basis are the same. 
Let us consider the products
\[
[B]' \cdot [A]' = \sum_{C}f(C)\; [C]' \in \Kj /\mathcal{J} \quad \text{ and } \quad [B] \cdot [A] = \sum_{C} g(C)\; [C] \in \Kj_>,
\] 
for $B, A \in \txi_>$. If one of $B- RE^{\theta}_{h,h+1}$ and $B- RE^{\theta}_{h+1,h}$ is diagonal, the identity $f(C) = g(C)$ 
follows by comparing the corresponding multiplication formulas 
\eqref{BLM4.6a}--\eqref{BLM4.6b'} and \eqref{BLM4.6aa}--\eqref{BLM4.6bb2}. This implies that $\sharp$ is an algebra isomorphism 
since  $[B]'$ (and respectively, $[B]$) with $B$ satisfying these conditions are  generators of the algebra $\Kj/\J$ (and respectively, $\Kj_>$).

Since $\sharp$ matches the Chevalley generators, $\sharp$ matches the monomial bases (which are given by 
the same ordered products of the generators), and $\sharp$ commutes with the bar involutions. 
Finally, $\sharp$ matches the canonical bases as partial orders $\sqsubseteq$
are compatible. 
\end{proof}

We record the following identity in $\Kj / \mathcal{J} $ ($\cong \Kj_>$), which follows from \eqref{BLM4.6aa}:
\begin{equation}
\label{eq:tef}
[D+E^{\theta}_{n,n+1} ]  \cdot  [D+E^{\theta}_{n+1,n}] = [D+ E^{\theta}_{n,n+2}] + v^{D_{n,n}-1} \overline{[D_{n,n}+1]} [D+ E^{\theta}_{n,n}]. 
\end{equation}

Combining the results on bases on $\Kj_>$ and $\Ki$ with Proposition~\ref{KtildeK}, we have established the following relation between $\Ki$ and $\Kj$.

\begin{thm}
\label{th:Kij}
The $\A$-agebra $\Ki$ is naturally isomorphic to a subquotient of the $\A$-algebra $\Kj$ (that is, a subalgebra of the quotient $\Kj/\J$), with compatible standard,
monomial, and canonical bases. 
\end{thm}
By abuse of notation, we will continue to use $[A], \tM_A$ and $\{A\}$ (for $A \in \txii$) to denote the elements of
standard, monomial and canonical bases of $\Ki$. 

\subsection{Another definition of $\Ki$}
 \label{subsec:Ki2}

Let $\Kj_1$ be the $\A$-submodule of $\Kj$ spanned by
$[A]$ where $A\in \txi$ satisfies that $\ro (A)_{n+1} = \co (A)_{n+1} = 1$. 
Then $\Kj_1$ is naturally a subalgebra of $\Kj$. One can reformulate
that $\Kj_1 =\oplus_{\la,\mu} [D] \Kj [D']$ for various diagonal matrices $D, D'$ whose $(n+1,n+1)$-entries are $1$
(recall $[D]$ are idempotents). 
Thus clearly the monomial and canonical bases of $\Kj$ restricts to
the monomial and canonical bases of $\Kj_1.$

By Lemma~\ref{lem:idealJ}, $\J_1 := \J \cap \Kj_1$ is a (two-sided) ideal of $\Kj_1$; more explicitly we have
$$
\J_1 = \A\mbox{-span} \{ [A] \mid  \ro (A)_{n+1} = \co (A)_{n+1} = 1, a_{n+1, n+1} <0 \}. 
$$ 
 
\begin{prop}
\label{CB-Ki}
\begin{enumerate}
\item
 The monomial and canonical bases of $\Kj$ (or of $\Kj_1$)
restrict to the monomial and canonical bases of $\J_1$. 

\item
The quotient $\A$-algebra $\Kj_1/\J_1$ has a monomial basis $\{ \tM_A +\J_1 \mid A \in \txii\}$
and a canonical basis $\{ \{A\} +\J_1 \mid A \in \txii\}$.
\end{enumerate}
\end{prop}

\begin{proof}
It is a direct consequence of Lemma~\ref{monomial-J} and Proposition \ref{CB-J}.
\end{proof}

The composition of the inclusion and the quotient homomorphism 
$
\Kj_1 \rightarrow \Kj_> \rightarrow \Kj_>/\mathcal{J}
$
has kernel  $\J  \cap \Kj_1=\mathcal{I}$. This induces an embedding of $\A$-algebras
$\Kj_1/\J_1  \rightarrow \Kj_>/\J$,
whose image is identified as $\Ki$ by Proposition~\ref{KtildeK}. Hence we have
obtained an isomorphism $\chi: \Kj_1/\J_1 \stackrel{\cong}{\rightarrow}  \Ki$, which preserves the standard basis as well as the monomial basis.
Since $\chi$ clearly commutes with the bar involutions and the partial orders $\sqsubseteq$ on both algebras are compatible,
$\chi$ also preserves the canonical basis. 
Summarizing, we have proved the following. 

\begin{prop}
 \label{prop:Kj1I}
We have an $\A$-algebra isomorphism $\chi: \Kj_1/\J_1 \cong \Ki$, which matches 
the corresponding standard, monomial, and canonical bases (that are all parametrized by $\txii$).
\end{prop}
Hence, $\Ki$ can be viewed naturally as a subquotient of $\Kj$ in two different and complementary ways (see Theorem~\ref{th:Kij} and Proposition~\ref{prop:Kj1I}). 
The constructions so far in this Appendix can be summarized in the following commutative diagram:
\[
\begin{CD}
\Kj_1/\J_1 @<<< \Kj_1 @>>> \Kj \\
@V\cong V\chi V @VVV   @VVV \\
\Ki @>>> \Kj_< @<\cong <\sharp< \Kj/\J
\end{CD}
\]

Denote $\QQ\Ki =\Qq \otimes_\A \Ki$. 
Similar arguments as for the Schur algebra $\Si$ (see Proposition~\ref{S_2-generator} and Corollary~\ref{cor:gen-i-Q})
give us the following.

\begin{cor}
 \label{cor:gen Ki}
\begin{enumerate}
\item[(a)]
The $\A$-algebra $\Ki$ is generated by the elements $[D+R E^{\theta}_{n,n+2}]$, 
$[D+ R E^{\theta}_{i,i+1}]$,  $[D+ RE^{\theta}_{i+1,i}]$, for $1\le i\le n-1$, $R\in \mbb N$,
and  $D \in {}^\imath \txid$.  

\item[(b)]
The $\Qq$-algebra $\QQ\Ki$ is generated by $[D]$, 
$[D+ E^{\theta}_{n,n+2}]$, 
$[D+ E^{\theta}_{i,i+1}]$,  $[D+ E^{\theta}_{i+1,i}]$, for $1\le i\le n-1$
and  $D \in {}^\imath \txid$.  
\end{enumerate}
\end{cor}

Based on the formulas \eqref{eq:T}, \eqref{eq:Tfe} and \eqref{enfn}, we introduce the following elements in $\Ki$:
\begin{equation}
 \label{eq:tD}
 \begin{split}
  \bt [D_{\lambda}] : & = [ D_\lambda - E^{\theta}_{n,n} + E^{\theta}_{n,n+2} ]  +v^{-\la_n} [D_\la] \\
  & = [D_{\lambda} - E^{\theta}_{n,n} + E^{\theta}_{n+1,n} ] \cdot  [ D_{\lambda} - E_{n,n}^{\theta} + E^{\theta}_{n,n+1}]
  - \llbracket \lambda_{n} - \lambda_{n+1} \rrbracket [D_{\lambda}]. 
\end{split}
\end{equation}

\begin{lem} \label{t-mult}
The following identity holds in the algebra $\Si$ (and respectively, $\Ki$):
for $A \in {}^{\imath}\Xi_d$ (and respectively, $\txii$), 
\[
\bt [D_{\ro(A)}]* [A] = \sum_{i} v^{\sum_{i \geq l} a_{n+2,l} - \sum_{i > l} a_{n,l} - \sum_{l > n+1} 
 \delta_{i, l} } \overline{[a_{n+2,i} +1]}  [A - E_{n,i}^{\theta} + E^{\theta}_{n+2,i} ],
\]
where the summation is taken over $1 \leq i \leq N$ such that 
$A - E_{n,i}^{\theta} + E^{\theta}_{n+2,i} 
\in {}^{\imath}\Xi_{d}$ 
(and respectively, $\txii$).
\end{lem}

\begin{proof}
We will only give the detail in the setting of $\Si$, and the other case is similar. 
It is understood that $[X]=0$ in the proof below for $X$ with any negative entry.
Recalling $a_{n+1, l} = \delta_{n+1, l}$, $\Si \subset \Sj$, and using the multiplication formulas  in $\Sj$, 
we have 
\begin{align}
\label{eq:tfe}
[ &D_{\ro(A)}  - E^{\theta}_{n, n} + E^{\theta}_{n,n+1}] \ast [D_{\ro(A)} - E^{\theta}_{n, n} + E^{\theta}_{n+1,n}]  \ast [A]\\
%
%
=&\sum_{1 \leq i \leq N, i\neq n+1} 
v^{\sum_{i \ge l}a_{n+1,l} - \sum_{i >l}a_{n,l} + \sum_{N+1-i \le l}a_{n,l} -2 \delta_{i >n+1}} \overline{[a_{n,N+1-i}+1]}[A-E_{n,i}^{\theta} +E^{\theta}_{n+2,i}]
\notag \\
&+ \Big( \sum_{1\leq i \leq N} 
v^{\sum_{i \ge l}a_{n+1,l} - \sum_{i >l}a_{n,l}+\sum_{i \le l}a_{n,l} - \delta_{i <n+1}-1} \overline{[a_{n,i}]} \Big) [A]
\notag \\
=&\sum_{1\leq i \leq N, i\neq n+1} 
v^{\sum_{l \le i}a_{n+2,l} - \sum_{l < i}a_{n,l}  - \delta_{i >n+1}} \overline{[a_{n+2,i}+1]}[A-E_{n,i}^{\theta} +E^{\theta}_{n+2,i}]
\notag \\
&+ \Big(  \sum_{1 \leq i \leq N} 
v^{ \sum_{i \le l}a_{n,l}  - \sum_{i >l}a_{n,l}+\delta_{i>n+1}- \delta_{i <n+1}-1} \overline{[a_{n,i}]} \Big) [A].
 \notag
\end{align}
Observe that 
\[
v^{-\sum_{i > l} a_{n,l} + \sum_{i \leq l} a_{n, l}} \overline{[a_{n,i}]} 
=
\frac{v^{- \sum_{i \geq l} a_{n,l} + \sum_{i < l} a_{n, l} } - v^{-\sum_{i > l} a_{n, l} + \sum_{i \leq l} a_{n, l} } } { v^{-2} -1}.
\]
This leads to the following simplifications:
\begin{align*}
\begin{split}
\sum_{i > n+1} v^{-\sum_{i > l} a_{n,l} + \sum_{i \leq l} a_{n, l}} \overline{[a_{n,i}]} 
& = \frac{- v^{-\sum_{n+2 >  l} a_{n,l} + \sum_{n+2 \leq l} a_{n, l}} + v^{- \ro(A)_n}} { v^{-2} -1}, \\
\sum_{i <  n+1} v^{-\sum_{i > l} a_{n,l} + \sum_{i \leq l} a_{n, l} -2} \overline{[a_{n,i}]} 
& = \frac{-v^{\ro(A)_n -2} + v^{-\sum_{n \geq l} a_{n,l} + \sum_{n < l} a_{n, l}-2}}{v^{-2} -1}.
\end{split}
\end{align*}
By adding the above two equations while keeping in mind again that $a_{n+1, l} = \delta_{n+1, l}$, we rewrite
the coefficient of $[A]$ on the right-hand side of \eqref{eq:tfe} as
\[
\sum_{1 \leq i \leq N} 
v^{ \sum_{i \le l}a_{n,l}  - \sum_{i >l}a_{n,l}+\delta_{i>n+1}- \delta_{i <n+1}-1} \overline{[a_{n,i}]} =
\llbracket   \ro(A)_n -1  \rrbracket +  v^{\sum_{n+1 \ge l}a_{n+2,l} - \sum_{n+1 > l}a_{n,l}}.
\] 
The lemma now follows from  \eqref{eq:tD} and the identity \eqref{eq:tfe}. 
\end{proof}

Denote by $\mbf T$ the $\Qq$-subalgebra of $\QQ \Ki$ generated by $[D_\la]$ and $\bt [D_\la]$ for  $\la \in \txii^{\text{diag}}$. 
\begin{cor}
 \label{cor:t-sub}
For $k\ge 1$ and $\la \in \txii^{\text{diag}}$,  
$\bt^k [D_\la] =\llbracket k \rrbracket! [D_{\la^k} + k E^{\theta}_{n,n+2} ] +$ 
a linear combination of $[D_{\la^r} + r E^{\theta}_{n,n+2} ]$, for some $D_{\la^r}$ and $0\le r \le k-1$. 
In particular, the $\Qq$-algebra $\mbf T$ has a  basis  
$\{ [D_\la + r E^{\theta}_{n,n+2} ] \mid \forall r \in \Z_{\ge 0}, \la \in \txii^{\text{diag}} \}$. 
\end{cor}

\begin{proof}
The first statement  follows by an easy induction on $k \ge 1$ via Lemma~\ref{t-mult}.
Note that the base case of the induction is the formula \eqref{eq:tD}. 
The second statement follows from the first one since the linear independence of the elements $[D_\la + r E^{\theta}_{n,n+2} ]$
in $\Ki$ is clear by definition. 
\end{proof}

\subsection{Isomorphism ${}_\A \Uidot \cong \Ki$}
 \label{subsec:KiU}

Below we formulate the counterparts of Sections~
\ref{sec:Ujdot}
--\ref{sec:duality j}.
The proofs are very similar and hence will be often omitted. 

The algebra $\Ui$ is defined to be the associative algebra over $\mathbb Q(v)$
generated by $\ibe{i}$, $\ibff{i}$, $\ibd{a}$, $\ibd{a}^{-1}$, $t$, for $i = 1, 2, \dots, n-1$ and $a = 1, 2, \dots, n$,
subject to the following relations:

 for $i, j = 1, 2, \dots, n-1$, $a,b = 1, 2, \dots, n$,
\begin{eqnarray}
\left\{
\begin{array}{rll}
 \ibd{a} \ibd{a}^{-1} &= \ibd{a}^{-1} \ibd{a} =1, & \\
 \ibd{a} \ibd{b} &=  \ibd{b}  \ibd{a}, & \\
 \ibd{a} \ibe{j} \ibd{a}^{-1} &= v^{-\delta_{a, j+1}-\delta_{N+1-a, j+1} +\delta_{a, j} } \ibe{j}, &  \\
 \ibd{a} \ibff{j} \ibd{a}^{-1} &= v^{-\delta_{a, j} + \delta_{a, j+1} + \delta_{N+1-a, j+1} } \ibff{j}, &  \\
 \ibe{i} \ibff{j} -\ibff{j} \ibe{i} &= \delta_{i,j} \frac{\ibd{i}\ibd{i+1}^{-1}
 -\ibd{i}^{-1}\ibd{i+1}}{v-v^{-1}},          \\
 \ibe{i}^2 \ibe{j} +\ibe{j} \ibe{i}^2 &= \llbracket 2\rrbracket  \ibe{i} \ibe{j} \ibe{i},   &\text{if }  |i-j|=1, \\
\ibff{i}^2 \ibff{j} +\ibff{j} \ibff{i}^2 &= \llbracket 2\rrbracket  \ibff{i} \ibff{j} \ibff{i}, &\text{if } |i-j|=1, \\
 \ibe{i} \ibe{j} &= \ibe{j} \ibe{i},     &\text{if } |i-j|>1,  \\
 \ibff{i} \ibff{j}  &=\ibff{j}  \ibff{i},  &\text{if } |i-j|>1, \\
  \ibff{i} t &= t \ibff{i}, &\text{if } i \neq n-1,\\
   t^2 \ibff{n-1} + \ibff{n-1} t^2 &= \llbracket 2\rrbracket  t \ibff{n-1} t + \ibff{n-1},\\
  \ibff{n-1}^2 t + t\ibff{n-1}^2
    &= \llbracket 2\rrbracket  \ibff{n-1} t \ibff{n-1},   \\
 \ibe{i} t &= t \ibe{i}, &\text{if } i \neq n-1,\\
    t^2 \ibe{n-1} + \ibe{n-1} t^2 &= \llbracket 2\rrbracket  t \ibe{n-1} t + \ibe{n-1},\\
 \ibe{n-1}^2 t + t \ibe{n-1}^2
   &= \llbracket 2\rrbracket  \ibe{n-1} t \ibe{n-1}.
   \end{array}
   \right.
\end{eqnarray}
This algebra is a coideal subalgebra of $\U(\gl(N-1))$, and hence 
it is a $\gl$-version of the coideal algebra in the same notation defined in \cite[\S2.1]{BW} (which is a coideal subalgebra of $\U (\mathfrak{sl}(N-1))$).

Similar to the construction of $\Ujdot$ in \S\ref{sec:Ujdot}, we can define the modified  
quantum algebra $\Uidot$ for  $\U^{\imath}$, where the unit of  $\U^{\imath}$ is replaced by a collection of orthogonal idempotents $D_\la$ for $\la \in {}^\imath \txid$. 
Moreover,  $\Uidot$ is naturally a $\U^{\imath}$-bimodule. 
By introducing similarly  ${}_{\lambda} \U^{\imath}_{\lambda'}$, for $\lambda, \lambda' \in {}^\imath \txid$, 
we have
\begin{align*}
\Uidot 
&= \bigoplus_{{\lambda}, {\lambda'} \in {}^\imath \txid} {_{\lambda} \U^{\imath}_{\lambda'} } 
\\
&= \sum_{{\lambda} \in {}^\imath \txid} \Ui D_{\lambda}
 = \sum_{{\lambda} \in {}^\imath \txid} D_{\lambda}\Ui.
\end{align*}

In the same way establishing the presentation for $\QQ\Ujdot$ given in Proposition~\ref{prop:presentationUjdot}, we can show that the $\Qq$-algebra
$\QQ\Uidot$ 
is isomorphic to the $\Qq$-algebra generated by the symbols $D_{\lambda}$, $\ibe{i}D_{\lambda}$,
$D_{\lambda}\ibe{i}$, $\ibff{i}D_{\lambda}$,  $D_{\lambda}\ibff{i}$, $ t D_{\lambda}$, and $D_{\lambda} t$, for $i = 1, \dots, n-1$ and ${\lambda} \in {}^\imath \txi^{\text{diag}}$,
subject to the following relations \eqref{i:align:doublestar}:

 for $i,j = 1, \dots, n-1$,
$\la, \la' \in {}^\imath \txid$,  and  for  $x, x'\in \{1, \ibe{i}, \ibe{j}, \ibff{i}, \ibff{j}, t\},$
\begin{eqnarray}
\label{i:align:doublestar}
\left\{
 \begin{array}{rll}
x D_{\lambda}  D_{\lambda'} x' &= \delta_{\la, \la'} x D_{\lambda}  x', &  \\
\ibe{i} D_{\lambda} &= D_{\lambda -\alpha_i}   \ibe{i}, & \\
\ibff{i} D_{\lambda} &=  D_{\lambda +\alpha_i}  \ibff{i},  & \\
t D_{\lambda} &= D_{\lambda} t, \\
 \ibe{i}D_{\lambda}\ibff{j} &=  \ibff{j}  D_{\lambda-\alpha_i -\alpha_j}   \ibe{i}, &\text{if } i \neq j,  \\
 \ibe{i} D_{\lambda} \ibff{i} &=\ibff{i} D_{\lambda -2\alpha_i}   \ibe{i}  + \llbracket \lambda_{i+1} - \lambda_{i} \rrbracket  D_{\lambda-\alpha_i},
 \\
(\ibe{i}^2 \ibe{j} +\ibe{j} \ibe{i}^2) D_{\lambda} &= \llbracket 2\rrbracket  \ibe{i} \ibe{j} \ibe{i} D_{\lambda}, & \text{if }  |i-j|=1, \\
( \ibff{i}^2 \ibff{j} +\ibff{j}\ibff{i}^2) D_{\lambda} &= \llbracket 2\rrbracket  \ibff{i} \ibff{j} \ibff{i}D_{\lambda} , &\text{if } |i-j|=1,\\
 \ibe{i} \ibe{j} D_{\lambda} &= \ibe{j} \ibe{i} D_{\lambda} ,   & \text{if } |i-j|>1, \\
\ibff{i} \ibff{j} D_{\lambda}  &=\ibff{j} \ibff{i} D_{\lambda} ,   & \text{if } |i-j|>1, \\
 t \ibff{i} D_{\lambda} &= \ibff{i} t D_{\lambda}, &\text{if } i \neq n-1,\\
   (t^2 \ibff{n-1} + \ibff{n-1} t^2) D_{\lambda} &= \big(\llbracket 2\rrbracket  t \ibff{n-1} t + \ibff{n-1}\big)D_{\lambda} ,\\
  (\ibff{n-1}^2 t + t\ibff{n-1}^2) D_{\lambda}
    &= \llbracket 2\rrbracket  \ibff{n-1} t \ibff{n-1} D_{\lambda},   \\
    t \ibe{i} D_{\lambda} &= \ibe{i} t D_{\lambda}, &\text{ if } i \neq n-1,\\
    (t^2 \ibe{n-1} + \ibe{n-1} t^2) D_{\lambda} &= \big(\llbracket 2\rrbracket  t \ibe{n-1} t + \ibe{n-1} \big) D_{\lambda},\\
 (\ibe{n-1}^2 t + t \ibe{n-1}^2) D_{\lambda}
   &= \llbracket 2\rrbracket  \ibe{n-1} t \ibe{n-1} D_{\lambda}.
   \end{array}
    \right.
\end{eqnarray}
To simplify the notation, we shall write $x_1D_{\lambda^1} \cdot x_2 D_{\lambda^2} \cdots x_l D_{\lambda^l} = x_1 x_2 \cdots x_l D_{\lambda^l}$,
if the product is not zero.

 \begin{thm}\label{thm:isom Ki}
We have an isomorphism of $\Qq$-algebras $\aleph^\imath  : \Uidot \rightarrow  \QQ\Ki$
defined by
\begin{align*}
%
D_{\lambda} \mapsto [D_\lambda],\qquad 
 t D_{\lambda} &\mapsto [ D_\lambda - E^{\theta}_{n,n} + E^{\theta}_{n,n+2} ]  +v^{-\la_n} [D_\la], \\
\ibe{i} D_{\lambda} \mapsto [ D_\lambda - E^{\theta}_{i,i} + E^{\theta}_{i+1, i} ] ,\qquad
\ibff{i} D_{\lambda} &\mapsto [ D_\lambda - E^{\theta}_{i+1,i+1} + E^{\theta}_{i,i+1} ],
\end{align*}
for all $1\le i\le n-1$ and  $\la \in {}^\imath \txid$.
\end{thm}

\begin{proof}
We first remark that the somewhat unusual formula for $tD_\la$ has its origin in the formulas \eqref{eq:T}-\eqref{eq:Tfe}; by the notation \eqref{eq:tD},
$\aleph^\imath (t D_\la) =\bt [D_\la]$. 

Via a direct computation one can check that $\aleph^{\imath}$ is an algebra homomorphism by using the multiplication formulas.
We illustrate here by 
checking that the map  $\aleph^{\imath}$ preserves one of the most complicated Serre-type relations, that is, 
\begin{equation} 
\label{Serre:t2f}
\aleph^{\imath} ((t^2 \ibff{n-1} + \ibff{n-1} t^2) D_{\lambda} )= \aleph^{\imath}(\big(\llbracket 2\rrbracket  t \ibff{n-1} t + \ibff{n-1} \big) D_{\lambda}).
\end{equation}
The first term on the left hand side of \eqref{Serre:t2f} is 
\begin{align*}
%
\aleph^{\imath} (& t^2 \ibff{n-1} D_{\la}) 
\\
=& v \overline{[2]} [D_{\lambda} - 3E^{\theta}_{n,n} + 2E^{\theta}_{n+2,n} + E^{\theta}_{n-1,n}] + \overline{[\lambda_n-1]} [D_{\lambda} - E^{\theta}_{n,n} + E^{\theta}_{n-1,n}] \\
&+v^{-\la_n+3}[D_{\lambda} - 2E^{\theta}_{n,n} + E^{\theta}_{n,n+2} + E^{\theta}_{n-1,n} ]\\
&+  v^{-\la_{n}+1}[D_{\lambda} - 2E^{\theta}_{n,n} + E^{\theta}_{n,n+2} + E^{\theta}_{n-1,n} ]+ v^{-2\la_{n}+2}[ D_\lambda - E^{\theta}_{n, n} + E^{\theta}_{n-1, n} ].
\end{align*}
The second term on the left hand side of \eqref{Serre:t2f} is
\begin{align*}
\aleph^{\imath} (& \ibff{n-1} t^2 D_{\la}) \\
=&
 v^{-1}\overline{[2]}[D_{\lambda} - 3E^{\theta}_{n,n} + 2E^{\theta}_{n,n+2} + E^{\theta}_{n-1,n}] 
+ v\overline{[2]}[D_{\lambda} - 2E^{\theta}_{n,n} + E^{\theta}_{n,n+2} + E^{\theta}_{n-1,n+2}] \\
&+\big(v^{-\la_n-1}+v^{-\la_n+1}\big) [D_{\lambda} - 2E^{\theta}_{n,n} + E^{\theta}_{n,n+2} + E^{\theta}_{n-1,n}] \\
&+ \big(v^{-\la_n}+v^{-\la_n+2}\big) [D_{\lambda} - E^{\theta}_{n,n} + E^{\theta}_{n-1,n+2}]
+ \big(\overline{[\la_n]}+ v^{-2\la_n}\big) [D_{\la} - E^{\theta}_{n,n} + E^{\theta}_{n-1,n}]. 
\end{align*}
After some simplification, we obtain the left hand side of \eqref{Serre:t2f} as
\begin{align*}
\aleph^{\imath} &((t^2 \ibff{n-1} + \ibff{n-1} t^2) D_{\lambda} ) \\
 =& \big(v^{-1}\overline{[2]} + v \overline{[2]}\big) [D_{\lambda} - 3E^{\theta}_{n,n} + 2E^{\theta}_{n+2,n} + E^{\theta} _{n-1,n}] \\
&+ \big( v^{-\la_{n}+1}+ v^{-\la_n+3} + v^{-\la_n-1}+v^{-\la_n+1}\big) [D_{\lambda} - 2E^{\theta}_{n,n} + E^{\theta}_{n,n+2} + E^{\theta}_{n-1,n} ]\\
&+  v\overline{[2]}[D_{\lambda} - 2E^{\theta}_{n,n} + E^{\theta}_{n,n+2} + E^{\theta}_{n-1,n+2}] \\
&+ \big(v^{-\la_n}+v^{-\la_n+2}\big) [D_{\lambda} - E^{\theta}_{n,n} + E^{\theta}_{n-1,n+2}]\\
&+ \big(\overline{[\lambda_n-1]} +v^{-2\la_{n}+2} +\overline{[\la_n]}+ v^{-2\la_n}\big) [D_{\lambda} - E^{\theta}_{n,n} + E^{\theta}_{n-1,n}].
\end{align*}

For the right hand side of \eqref{Serre:t2f}, we have
\begin{align*}
\aleph^{\imath} ( & t \ibff{n-1} t D_{\la}) 
\\
=& \overline{[2]} [D_{\la} - 3E^{\theta}_{n,n} + E^{\theta}_{n-1,n} + 2E^{\theta}_{n,n+2}] 
+ v^{-1} \overline{[\la_n-1]} [ D_\lambda - E^{\theta}_{n,n} + E^{\theta}_{n-1, n} ] 
\\
&+ v^{-\la_n+2}[D_{\la} - 2E^{\theta}_{n,n} + E^{\theta}_{n-1,n} + E^{\theta}_{n,n+2}] \\
&+ [D_{\la} - 2E^{\theta}_{n,n} + E^{\theta}_{n-1,n+2} + E^{\theta}_{n,n+2}] + v^{-\la_n+1}[D_{\la} - E^{\theta}_{n,n} + E^{\theta}_{n-1,n+2}] \\
&+  v^{-\la_n} [D_{\la} - 2E^{\theta}_{n,n} + E^{\theta}_{n-1,n} + E^{\theta}_{n,n+2}] + v^{-2\la_n+1} [ D_\lambda - E^{\theta}_{n,n} + E^{\theta}_{n-1, n} ],
\end{align*}
and hence, 
\begin{align*}
 \aleph^{\imath}( & \big(\llbracket 2\rrbracket  t \ibff{n-1} t + \ibff{n-1} \big) D_{\lambda}) 
 \\
 =&\llbracket 2 \rrbracket \overline{[2]} [D_{\la} - 3E^{\theta}_{n,n} + E^{\theta}_{n-1,n} + 2E^{\theta}_{n,n+2}] 
\\
&+ \llbracket 2 \rrbracket \big(v^{-\la_n+2}+ v^{-\la_n} \big)[D_{\la} - 2E^{\theta}_{n,n} + E^{\theta}_{n-1,n} + E^{\theta}_{n,n+2}] \\
&+ \llbracket 2 \rrbracket [D_{\la} - 2E^{\theta}_{n,n} + E^{\theta}_{n-1,n+2} + E^{\theta}_{n,n+2}] 
+ v^{-\la_n+1} \llbracket 2 \rrbracket [D_{\la} - E^{\theta}_{n,n} + E^{\theta}_{n-1,n+2}] \\
&+ \Big(\llbracket 2 \rrbracket \big(v^{-2\la_n+1}+v^{-1} \overline{[\la_n-1]}\big) +1\Big) [ D_\lambda - E^{\theta}_{n,n} + E^{\theta}_{n-1, n} ] .
\end{align*}
Now \eqref{Serre:t2f} follows by comparing the coefficients of equal terms.

The strategy of the proof that $\aleph^{\imath}$ is a linear isomorphism 
is essentially the same as for Theorem~\ref{thm:UisoK}, and let us summarize here. 
Recall $\Uidot = \dot{\mathbf U}^{\imath} (\mathfrak{gl}(N-1))$. 
The idea is to show that (by restricting to one idempotented summand)
monomial bases of $\Uidot D_{\varepsilon_{n+1}}$ and ${} _{\QQ} \Ki   [D_{\varepsilon_{n+1}}]$ are parametrized by two sets which are in natural bijection
(by passing to $\U(\mathfrak{gl}(N-1))^-D_0$ using \cite{Le02, K14} and ${} _{\QQ} \K^-  [D_0]$,
where $\QQ\K^-$ is the negative half of $\QQ\K =\dot{\U}(\mathfrak{gl}(N-1))$
constructed in \cite{BLM});
and then one shows that $\aleph^{\imath}$ sends a suitable monomial basis of $\Uidot D_{\varepsilon_{n+1}}$ to a monomial basis of ${} _{\QQ} \Ki   [D_{\varepsilon_{n+1}}]$. 

The main difference here is that the monomial basis of $\QQ\Ki$ given in Theorem~\ref{th:Kij} needs
some adjustment before matching with a monomial basis of $\Uidot$.
Note by \eqref{enfn} and Corollary~\ref{cor:t-sub} that the twin products in any monomial appearing in the monomial basis of $\QQ\Ki$
(see Theorem~\ref{th:Kij}) are of the form
$f_R(\bt)[D]$ for some polynomial $f_R$ of degree $R$. 
We claim that the monomial basis of $\QQ\Ki$ gives rise to another $\mathfrak l$-monomial basis $\mathfrak M$ of $\QQ\Ki$ when replacing each twin product $f_R(\bt)[D]$ by
its leading term $\bt^R[D]$.
This claim follows easily by an induction on $k$ for the filtration subspace $Fr^k\QQ\Ki$. (Here an increasing filtration $\{Fr^k\QQ\Ki\}_{k\ge 0}$ 
of $\QQ\Ki$ is defined by declaring the degrees of the generators of $\QQ\Ki$ in Corollary~\ref{cor:gen Ki}, 
$[D]$, 
$[D+ E^{\theta}_{n,n+2}]$ (or $\bt [D]$), 
$[D+ E^{\theta}_{i,i+1}]$,  $[D+ E^{\theta}_{i+1,i}]$,  to be $0,1,1,1$.)
With the help of the $\mathfrak l$-monomial basis $\mathfrak M$ of $\QQ\Ki$, the same argument in Theorem~\ref{thm:UisoK} goes through here.
\end{proof}

The isomorphism $\aleph^\imath: \Uidot \rightarrow  \QQ\Ki$ allows one to define an $\A$-form of $\Uidot$ as 
$${}_\A \Uidot :=(\aleph^\imath)^{-1} (\Ki)$$ 
such that $\Uidot = \Qq \otimes_{\A} {}_\A \Uidot$ (and ${}_\A \Uidot $ is a free $\A$-module since so is $\Ki$).
Moreover, the isomorphisms $\aleph^\imath$ and  $\aleph: \Ujdot \rightarrow  \QQ\Kj$
 in Theorem~\ref{thm:UisoK} allow us to transfer the relation between
$\Ki$ and $\Kj$ in Theorem~\ref{th:Kij} to a relation between $\Uidot$ and $\Ujdot$.
By abuse of notation, we still denote by $\J$ the ideal of $\Ujdot$ which corresponds to $\J \subset \Kj$ (see Lemma~\ref{lem:idealJ}). 

\begin{prop}
There is a $\Qq$-algebra embedding $\Uidot \longrightarrow \Ujdot/\J$, which sends the generators $D_\la$, $e_iD_\la, f_i D_\la$
(for $1\le i \le n-1$ and $\la \in {}^\imath  \txid$) to the generators in the same notation, and sends 
$tD_\la$ $\mapsto$ $f_ne_n D_\la -  \llbracket \la_n -\la_{n+1} \rrbracket D_\la$; here we recall $\la_{n+1} =1$. 
\end{prop}

\begin{proof}
By using  Theorem~\ref{BLM3.4b}  and keeping in mind $D_{n+1,n+1}=1$ for $D \in {}^\imath \txid$, we have that  
\[
[D+E^{\theta}_{n,n+1} ] * [D+E^{\theta}_{n+1,n}] = [D+ E^{\theta}_{n,n+2}] + v^{D_{n,n}-1} \overline{[D_{n,n}+1]} [D+ E^{\theta}_{n,n}]. 
\]
Setting $D =D_\la -E^{\theta}_{n,n}$  
leads to the equivalent formula in the proposition. 
\end{proof}

The bar involution on $\Ui$ (see \cite{BW}) induces a compatible bar involution on $\Uidot$, denoted also by $\bar{\phantom{x} }$, 
which fixes all  the generators $D_\la, t D_\la, \ibe{i} D_\la, \ibff{i} D_\la$. 
The isomorphism $\aleph^\imath $ intertwines the bar involutions on $\Uidot$ and  on $\QQ\Ki$,  i.e., $\aleph^\imath (\bar{u}) = \overline{\aleph^\imath (u)}$, for $u\in \Ui$.
%
Clearly, $_{\A}\Uidot$ is stable under the bar involution.

The isomorphism $\aleph^\imath : {}_\A \Uidot \rightarrow \Ki$ allows us to transport 
 the canonical basis  for $\Ki$ to a canonical basis for ${}_\A \Uidot$.
Introduce the divided powers  $\ibe{i}^{(r)} = \ibe{i}^r/ \llbracket r \rrbracket!$ and $\ibff{i}^{(r)} = \ibff{i}^r/ \llbracket r \rrbracket!$, for $r\ge 1$.
Then we have 
\[
\aleph^\imath (\ibe{i}^{(r)}D_{\lambda}) = [D_{\lambda} - rE^{\theta}_{i,i} + rE^{\theta}_{i+1,i}] 
\quad \text{ and } \quad  \aleph^\imath (\ibff{i}^{(r)}D_{\lambda}) = [D_{\lambda} - rE^{\theta}_{i+1,i+1} + rE^{\theta}_{i,i+1}].
\]

\subsection{Homomorphism from $\Ki$ to $\Si$}
 \label{subsec:KiS}

By construction we have the following commutative diagram
\[
\begin{CD}
\Kj_1 @>>> \Kj \\
@V\phi_d|_1 VV @VV \phi_d V \\
\Si @>>> \Sj
\end{CD}
\]
where $\phi_d|_1: \Kj_1 \rightarrow \Si$ is given by a restriction of $\phi_d$.
The surjective homomorphism $\phi_d|_1$ factors through the ideal $\J_1$, and so we obtain
a surjective homomorphism  $\phi_d^\imath: \Ki \rightarrow \Si$. 
 
The following is a counterpart of Proposition~\ref{prop:homom KSj}, which is proved in the same way, now by applying Corollary~\ref{cor:gen Ki}.
\begin{prop}
 \label{prop:homom KSi}
There exists a unique surjective $\A$-algebra homomorphism $\phi_d^\imath: \Ki \rightarrow \Si$ such that
for $R \ge 0$, $i \in [1,n-1]$ and $D \in {}^\imath \txid$,
\begin{align*}
\phi_d^\imath ( [D+ R E^{\theta}_{n,n+2}]) &= 
\begin{cases}
[D+ R E^{\theta}_{n,n+2}], &\text{ if } D+ R E^{\theta}_{n,n+2} \in {}^\imath \Xi_d;\\
0, &\text{ otherwise; }
\end{cases}\\
\phi_d^\imath ( [D+ R E^{\theta}_{i,i+1}]) &= 
\begin{cases}[D+ R E^{\theta}_{i,i+1}], &\text{ if } D+ R E^{\theta}_{i,i+1} \in {}^\imath \Xi_d;\\
0, &\text{ otherwise; }
\end{cases}\\
\phi_d^\imath ( [D+ R E^{\theta}_{i+1,i}]) &= 
\begin{cases}
[D+ R E^{\theta}_{i+1,i}], &\text{ if } D+ R E^{\theta}_{i+1,i} \in {}^\imath \Xi_d;\\
0, &\text{ otherwise. }
\end{cases}
\end{align*}
\end{prop}

There is an algebra embedding $\Ui \to \U(\gl(N-1))$ (cf. \cite[Proposition 2.2]{BW}), with convention and notation adjusted
in a way similar to the embedding $\Uj \rightarrow \U(\gl(N))$ in Proposition~\ref{prop:embedding}.
Denote by $\VVi$ the natural representation  of $\U(\gl(N-1))$. Then the tensor space $\VVi^{\otimes d}$
is naturally a $\U(\gl(N-1))$-module, which becomes a $\Ui$-module by restriction, and hence a $\Uidot$-module.
The right action of the Iwahori-Hecke algebra $\Hb$ on $\Tdi$ is similar to the one on $\Td$ in
Lemma~\ref{Hecke-action} (and is actually the same as the one given in Lemma~\ref{lem:Hecke action} below).

The action of the Iwahori-Hecke algebra $\Hb$ on  $\VVi^{\otimes d}$ is very similar to its action on $\VV^{\otimes d}$
given in \eqref{eq:Hb1}-\eqref{eq:Hb2}. 
The $(\Uidot, \Hb)$-duality established in \cite[Theorem~5.4]{BW} states that the actions of
$\Uidot$ and $\Hb$ on $\VVi^{\otimes d}$ commute and they form double centralizers. 
We note that the (algebraic) $\imath$Schur duality using the Schur algebra instead of the coideal algebra $\Uidot$ has appeared 
in \cite{G97}. 
Similar to the isomorphism $\Omega: \VV^{\otimes d} \rightarrow  \QQ\Td$ in \eqref{eq:Omega},
now we have an isomorphism $\Omega^\imath: \VVi^{\otimes d}   \rightarrow  \QQ\Tdi$ given by an analogous formula. 
As a counterpart of Theorem~\ref{thm:same duality},  
we have the following geometric realization of the $(\Uidot, \Hb)$-duality.

\begin{prop}
We have the following commutative diagram of double centralizing actions
under the identification
$\Omega^\imath: \VVi^{\otimes d}   \rightarrow  \QQ\Tdi$:
\begin{eqnarray*}
\begin{array}{ccccc}
  \QQ\Ki &\circlearrowright   
 & \QQ\Tdi & \circlearrowleft  \; & \QQ\Hb 
  \\
\aleph^\imath  \uparrow\;\; & & \Omega^\imath \uparrow & & \parallel
 \\
\Uidot    & \circlearrowright  
& \VVi^{\otimes d}  &  \circlearrowleft  \; & \QQ\Hb 
\end{array}
\end{eqnarray*}
\end{prop}

\subsection{Compatibility of canonical bases}
 \label{subsec:KiCB}

In this subsection we first work in the setting of Sections~\ref{sec:Schur}-\ref{sec:qalg} for $\Sj$ and $\Kj$.  
 We shall use the notations with subscripts, $[A]_d$ and $\{A\}_d$, to denote a standard and canonical basis element in $\Sj$, and as before use $[A]$ 
and $\{A\}$ to denote a standard and a canonical basis element in $\Kj$, respectively. 
Recall from Proposition~\ref{prop:homom KSj}
the surjective $\A$-algebra homomorphism $\phi_d: \Kj \rightarrow \Sj$. 

\begin{lem}\label{lem:barscommute}
The bar involutions on $\Kj$ (as well as on $\Ujdot$) and $\Sj$ commute with the homomorphism $\phi_d : \Kj \rightarrow \Sj$.
\end{lem}

\begin{proof}
The lemma follows by checking on the generators with the help of Proposition~\ref{prop:homom KSj}. 
\end{proof}

A type $A$ version of the following lemma appears in \cite[Lemma 6.4]{DF14}, and the proof below follows the one in \cite[Proposition~ 6.3]{Fu12}. 
We thank Jie Du and Qiang Fu for clarifying our misunderstanding of their crucial lemma.

\begin{lem} 
\label{lem:AA}
Let $A\in \tilde \Xi$. Then $\phi_d: \Kj \rightarrow \Sj$ sends
\[
\phi_d([A]) =
\begin{cases}
[A]_d, &\text{ if } A \in \Xi_d;\\
0, &\text{ otherwise}.
\end{cases}
\]
\end{lem}

\begin{proof}
We define an $\A$-linear map $\phi_d': \Kj \rightarrow \Sj$ by sending $[A]$ to $[A]_d$ for $A\in \Xi_d$ and to $0$ otherwise. 
We shall show that $\phi_d' =\phi_d$.
Observe that $\phi_d'$ coincides with $\phi_d$ given in Proposition ~\ref{prop:homom KSj} on the generators of $\Kj$.
So it suffices to show that $\phi_d'$ is an algebra homomorphism.
To that end, by the description of generators for $\Kj$ in Proposition~\ref{prop:powergen}, it suffices to check that 
\begin{equation}
\label{eq:BA}
\phi_d' ([B] \cdot [A]) = \phi_d' ([B]) * \phi_d' ([A]),
\end{equation}
for $B =(b_{ij}) = 
[D+ R E^{\theta}_{h,h+1}]$ or $[D+ R E^{\theta}_{h+1,h}]$ (for $R \ge 0$, $h \in [1,n]$ and $D \in \txid$) and for all $A =(a_{ij}) \in \tilde \Xi$.
We can further assume without loss of generality that $\co(B) =\ro(A)$ and $\sum_{i,j} a_{ij} =2d+1 =\sum_{i,j} b_{ij}$. 
We will treat the case for $B = [D+ R E^{\theta}_{h,h+1}]$ in detail. 
The verification of \eqref{eq:BA}  is divided into 3 cases.

(1) Assume both $B, A \in \Xi_d$. The identity \eqref{eq:BA} follows directly from the multiplication formulas
in Theorem~\ref{BLM3.4b} and \eqref{BLM4.6a}--\eqref{BLM4.6b'}. 

(2) Assume $A \not \in \Xi_d$. Then there exists $i_0$ such that $a_{i_0,i_0} <0$. It follows by definition that $\phi_d' ([B]) * \phi_d' ([A])=0$.
On the other hand, the matrices $A+\sum_u t_u (E^{\theta}_{hu}-E^{\theta}_{h+1,u})$
in the product $[B] \cdot [A]$ (see \eqref{BLM4.6a}) have negative $(i_0, i_0)$-entry 
with possible exceptions when $i_0=h$ and $t_{hh} > t_{hh} +a_{hh} \geq 0$. In such exceptional cases, the coefficient of
such a matrix is a product with one factor    
$\begin{bmatrix}
a_{hh} + t_h\\
 t_h
\end{bmatrix}
$,  which is 0. Therefore $\phi_d' ([B] \cdot [A]) =0$ by definition of $\phi_d'$, and \eqref{eq:BA} holds.

(3) Assume $B \not \in \Xi_d$. By definition we have $\phi_d' ([B]) * \phi_d' ([A])=0$.
By assumption, there exists $i_0 \le n+1$ such that $b_{i_0,i_0}<0$. 
We separate the proof that $\phi_d' ([B] \cdot [A]) =0$ into 2 subcases (i)--(ii) below. 
{(i)} Assume $i_0 \neq h$. 
Then  the $i_0$-component of $\ro(B)$ is negative and all the resulting new matrices in the product
$[B] \cdot [A]$ (see \eqref{BLM4.6a}) have $\ro(B)$ as their row vectors and so they all contain some negative entry. Thus $\phi_d' ([B] \cdot [A]) =0$. 
{(ii)} Assume $i_0 = h$. Then $\co(B)$ has a negative $h$th component. Since $\ro(A) =\co(B)$, 
we have $A \not \in \Xi_d$. Hence we are back to Case (2) above, and 
so $\phi_d' ([B] \cdot [A]) =0$. 

Summarizing (1)-(3), we have verified \eqref{eq:BA} for $B = [D+ R E^{\theta}_{h,h+1}]$.
The completely analogous remaining case for $B = [D+ R E^{\theta}_{h+1,h}]$ will be skipped. The lemma is proved. 
\end{proof}

The next theorem shows that the canonical bases of $\Kj$ and $\Sj$ are compatible under the homomorphism $\phi_d$. 
(A similar result holds in the type $A$ setting; cf. \cite{SV00} and \cite{DF14}). 

\begin{thm} 
 \label{th:CBtoCB}
The homomorphism $\phi_d: \Kj \rightarrow \Sj$ sends  canonical basis elements for $\Kj$ 
to canonical basis elements for $\Sj$ or zero. More precisely, for $A\in \tilde \Xi$, we have
\[
\phi_d(\{A\}) =
\begin{cases}
\{A\}_d, &\text{ if } A \in \Xi_d;\\
0, &\text{ otherwise}.
\end{cases}
\]
\end{thm}

\begin{proof}
By Lemma~\ref{lem:barscommute}, we have $\overline{\phi_d(\{A\})} 
=\phi_d( \{A\})$, that is, $\phi_d(\{A\})$ is bar invariant. 
By Theorem~\ref{BLM4.7}, we have 
\[
\{A\} =[A] +\sum_{A'\sqsubset A, A'\in \txi} \pi_{A',A} [A'], \quad  \pi_{A', A} \in v^{-1} \mbb Z[v^{-1}].
\]

First assume that $A \in \Xi_d$. Then $\phi_d([A]) =[A]_d$ by Lemma~\ref{lem:AA}, and we have 
\begin{equation}
\label{eq:CBorder2}
\phi_d(\{A\}) = [A]_d +\sum_{A'\sqsubset A, A' \in \Xi_d}  \pi_{A',A} [A']_d.
\end{equation}
By Lemma~\ref{BLM3.6}, the (geometric) partial order $\leq$ is stronger than $\sqsubseteq$, and note that
the canonical basis element $\{A\}_d$ is characterized by the bar-invariance and the property
that $\{A\}_d \in [A]_d +\sum\limits_{A'\sqsubset A, A' \in \Xi_d}  v^{-1}\Z[v^{-1}] [A']_d$; compare \eqref{{A}}--\eqref{eq:v-}. 
By a comparison with \eqref{eq:CBorder2} above,
we must have $\phi_d(\{A\}) =\{A\}_d$.
 
Now assume that $A \not\in \Xi_d$. Then $\phi_d([A]) = 0$ by Lemma~\ref{lem:AA}, and  we have 
\[
\phi_d(\{A\}) = \sum_{A'\sqsubset A, A' \in \Xi_d} \pi_{A',A}  [A']_d,
\]
which lies in $\sum_{A'' \in \Xi_d}v^{-1}\mathbb{Z}[v^{-1}] \{A''\}$ by applying an inverse version of \eqref{{A}}.
Since $\phi_d(\{A\})$ is bar invariant, it must be $0$. The theorem is proved.
\end{proof}

By definition and Lemma~\ref{lem:AA}, $\phi_d : \Kj \rightarrow \Sj$ factors through $\J$ and hence we obtain a homomorphism
$\bar{\phi}_d  : \Kj/\J \rightarrow \Sj$. Then putting all constructions together we have the following commutative diagram
\begin{align}
\label{CD:ij}
\begin{CD}
\Ki @>>> \Kj/\J \\
@V\phi_d^\imath VV @VV \bar{\phi}_d  V \\
\Si @>>> \Sj
\end{CD}
\end{align}
 
It follows by the above results in \S \ref{subsec:Ki}
--\S\ref{subsec:KiCB} 
that  the standard bases are compatible under the two horizontal homomorphisms and $\bar{\phi}_d$ in \eqref{CD:ij},
and hence they are also compatible under  $\phi_d^\imath$. 
In the same vein, we  apply Theorem~\ref{th:CBtoCB} to conclude the following analogue of Theorem~\ref{th:CBtoCB},
which can be restated that 
$\phi_d^\imath: \Uidot \rightarrow \Si$ is a homomorphism of based modules in the sense of \cite[Chapter~27]{Lu93}.

\begin{prop}
The homomorphism $\phi_d^\imath: \Uidot \rightarrow \Si$ sends canonical base elements to canonical base elements or zero.
\end{prop}


\end{document}